\documentclass[a4paper,10pt]{article}
\usepackage[utf8]{inputenc}
\usepackage[T1]{fontenc}
\usepackage[english]{babel}
\usepackage{amsmath}
\usepackage{amsthm}
\usepackage{amsfonts}
\usepackage{amssymb}
\usepackage{listings}
\usepackage{graphicx}
\usepackage{mathrsfs}
\usepackage{hyperref}
\usepackage{algpseudocode}
\usepackage{algorithmicx}
\usepackage{algorithm}
\usepackage{subfig}
\usepackage{stmaryrd}
\usepackage{bm}
\usepackage{tikz}
\usepackage{braket}
\usepackage{wasysym}
\usepackage{wrapfig}
\usepackage[framemethod=TikZ]{mdframed}
\usepackage{xcolor}
\usepackage{faktor}
\usepackage{enumitem}
\usepackage{xifthen}
\usetikzlibrary{calc}
\usepackage{tkz-fct}  

\usepackage[top=2.00cm,bottom=3.00cm,left=2.00cm,right=2.00cm]{geometry}

\DeclareRobustCommand{\rchi}{{\mathpalette\irchi\relax}}
\newcommand{\irchi}[2]{\raisebox{\depth}{$#1\chi$}} 

\newcommand{\mc}{\mathscr}
\newcommand{\f}{\mathbb}
\newcommand{\vf}{\varphi}
\newcommand{\ol}{\overline}
\newcommand{\cu}{\subseteq}
\newcommand{\wt}{\widetilde}

\newcommand{\serie}[1]{\{#1_{n}\}_n}
\newcommand{\ms}[1]{\{#1_{\nn}\}_n}
\newcommand{\mso}[1]{\{#1_{\nn}^\om\}_n}
\newcommand{\GLT}{\sim_{GLT}}
\newcommand{\dacs}[2]{d_{\rm a.c.s.}\left(#1,#2\right)}
\newcommand{\acs}{\xrightarrow{a.c.s.}}

\newcommand{\ve}{\varepsilon}

\newcommand{\om}{\Omega}
\newcommand{\nn}{{\bm n}}
\newcommand{\ii}{{\bm i}}
\newcommand{\jj}{{\bm j}}
\newcommand{\hh}{{\bm h}}
\newcommand{\kk}{{\bm k}}
\newcommand{\uu}{{\bm 1}}
\newcommand{\GLTom}{\sim_{GLT}^\om}
\newcommand{\size}[1]{#1\times #1}
\newcommand{\dno}{d_n^\om}
\newcommand{\nnn}{N(\nn)}
\newcommand{\sdno}{\size{\dno}}
\newcommand{\snnn}{\size{\nnn}}
\newcommand{\es}{\sim_\lambda}
\newcommand{\svs}{\sim_\sigma}
\newcommand{\btheta}{\boldsymbol\theta}

\DeclareMathOperator{\diag}{diag}
\DeclareMathOperator{\ess}{ess}
\DeclareMathOperator{\rk}{rk}

                \newcommand{\bfx}{\mathbf x}

\pgfmathdeclarefunction{phi_x}{2}{%
	\pgfmathparse{6*(#1+1)/sqrt(40+(#1+1)*(#1+1)+#2*#2)-2}%
}

\pgfmathdeclarefunction{phi_y}{2}{%
	\pgfmathparse{6*#2/sqrt(40+(#1+1)*(#1+1)+#2*#2)}%
}

\pgfmathdeclarefunction{seg_x}{5}{%
	\pgfmathparse{phi_x( #1 + #5*(#3-#1)  , #2 + #5*(#4-#2)  )}%
}

\pgfmathdeclarefunction{seg_y}{5}{%
	\pgfmathparse{phi_y( #1 + #5*(#3-#1)  , #2 + #5*(#4-#2)  )}%
}

\theoremstyle{definition}

\theoremstyle{plain}

\newtheorem{theorem}{Theorem}[section]

\newtheorem{lemma}{Lemma}[section]
\newtheorem{corollary}{Corollary}[section]

\newtheorem{remark}{Remark}[section]

\title{A systematic approach to reduced GLT}
\author{Giovanni Barbarino}

\begin{document}

\maketitle

\begin{abstract}
	This paper concerns the spectral analysis of matrix-sequences that are generated by the discretization and numerical approximation of partial differential equations (PDEs), in case  the domain is a generic Peano-Jordan measurable set. It is observed that such matrix-sequences often present a spectral symbol, that is a measurable function describing the asymptotic behaviour of the eigenvalues. When the domain is a hypercube, the analysis can be conducted using the theory of generalized locally Toeplitz (GLT) sequences, but in case of generic domain, a new kind of matrix-sequences and theory has to be formalized. We thus introduce the Reduced GLT sequences and symbols, developing in full detail its theory, and presenting some application to finite differences and finite elements discretization for convection-diffusion-reaction differential equations. 
\end{abstract}
%

Partial differential equations (PDEs) are extensively used in physics, engineering and applied sciences in order to model real-world problems. A closed form for the analytical solution of such PDEs is normally not available. It is therefore of fundamental importance to approximate the solution $u$ of a PDE by means of some numerical method.

Despite the differences that allow one to distinguish among the various numerical methods, the principle on which most of them are based is essentially the same: they first discretize the continuous DE by introducing a mesh, characterized by some discretization parameter $n$, and then they compute the corresponding numerical solution 
$u_n$, which will (hopefully) converge in some topology to the solution $u$ of the DE as $n\to\infty$, i.e., as the mesh is progressively refined.
If we consider a linear PDE
\[
\mc Lu = f,
\]
and a linear numerical method, then the actual computation of the numerical solution reduces to solving a linear system
\[
A_nu_n =f_n
\]
whose size $d_n$ diverges with $n$. 

Solving high-dimensional linear   systems in an efficient way   is fundamental to compute accurate solutions in a reasonable time. In   this   direction, it is known that the convergence properties of mainstream iterative solvers, such as multigrid and preconditioned Krylov methods, strongly depend on the spectral features of the matrices to which they are applied.  
The knowledge of the asymptotic distribution of the sequence $\serie A$ is   therefore a useful tool we can use to choose or to design the best solver and method of discretization.

The discretization methods for linear PDEs often lead to sequences of linear systems admitting a so-called  \textit{Spectral Symbol}. It is an entity associated with a matrix-sequence  of increasing size, and it represents the asymptotic distribution of the spectra of the matrices.  
More specifically, given a matrix-sequence $ \serie A $, where $ A_n\in \mathbb C^{d_n\times d_n} $, with $d_n\xrightarrow{n\to\infty}\infty$,  
then we say that $ \serie A $ possesses a spectral symbol $f:D\cu \mathbb R^q\to \mathbb C$, $q\ge 1$, when it satisfies 
\[
\lim_{n\to\infty} \frac{1}{d_n} \sum_{i=1}^{d_n} F(\lambda_i(A_n)) = \frac{1}{l(D)}\int_D F(f(x)) dx
\]
for every continuous function $F:\mathbb C\to \mathbb C$ with compact support. Here $D$ is a measurable set with finite Lebesgue measure $l(D)>0$ and $\lambda_i(A_n)$ are the eigenvalues of $A_n$. In this case we write 
\[ \serie A\sim_\lambda f. \]
In case of Toeplitz or Toeplitz-like matrices, this topic has been the subject of several studies and research, starting from Szegő \cite{szego}, Avram \cite{avram}, Parter \cite{parter},   and   followed by various other authors \cite{tau,BG0,BG,BS,BS2,Tilli1,Tyrty1,Tyrty2,wid,Tyrty4}.
Asymptotic distributions also naturally arise in the theory of orthogonal polynomials, in which case the
zeros of the polynomials are seen as   the   eigenvalues of   appropriate Jacobi matrices \cite{orth1,orth2,orth3,orth4,Tyrty3}.

A powerful set of tools to compute and analyse the symbols comes from 
the theory of generalized locally Toeplitz (GLT) sequences.
It stems from Tilli's work on locally Toeplitz (LT) sequences \cite{Tilli} and from the spectral theory of Toeplitz matrices.
Nowadays, the main and most comprehensive sources in the literature for theory and applications of GLT spaces are the books \cite{GLT-book,GLT-bookII,GLT-bookIII,GLT-bookIV}, in which we can find a careful and complete description of GLT sequences, block GLT sequences, and their respective multivariate versions. In short, the GLT theory enables us to derive  the symbol for large families of matrix-sequences, from simple components. Since the relation linking  the sequences to the so-built symbols turns out to be an isomorphism of spaces, we can denominate the chosen symbol for $\serie A$ as its GLT symbol, and we write
 we write 
\[ \serie A\GLT f. \]

When dealing with Linear PDE such as
\[
\begin{cases}
\mc L(u)(x) = f(x) & x\in [0,1]^d\\
B.C. & x\in \partial([0,1]^d) 
\end{cases}
\]
the discretization methods often lead to sequences of linear systems admitting a GLT symbol with domain $([0,1]\times [-\pi,\pi])^d$, that we can exploit to gain spectral information on our systems and predict the speed rates of several iterative solving algorithms.  Interestingly enough, it has been observed that when the domain $\om$ of $u$ is regular enough, a similar analysis can be conducted.

We already know that there are  well-known  cases of linear PDE on non-rectangular domains. For example, in the context of Finite Elements methods with constant coefficients, the domains of the basis functions can be arbitrary since the main focus is on the values of the bilinear form evaluated on couples of basis functions, so the resulting symbols have domain $[-\pi,\pi]^d$. The case of FE or collocation methods with variable coefficients has been studied on the condition that the physical domain $\om$ can be described by a global geometry function $G : [0, 1] \to \om$, which is invertible and satisfies $G(\partial ([0, 1]^d)) = \partial\om$.

Now we want to explore a more general case, starting from a domain $\om$ with few properties. We consider a bounded $\om$, so that we can operate an affine invertible transformation $l:\f R^d\to \f R^d$ , consisting in the composition of a  translation and a dilation, such that $l^{-1}(\om)=\om'\cu [0,1]^d$. Notice that if $v = u\circ l:\om'\to\f C$, then
\[
\begin{cases}
\mc L(u)(x) = f(x) & x\in \om\\
B.C. & x\in \partial\om
\end{cases}
\to 
\begin{cases}
\mc L(u)(l(y)) = f(l(y)) & y\in \om'\\
\wt{B.C.} & y\in \partial\om' 
\end{cases}
\to 
\begin{cases}
\wt{\mc L}(v)(y) = g(y) & y\in \om'\\
\wt{B.C.} & y\in \partial\om' 
\end{cases}
\]
so we can solve the problem on $\om'$ for $v$, and then compute $u = v\circ l^{-1}$. From now on we will only consider  domains $\om$ contained in $[0,1]^d$, and we work in the restricted euclidean topology and Lebesgue measure $\mu$ of $[0,1]^d$, unless specified differently.  

The analysis will lead us to introduce a variation of the classical GLT sequences, that we call \textit{Reduced GLT} sequences. 
The existence of such a class of sequence has previously been hinted by some authors (\cite{Beck,tyr,glt_1,stesso,tablino}), but has never been fully explored and formalized. In a recent paper \cite{Bianchi}, the authors computed the symbols for  classes of graph sequences having a grid geometry with a uniform local structure in a domain $\Omega\cu [0,1]^d$, actually following a similar reasoning of the ones in this document, but they never fully formalized the class of sequences in use.

The paper is organized as follows. 
In Section \ref{Review} we recall first the multidimensional notation we will be using throughout all the document, and then we report the main concepts and results already present in previous literature, that we will need to develop our new theory. In particular, we remind the concepts of symbol, approximating classes of sequences and multilevel GLT sequences.
Section \ref{characteristic_sequences} is devoted to discussions on the domain $\Omega$ and the grids we use to discretize our problems.
In Section \ref{Res_Exp} we introduce the restriction and expansion operators that we need to generate the Reduced GLT sequences from classical GLT sequences.
 Thanks to the properties of these operators, we will be able to derive a number of preliminary results, that will lead in Section \ref{RedGLT} to the full formalization of the theory of Reduced GLT sequences. 
 The following two Section \ref{SW} and \ref{P1} show how the theory of Reduced GLT can be applied to discretisation of linear PDEs on a generic domain $\Omega$, in case of, respectively, a finite differences discretization, and a finite elements discretization. 
 In the final Section \ref{FW}, we report  further studies that are currently been conducted on other applications for the Reduced GLT sequences, and also a possible generalization. 

\section{Generalized Locally Toeplitz Sequences}\label{Review}

Here we recall the basic notions, results and concepts of multilevel GLT sequences and linked subjects, without going too much into technical details. All the results we report in this section can be found in \cite{GLT-bookII}, altogether with an extensive and complete discussion about the GLT sequences. 

\subsection{Multidimensional Notation}

When dealing with multilevel sequences, matrices and vectors, we will use the multi-index notation.
A multi-index $\ii\in\mathbb N^d$, also called a $d$-index, is simply a vector in $\mathbb N^d$; its components are denoted by $i_1,\ldots,i_d$.

\begin{itemize}
	
	\item $\bm 0,\,\uu,\,\mathbf2,\,\ldots$ are the vectors of all zeros, all ones, all twos, $\ldots$ (their size will be clear from the context).
	
	\item For any $d$-index $\nn$, $N(\nn)=\prod_{j=1}^dn_j$ and $\nn\to\infty$ means that $\min(\nn)=\min_{j=1,\ldots,d}n_j\to\infty$.
	
	\item If $\bm h,\bm k$ are $d$-indices, $\bm h\le\bm k$ means that $h_r\le k_r$ for all $r=1,\ldots,d$, while $\bm h\not\le\bm k$ means that $h_r>k_r$ for at least one $r\in\{1,\ldots,d\}$.
	
	\item If $\bm h,\bm k\in\mathbb N^d$ are multi-indices, $\bm h\preceq\bm k$ means that $\bm h$ precedes (or equals) $\bm k$ in the lexicographic ordering (which is a total ordering on $\mathbb N^d$). 
	\item The multi-index range $\bm h,\ldots,\bm k$ is the set $\{\jj\in\mathbb Z^d:\,\bm h\le\jj\le \bm k\}$. 
	We assume for the multi-index range $\bm h,\ldots,\bm k$ the standard lexicographic ordering:
	\begin{equation}\label{ordering}
	\left[\ \ldots\ \left[\ \left[\ (j_1,\ldots,j_d)\ \right]_{j_d=h_d,\ldots,k_d}\ \right]_{j_{d-1}=h_{d-1},\ldots,k_{d-1}}\ \ldots\ \right]_{j_1=h_1,\ldots,k_1}.
	\end{equation}
	For instance, in the case $d=2$ the ordering is
	\begin{align*}
	&(h_1,h_2),\,(h_1,h_2+1),\,\ldots,\,(h_1,k_2),\,(h_1+1,h_2),\,(h_1+1,h_2+1),\,\ldots,\,(h_1+1,k_2),\\
	&\ldots\,\ldots,\,(k_1,h_2),\,(k_1,h_2+1),\,\ldots,\,(k_1,k_2).
	\end{align*}
	
	\item When a $d$-index $\jj$ varies over a multi-index range $\bm h,\ldots,\bm k$ (this is sometimes written as $\jj=\bm h,\ldots,\bm k$), it is understood that $\jj$ varies from $\bm h$ to $\bm k$ following the specific ordering \eqref{ordering}. For instance, if $\nn\in\mathbb N^d$ and if we write $\bm x=[x_\ii]_{\ii=\uu}^\nn$, then $\bm x$ is a vector of size $N(\nn)$ whose components $x_\ii,\ \ii=\uu,\ldots,\nn$, are ordered in accordance with \eqref{ordering}: the first component is $x_\uu=x_{(1,\ldots,1,1)}$, the second component is $x_{(1,\ldots,1,2)}$, and so on until the last component, which is $x_\nn=x_{(m_1,\ldots,m_d)}$. Similarly, if  $X=[x_{\ii\jj}]_{\ii,\jj=\uu}^{\nn}$, then $X$ is a $N(\nn)\times N(\nn)$ matrix whose components are indexed by two $d$-indices $\ii,\jj$, both varying from $\uu$ to $\nn$ according to the lexicographic ordering \eqref{ordering}.
	\item Given $\bm h,\bm k\in\mathbb N^d$ with $\bm h\le\bm k$, the notation $\sum_{\jj=\bm h}^{\bm k}$ indicates the summation over all $\jj$ in $\bm h,\ldots,\bm k$.
	\item If $\bm h$ is a $d$-index in the range $\uu,\dots,\nn$, 
	then 
	\begin{align*}
	|\bm h| :&= h_d + n_d(h_{d-1} -1 + n_{d-1}( h_{d-2}-1
	+ \dots +n_3(h_2-1 + n_2(h_1-1)))\dots)
	\\ 
	& = h_d +  \sum_{i=1}^{d-1} \left( (h_i-1)\prod_{j=i+1}^{d} n_j   \right)
	\end{align*}
	maps the $d$-indices to the set $\{1,2,\dots,N(\nn)\}$, and the map is increasing with respect to the lexicographic ordering, since $\bm h\succeq \bm k\iff |\bm h|\ge |\bm k|$. 
	\item Operations involving $d$-indices that have no meaning in the vector space $\mathbb Z^d$ must always be interpreted in the componentwise sense. For instance,  $\ii\jj=(i_1j_1,\ldots,i_dj_d)$, $\ii/\jj=(i_1/j_1,\ldots,i_d/j_d)$, etc.
\end{itemize}

In this context, by a sequence of matrices (or matrix-sequence) we mean a sequence of the form $\{A_\nn\}_n$, where  $\nn=(n_1,\dots,n_d)$ depends on $n$ and $\nn\to\infty$ as $n\to\infty$. In many cases, it is natural to assume that $\nn+\uu = n{\bm \nu}$, where $\bm \nu$ is a vector of rational constants and $n$ diverges to infinity. It is always understood that a matrix $A_\nn$ parameterized by a $d$-index $\nn$ has dimension $N(\nn) = n_1\cdot\dots\cdot n_d$; its entries will be indexed by two $d$-indices $\ii,\jj$.

\subsection{Singular Values Symbol and Approximating Classes of Sequences}

Along with the concept of spectral symbol already introduced, we need to recall the notion of \textit{Singular Values Symbol}, that is, a measurable function describing the asymptotic distribution of the singular values of a matrix-sequence. Given a multilevel sequence $\ms A$, a singular value symbol associated with $\ms A$ is a measurable function $f:D\cu \mathbb R^q\to \mathbb C$ satisfying 
\[
\lim_{n\to\infty} \frac{1}{N(\nn)} \sum_{i=1}^{N(\nn)} F(\sigma_i(A_\nn)) = \frac{1}{l(D)}\int_D F(|f(x)|)dx
\]
for every continuous function $F:\mathbb R\to \mathbb C$ with compact support, where $D$ is a measurable set with finite Lebesgue measure $l(D)>0$ and $\sigma_i(A_\nn)$ are the singular values of $A_n$. In this case we write 
\[ \ms A\sim_\sigma f. \]

A sequence of matrices $\{Z_\nn\}_n$ such that $\{Z_\nn\}_n\sim_\sigma0$ is referred to as a zero-distributed sequence. In other words, $\{Z_\nn\}_n$ is zero-distributed iff
\[ \lim_{n\to\infty}\frac1{N(\nn)}\sum_{i=1}^{N(\nn)}F(\sigma_i(Z_\nn))=F(0),\qquad\forall\,F\in C_c(\mathbb R), \]
where $N(\nn)$ is the size of $Z_\nn$.
Given a sequence of matrices $\{Z_\nn\}_n$, with $Z_\nn$ of size $N(\nn)$, the following properties hold. In what follows, we use the natural convention $C/\infty=0$ for all numbers $C$.
\begin{enumerate}[leftmargin=23pt]
	\item[\textbf{Z\,1.}] $\{Z_\nn\}_n\sim_\sigma0$ iff $Z_\nn=R_\nn+N_\nn$ with $\displaystyle\lim_{n\to\infty}(N(\nn))^{-1}{\rm rank}(R_\nn)=\lim_{n\to\infty}\|N_\nn\|=0$.
	\item[\textbf{Z\,2.}] $\{Z_\nn\}_n\sim_\sigma0$ if there exists a $p\in[1,\infty]$ such that $\displaystyle\lim_{n\to\infty}(N(\nn))^{-1/p}\|Z_\nn\|_p=0$.
\end{enumerate}

The space of matrix-sequences also presents  a metric structure, induced by a distance inspired from the concept of \textit{Approximating Class
of Sequences} (a.c.s.). In fact, Given a sequence of matrix-sequences $\{B_{\nn,m}\}_n$, it is said to be an a.c.s. for $\ms A$ if there exist $\{N_{\nn,m}\}_n$ and $\{R_{\nn,m}\}_n$ such that for every $m$ there exists $n_m$ with
\[
A_\nn = B_{\nn,m}  + N_{\nn,m} + R_{\nn,m}, \qquad \|N_{\nn,m}\|\le \omega(m), \qquad \rk(R_{\nn,m})\le N(\nn)c(m)
\]
for every $n>n_m$, and
\[
\omega(m)\xrightarrow{m\to \infty} 0,\qquad c(m)\xrightarrow{m\to \infty} 0.
\]
In this case, we say that $\{B_{\nn,m}\}_n$ is a.c.s. convergent to $\ms A$, and we use the notation $\{B_{\nn,m}\}_n\acs \ms A$.
In other words, $\{B_{\nn,m}\}_n$ converges to $\ms A$ if the difference $\{A_\nn -B_{\nn,m}\}_n$ can be decomposed into $\{N_{\nn,m}\}_n$ of 'small norm' and $\{R_{\nn,m}\}_n$ of 'small rank'. 

We say that  a sequence $\ms A$ is sparsely unbounded (s.u.), whenever the rate of diverging singular values goes to zero. This happens, for example, whenever the sequence admits a singular value symbol. Using this notion, we can enunciate the property of the a.c.s. we will need in the document. 
\begin{enumerate}[leftmargin=39pt]
	\item[\textbf{ACS\,1.}] $\{A_\nn\}_n\sim_\sigma f$ iff there exist sequences of matrices $\{B_{\nn,m}\}_n\sim_\sigma f_m$ such that $\{B_{\nn,m}\}_n\stackrel{\rm a.c.s.}{\longrightarrow}\{A_\nn\}_n$ and $f_m\to f$ in measure.
	\item[\textbf{ACS\,2.}] Suppose each $A_\nn$ is Hermitian. Then, $\{A_\nn\}_n\sim_\lambda f$ iff there exist sequences of Hermitian matrices $\{B_{\nn,m}\}_n\sim_\lambda f_m$ such that $\{B_{\nn,m}\}_n\stackrel{\rm a.c.s.}{\longrightarrow}\{A_\nn\}_n$ and $f_m\to f$ in measure.
	\item[\textbf{ACS\,3.}] If $\{B_{\nn,m}\}_n\stackrel{\rm a.c.s.}{\longrightarrow}\{A_\nn\}_n$ and $\{B_{\nn,m}'\}_n\stackrel{\rm a.c.s.}{\longrightarrow}\{A_\nn'\}_n$, with $A_\nn$ and $A_\nn'$ of the same size $N(\nn)$, then
	\begin{itemize}[leftmargin=*]
		\item $\{B_{\nn,m}^*\}_n\stackrel{\rm a.c.s.}{\longrightarrow}\{A_\nn^*\}_n$,
		\item $\{\alpha B_{\nn,m}+\beta B_{\nn,m}'\}_n\stackrel{\rm a.c.s.}{\longrightarrow}\{\alpha A_\nn+\beta A_\nn'\}_n$ for all $\alpha,\beta\in\mathbb C$,
		\item $\{B_{\nn,m}B_{\nn,m}'\}_n\stackrel{\rm a.c.s.}{\longrightarrow}\{A_\nn A_\nn'\}_n$ whenever $\{A_\nn\}_n,\{A_\nn'\}_n$ are s.u.,
		\item $\{B_{\nn,m}C_\nn\}_n\stackrel{\rm a.c.s.}\longrightarrow\{A_\nn C_\nn\}_n$ whenever $\{C_\nn\}_n$ is s.u.
	\end{itemize}
\item[\textbf{ACS\,4.}] Let $p\in[1,\infty]$ and assume for each $m$ there is $n_m$ such that, for $n\ge n_m$, $\|A_\nn-B_{\nn,m}\|_p\le\varepsilon(m,n)(N(\nn))^{1/p}$, where $\lim_{m\to\infty}\limsup_{n\to\infty}\varepsilon(m,n)=0$.
Then $\{B_{\nn,m}\}_n\stackrel{\rm a.c.s.}{\longrightarrow}\{A_\nn\}_n$.
\end{enumerate}

It turns out that the notion of a.c.s.\ begets a metric structure on the space of sequences $\mathscr E$. 
The distance
\[
\dacs{\ms{A}}{\ms{B}} = \limsup_{n\to \infty} p(A_\nn-B_\nn), \qquad
p(C_\nn)= \min_{i=1,\dots,N(\nn)+1}\left\{ \frac{i-1}{N(\nn)} + \sigma_i(C_\nn) \right\}
\]
where, by convention, $\sigma_{N(\nn)+1}(C_\nn)=0$, 
has been proved to induce the a.c.s. convergence between sequences.
Moreover,  $d_{\rm a.c.s.}(\{A_\nn \}_n,\{B_\nn \}_n)=0$ iff $\{A_\nn -B_\nn \}_n$ is zero-distributed, and $d_{\rm a.c.s.}$ turns $\,\mathscr E$ into a complete pseudometric space $(\mathscr E,d_{\rm a.c.s.} )$ where the statement ``$\{\{B_{\nn,m}\}_n\}_m$ converges to $\{A_\nn \}_n$'' is equivalent to ``$\{\{B_{\nn,m}\}_n\}_m$ is an a.c.s.\ for $\{A_\nn \}_n$''. In particular, we can reformulate the definition of a.c.s.\ in the following way: {\em a sequence of sequences of matrices $\{\{B_{\nn,m}\}_n\}_m$ is said to be an a.c.s.\ for $\{A_\nn \}_n$ if $\{B_{\nn,m}\}_n$ converges to $\{A_\nn \}_n$ in $(\mathscr E,d_{\rm a.c.s.} )$ as $m\to\infty$, i.e., if $d_{\rm a.c.s.} (\{B_{\nn,m}\}_n,\{A_\nn \}_n)\to0$ as $m\to\infty$.} The theory of a.c.s.\ may then be interpreted as an approximation theory for sequences of matrices, and for this reason we will use the convergence notation $\{B_{\nn,m}\}_n\stackrel{\rm a.c.s.}{\longrightarrow}\{A_\nn \}_n$ to indicate that $\{\{B_{\nn,m}\}_n\}_m$ is an a.c.s.\ for $\{A_\nn \}_n$. 

In view of what follows, let $D\subset\mathbb R^k$ be a measurable set such that $0<\mu_k(D)<\infty$ and define $\mc M_D$ the space of measurable functions over $D$. If
	\[
p_{\rm mea}(f) := \inf_{L\ge 0} \left\{ \frac{\ell_d\set{x\in D | |f|> L}}{\ell_d(D)} + L  \right\},
\]
	\[
d_{\rm mea}(f,g) = p_{\rm mea}(f-g),
\]
then $d_{\rm mea} $ is a distance on $\mc M_D$ such that $d_{\rm mea} (f,g)=0$ iff $f=g$ a.e.; moreover, $d_{\rm mea} $ turns $\mc M_D$ into a complete pseudometric space $(\mc M_D,d_{\rm mea} )$ where the statement ``$f_m$ converges to $f$'' is equivalent to ``$f_m$ converges to $f$ in measure''.

\subsection{Multilevel GLT}

We now recall the theory of the multilevel generalized locally Toeplitz (GLT) sequences and symbols. 
A $d$-level GLT sequence $\{A_\nn\}_n$ is a special $d$-level matrix-sequence equipped with a measurable function $\kappa:[0,1]^d\times[-\pi,\pi]^d\to\mathbb C$, the so-called GLT symbol. Unless otherwise specified, the notation $\{A_\nn\}_n\GLT\kappa$ means that $\{A_\nn\}_n$ is a $d$-level GLT sequence with symbol $\kappa$.
The symbol of a $d$-level GLT sequence is unique in the sense that if $\{A_\nn\}_n\GLT\kappa$ and $\{A_\nn\}_n\GLT\xi$ then $\kappa=\xi$ a.e.\ in $[0,1]^d\times[-\pi,\pi]^d$; conversely, if $\{A_\nn\}_n\GLT\kappa$ and $\kappa=\xi$ a.e.\ in $[0,1]^d\times[-\pi,\pi]^d$ then $\{A_\nn\}_n\GLT\xi$.
We report all the main properties of the GLT space summarized in 9 points.
\begin{enumerate} 
	\item[\textbf{GLT 1.}] If $\{A_\nn\}_n\sim_{\rm GLT}\kappa$ then $\{A_\nn\}_n\sim_{\sigma}\kappa$. If $\{A_\nn\}_n\sim_{\rm GLT}\kappa$ and each $A_\nn$ is normal, then $\{A_\nn\}_n\sim_{\lambda}\kappa$.
	\item[\textbf{GLT 2.}] If $\{A_\nn\}_n\sim_{\rm GLT}\kappa$ and $A_\nn=X_\nn+Y_\nn$, where
	\begin{itemize}
		\item every $X_\nn$ is Hermitian,
		\item $N(\nn)^{-1/2}\|Y_\nn\|_2\to 0$,
	\end{itemize}
	then $\{A_\nn\}_n\sim_{\lambda}\kappa$.
	\item[\textbf{GLT 3.}] Here we list three important examples of GLT sequences. 
	\begin{itemize}
		\item Given a function $f$ in $L^1([-\pi,\pi]^q)$, its associated Toeplitz sequence is $\{T_\nn(f)\}_n$, where the elements are multidimensional Fourier coefficients of $f$:
		\[ T_\nn( f ) = [ f_{\ii-\jj} ]^\nn_{\ii, \jj={\bf 1}}, \qquad f_\kk = \frac{1}{(2\pi)^q} \int_{-\pi}^{\pi} f(\btheta) e^{-\text i \kk\cdot\btheta} {{\rm d}}\theta. \]
		$\{T_\nn(f)\}_n$ is a GLT sequence with symbol $\kappa(\bfx,\btheta)=f(\btheta)$.
		\item Given an almost everywhere continuous function, $a:[0,1]^q\to\mathbb C$, its associated diagonal sampling sequence $\{D_\nn(a)\}_n$ is defined as
		\[ D_\nn(a) = \text{diag}\left(\left\{a\left(\frac \ii\nn\right) \right\}_{\ii=1}^\nn\right). \]
		$\{D_\nn(a)\}_n$  is a GLT sequence with symbol $\kappa(\bfx,\btheta)=a(\bfx)$.
		\item Any zero-distributed sequence $\{Z_\nn\}_n\svs0$ is a GLT sequence with symbol $\kappa(\bfx,\btheta)=0$ .
	\end{itemize}
	\item[\textbf{GLT 4.}] If $\{A_\nn\}_n\sim_{\rm GLT}\kappa$ and $\{B_\nn\}_n\sim_{\rm GLT}\xi$, then
	\begin{itemize}
		\item $\{A_\nn^H\}_n\sim_{\rm GLT}\overline\kappa$, where $A_\nn^H$ is the conjugate transpose of $A_\nn$,
		\item $\{\alpha A_\nn+\beta B_\nn\}_n\sim_{\rm GLT}\alpha\kappa+\beta\xi$ for all $\alpha,\beta\in\mathbb C$,
		\item $\{A_\nn B_\nn\}_n\sim_{\rm GLT}\kappa\xi$.
	\end{itemize}
	\item[\textbf{GLT 5.}] If $\{A_\nn\}_n\sim_{\rm GLT}\kappa$ and $\kappa\ne0$ a.e., then $\{A_\nn^\dag\}_n\sim_{\rm GLT}\kappa^{-1}$, where $A_\nn^\dag$ is the Moore--Penrose pseudoinverse of $A_\nn$. 
	\item[\textbf{GLT 6.}] If $\{A_\nn\}_n\sim_{\rm GLT}\kappa$ and each $A_\nn$ is normal, then $\{f(A_\nn)\}_n\sim_{\rm GLT}f(\kappa)$ for all continuous functions $f:\mathbb C\to\mathbb C$.
	\item[\textbf{GLT 7.}] $\{A_\nn\}_n\sim_{\rm GLT}\kappa$ if and only if there exist GLT sequences $\{B_{\nn,m}\}_n\sim_{\rm GLT}\kappa_m$ such that $\kappa_m$ converges to $\kappa$ in measure and $\{B_{\nn,m}\}_n\xrightarrow{\text{a.c.s.}}\{A_\nn\}_n$ as $m\to\infty$.
\item[\textbf{GLT\,8.}] If $\{A_\nn\}_n\sim_{\rm GLT}\kappa$ and $\{B_\nn\}_n\sim_{\rm GLT}\xi$ then $d_{\rm a.c.s.}(\{A_\nn\}_n,\{B_\nn\}_n)=d_{\rm mea}(\kappa,\xi)$.
	\item[\textbf{GLT 9.}]  
	For any measurable function $\kappa:[0,1]^d\times[-\pi,\pi]^d\to\mathbb C$ there exists a $d$-level sequence $\ms A$ 
	and functions  $a_{i,m},f_{i,m},\ i=1,\ldots,N_m$, such~that
	\begin{itemize}
		\item $\ms A\GLT \kappa$,
		\item $a_{i,m}\in C^\infty([0,1]^q)$ and $f_{i,m}$ is a trigonometric polynomial in $q$ variables,
		\item $\sum_{i=1}^{N_m}a_{i,m}(\bfx)f_{i,m}(\btheta)$ converges to $\kappa(\bfx,\btheta)$ a.e.,
		\item  $\bigl\{\sum_{i=1}^{N_m}D_{\nn}(a_{i,m})T_{\nn}(f_{i,m})\bigr\}_n\xrightarrow{\rm a.c.s.}\{A_\nn\}_n$ as $m\to\infty$. 
	\end{itemize}
\end{enumerate}

$ $\\
\noindent A similar scheme can be found in \cite{GLT-bookII}, where all the points are the same, except for \textbf{GLT1}, \textbf{GLT2} and \textbf{GLT6}, that can be deduced from the results in \cite{normal} and \cite{perturbation}. Moreover,
\textbf{GLT8} has been substituted with its more powerful version from \cite{GLT-bookIV}  and \textbf{GLT9} has been expanded to include the fact that every measurable function is a GLT symbol for some sequence. 

In the applications, one usually identifies the matrix-sequence at hand as a combination or limit of the simpler sequences in \textbf{GLT3}, for which a symbol is already known. Using the algebraic properties of \textbf{GLT4}, \textbf{GLT5} and \textbf{GLT6}, or the metric property of \textbf{GLT7}, one can compute the GLT symbol of the sequence, that is automatically a singular value symbol by \textbf{GLT1}. Eventually, using the perturbation result in \textbf{GLT2}, one can prove that the GLT symbol is also a spectral symbol.

\section{Characteristic Sequences}\label{characteristic_sequences}

We know by \textbf{GLT9} that every measurable function with support in $([0,1]\times[-\pi,\pi])^d$ is a GLT symbol for a sequence of matrices.
Using this connection, we can associate to each $\om$ a diagonal sequence $\ms{D}$ such that $\ms{D}\GLT \rchi_\om$.

An important remark to be noted here is that we do not have a single choice of domain, functions and sequence. In fact two measurable sets $\om,\om'$ are identified whenever they differ for a negligible set, and it happens if and only if $\rchi_\om$ and $\rchi_{\om'}$ differ on the same negligible set. Moreover, two sequences have the same GLT symbol if and only if they differ by a zero-distributed sequence by \textbf{GLT3} and \textbf{GLT4}. 

In the case of characteristic function, though, it is always possible to choose  $\ms{D}$ to be diagonal sequences with binary entries. This is easy to see in the case the characteristic function  $\rchi_\om$ is continuous almost everywhere, since we know from \textbf{GLT3}  that
\[
\{ D_\nn(\rchi_\om)  \}_n\GLT \rchi_\om.
\]
In the remaining cases, one can obtain $\rchi_\om$ as limit of characteristic functions of regular domains, so it is possible to reach the same conclusion using a diagonal argument.

Let us focus on the case $\rchi_\om$ is continuous a.e.,  that is a condition common to almost every domain used in linear PDE.  Given a  set $\om$, the following assertions are equivalent:
\begin{itemize}
	\item the function $\rchi_\om$ is continuous a.e.,
	\item the function $\rchi_\om$ is Riemann integrable,
	\item $\mu(\partial \om) =0$,
	\item the set $\om$ is Peano-Jordan measurable.
\end{itemize}
Moreover,  every measurable set $\om$ respecting the condition, is equal, up to a negligible set, to its interior $\om^\circ$ and to its closure $\ol\om$. 
The matrices $D_{\nn}(\rchi_\om)$ give us a natural way to link its rows and columns to the points of  type $z_{\ii}:=\frac{\bm i}{\bm{n}}$ with $\bm 1\le \bm i\le \nn$ inside and outside of $\om$.
A Peano-Jordan measurable set $\om$ is also well described by the diagonal matrices $D_\nn(\rchi_\om)$, and consequently by the points $z_\ii$, in the sense described by the following result.
\begin{lemma}\label{dimension}
	If $\om$ is a Peano-Jordan measurable set, then
	\[
	\lim_{n\to\infty} \frac{\rk(D_{\nn}(\rchi_\om)) }{N(\nn)} = \mu(\om).
	\]	
\end{lemma}
\begin{proof}
	We know from \textbf{GLT1} that \[
	\{ D_{\nn}(\rchi_\om)  \}_n\sim_\sigma \rchi_\om
	\]
	so in particular, if we consider a continuous function $F:\f R\to \f C$  with compact support and such that $F(1)=1$, $F(0)=0$, then
	\[
	\lim_{n\to\infty} \frac{\rk(D_{\nn}(\rchi_\om))}{N(\nn)} = \lim_{n\to\infty} \frac 1{N(\nn)} \sum_{i=1}^{N(\nn)} F(\sigma_i( D_{\nn}(\rchi_\om)) ) = \int_{[0,1]^d}  F(\rchi_\om(x))dx=\mu(\om).
	\]
\end{proof}
Actually, when $\om$ is Peano-Jordan measurable, we can show also that the number of points  $z_\ii$ arbitrarily  close to the boundary is negligible with respect to $N(\nn)$. Call 
\[
K_c = \set{ p\in[0,1]^d | d(p, \partial \om) \le c }
\]
the set of points whose distance from $\partial \om$ is at most $c\ge 0$. In the next result, we prove that $K_c$ contains few points $z_\ii$  when $c$ tends to zero, so that in the applications we can ignore the conditions that arise from grid points that are close enough to the boundary.
\begin{lemma}\label{preliminaries_near_border_points}
	Given a sequence $h_n$ of real nonnegative numbers converging to zero, and a Peano-Jordan measurable set $\om$, then 
	\[
	\lim_{n\to \infty} \frac{  
		\rk(D_{\nn}(\rchi_{K_{h_n}}))
	}{N(\nn)} = 0.
	\]
\end{lemma}
\begin{proof}
	Remember that $\partial \om$ is always a closed set contained into $[0,1]^d$. Notice that $K_c$ converge to $K_0 = \partial \om$ as $c$ tends to zero, so we know that
	\[
	\lim_{c\to 0} \mu(K_c) = \mu(\partial \om) = 0.
	\]
	$K_c$ is a closed subset of $[0,1]^d$ for every $c$ since 
	\[
	p\not\in K_c\implies p\not\in[0,1]^d \vee d(p,\partial \om)>c
	\]
	and in both case there's an open neighbourhood of $p$ disjoint from $K_c$. 
	Moreover, if $c>0$ then
	\[
	p\in \partial K_c \implies p\in\partial [0,1]^d \vee d(p,\partial \om)=c
	\]
	and it is known that the set of points at fixed positive distance from a closed set is negligible \cite{erdos}, so we can conclude that $\mu(\partial K_c) =0$. This is actually true also for $K_0$ since \[\partial K_0 = \partial \partial \om\cu \partial \om\implies \mu(\partial K_0)\le \mu(\partial \om) = 0.\]  
	We can thus use  \autoref{dimension} to infer that for every $c\ge 0$ 
	\[
	\lim_{n\to \infty}\frac{\rk(D_{\nn}(\rchi_{K_c}))}{N(\nn)}  = \mu(K_c).
	\]
	Notice that if $h_n	< h_m$ then $K_{h_n}\cu K_{h_m}$ and consequently $\rk(D_{\nn}(\rchi_{K_{h_n}}))\le \rk(D_{\nn}(\rchi_{K_{h_m}}))$. When we fix an index $m>0$, we know that definitively $h_n< h_m$ since $h_n$ are converging to zero, so the following relation holds
	\[
	\limsup_{n\to \infty}\frac{\rk(D_{\nn}(\rchi_{K_{h_n}}))}{N(\nn)}  \le \limsup_{n\to \infty}\frac{\rk(D_{\nn}(\rchi_{K_{h_m}}))}{N(\nn)}  = \mu(K_{h_m}) \qquad \forall m 
	\]
	\[
	\implies \limsup_{n\to \infty}\frac{\rk(D_{\nn}(\rchi_{K_{h_n}}))}{N(\nn)} \le \inf_{m\in \f N}  \mu(K_{h_m}) = 0.
	\] 
\end{proof}

The points $z_\ii$ form an uniform grid on $[0,1]^d$, but in applications the most used grid, denoted as $\Xi_n$, is composed by points of the form 
\[
\frac{\ii}{\nn+\uu} = \left(
\frac{i_1}{n_1+1}, \frac{i_2}{n_2+1}, \dots, \frac{i_d}{n_d+1} 
\right),
\qquad i_j=0,1,2,\dots,n_j,n_j+1,\quad \, j=1,2,\dots,d.
\]
Consequentially we define a new diagonal matrix associated to $\om$ 
\[
I_{\nn}(\rchi_\om) := \diag\left( \rchi_\om\left( \frac{\ii}{\nn+\uu} \right)  \right)_{\ii = \uu,\dots,\nn}
\]
that has dimension $N(\nn)\times N(\nn)$, the same as $D_{\nn}(\rchi_\om)$.
More in general, for any continuous a.e. function $a:[0,1]^d\to  \f C$ we denote
\[
I_{\nn}(a) := \diag\left( a\left( \frac{\ii}{\nn+\uu} \right)  \right)_{\ii = \uu,\dots,\nn}
\] 
so that $I_{\nn}(a)$ and $D_{\nn}(a)$ have the same dimension, and can actually be proved that they enjoy the same GLT and spectral symbol.

\begin{lemma}\label{preliminaries_equivalence_D_nn_a}
	If $a:[0,1]^d\to  \f C$ is a continuous a.e. function, then
	\[
	\{ I_{\nn}(a)  \}_n\GLT a.
	\]
\end{lemma}
\begin{proof}
	Notice that $a:[0,1]^d\to  \f C$ is a continuous a.e. if and only if when we split it into real and imaginary part $a = a_1 + \textnormal i a_2$, both the real functions $a_1$ and $a_2$ are continuous a.e.. In the same way, we can split $a_1$ and $a_2$ in their positive and negative parts, and they are still continuous a.e..  By \textbf{GLT4}, we can thus suppose that $a:[0,1]^d\to \f R^+$, since it is sufficient to prove the general thesis.
	

	The proof is divided into 3 steps, where we prove that the statement holds first when $a$ is continuous, then when $a$ is Riemann-integrable and eventually when $a$ is continuous a.e..\\
	
	\textit{Step 1.} Suppose $a$ is continuous and call $\omega_a$ its continuity module. Notice that
	\[
	\left\|  \frac{\ii}{\nn} - \frac{\ii}{\nn+\uu}\right\|_2^2 \le \sum_{k=1}^d \left( \frac{i_k}{n_k(n_k+1)}    \right)^2  \le   \sum_{k=1}^d\frac{1}{n_k^2}=:h_n^2 \xrightarrow{n\to\infty} 0,
	\] 
	so we can obtain a bound on the norm of  $I_\nn(a) - D_\nn(a)$ as
	\[
	\| I_\nn(a) - D_\nn(a) \| = \max_{\ii=\uu,\dots,\nn} \left| a\left( \frac{\ii}{\nn+\uu} \right) - a\left( \frac{\ii}{\nn} \right)  \right|
	\le \omega_a(h_n)\xrightarrow{n\to\infty} 0.
	\]
	By \textbf{Z1}, $\{I_\nn(a) - D_\nn(a) \}_n$ is zero-distributed and consequentially \textbf{GLT4} tells us that $\{ I_{\nn}(a)  \}_n\GLT a$.\\
	
	\textit{Step 2.} Suppose $a$ is Riemann-integrable, and consider a sequence of continuous function $a_m$ converging to $a$ in $L^1$ norm. 
	A continuous function is in particular Riemann-integrable, so $a_m-a$ is also Riemann-integrable and we can compute
	\[
	N(\nn)^{-1}\|I_{\nn}(a_m)-I_{\nn}(a)\|_1 = \frac{1}{N(\nn)}\sum_{\ii=\uu}^{\nn} \left| a_m\left( \frac{\ii}{\nn+\uu} \right) - a\left( \frac{\ii}{\nn+\uu} \right)  \right|\xrightarrow{n\to\infty} \| a-a_m\|_1 \xrightarrow{m\to\infty} 0.
	\]
	We can thus write the difference as $\|I_{\nn}(a_m)-I_{\nn}(a)\|_1 = N(\nn)\ve(n,m)$ where $\lim_{m\to \infty}\lim_{n\to \infty}\ve(n,m)=0$ and using \textbf{ACS 4}, we discover that $\{I_{\nn}(a_m)\}_n\acs \{I_{\nn}(a)\}_n$.
	We know from Step 1 that $\{ I_{\nn}(a_m)  \}_n\GLT a_m$ for every $m$, and $a_m\xrightarrow{m\to\infty} a$ in measure, so we conclude that $\{ I_{\nn}(a)  \}_n\GLT a$ by \textbf{GLT 7}.\\
	
	\textit{Step 3.} Suppose $a$ is continuous a.e and call $a_m(x): = \max\{ a(x),m\}$ its truncated function for every $m\in\f N$. Notice that $a_m$ are still continuous a.e. and also bounded, thus Riemann-integrable. Moreover, since $a$ is measurable we know that 
	\[
	\mu\{ x | a(x)>m \} =: h_m \xrightarrow{m\to\infty} 0.
	\]
	We know from Step 2 that $\{ I_{\nn}(a_m)  \}_n\GLT a_m$ for every $m$, so we can fix $1>\ve>0$ and consider $G_m(x)$ continuous and compact supported functions such that $\rchi_{[0,m-\ve ]}\le G_m\le  \rchi_{[-\ve,m]}$ to obtain
	\begin{align*}
	N(\nn)^{-1}\rk( I_\nn(a_m) - I_\nn(a) ) 
	&=  N(\nn)^{-1}\#\Set{\ii | a\left(\frac{\ii}{\nn + \uu} \right)>m, \, \uu\le \ii\le \nn}\\
	&= 1 - N(\nn)^{-1}\#\Set{\ii | a\left(\frac{\ii}{\nn + \uu} \right)\le m, \, \uu\le \ii\le \nn}\\
	&\le  1 - N(\nn)^{-1}\sum_{\ii=\uu}^{\nn} G_m(\sigma_\ii(D_\nn(A_m)) ).
	\end{align*}
	Note that $G_m(m)=0$, so $G_m(a_m)=G_m(a)$ and taking the limits of the preceding relations, one can see that 
	\begin{align*}
	\limsup_{n\to \infty}N(\nn)^{-1}\rk( I_\nn(a_m) - I_\nn(a) ) &\le 1 - \frac{1}{(2\pi)^d}\int_{[0,1]^d\times[-\pi,\pi]^d} G_m(a_m(x))dx\\
	&=  1 - \frac{1}{(2\pi)^d}\int_{[0,1]^d\times[-\pi,\pi]^d} G_m(a(x))dx\\
	&\le 1 -  \frac{1}{(2\pi)^d}\int_{[0,1]^d\times[-\pi,\pi]^d} \rchi_{[0,m-\ve ]}(a(x))dx\\
	&\le 1 -  \frac{(2\pi)^d - h_{m-1}}{(2\pi)^d}=:c(m)\xrightarrow{m\to\infty} 0.\\
	\end{align*}
	Consequently, for every $m$ we can find $n_m$ such that for every $n>n_m$, $\rk( I_\nn(a_m) - I_\nn(a) ) \le c(m) N(\nn)$ with $c(m)\xrightarrow{m\to\infty} 0$, and it leads to  $\{I_{\nn}(a_m)\}_n\acs \{I_{\nn}(a)\}_n$.
	We know that $a_m\xrightarrow{m\to\infty} a$ in measure, so we conclude again by \textbf{GLT7} that $\{ I_{\nn}(a)  \}_n\GLT a$.
\end{proof}
This result shows that for every $a:[0,1]^d\to  \f C$ continuous a.e. function, the sequences $\{ I_{\nn}(a)  \}_n$ and $\{ D_{\nn}(a)  \}_n$ have the same GLT (and consequently, spectral) symbol. In particular, if $\om$ is Peano-Jordan measurable, $\rchi_{\om}$ is continuous a.e., so $\{ I_{\nn}(\rchi_\om)  \}_n\GLT \rchi_\om$. In this case, it is also possible show that the difference $I_{\nn}(\rchi_\om) -D_{\nn}(\rchi_\om)$ has rank negligible when compared to $N(\nn)$.

\begin{lemma}\label{preliminaries_equivalence_D_nn_chi}
	If $\om$ is Peano-Jordan measurable, then
	\[
	\rk \left(I_{\nn}(\rchi_\om) -D_{\nn}(\rchi_\om)\right) = o(N(\nn)).
	\]
\end{lemma}
\begin{proof}
	It is enough to show that
	\[
	E_n:= \Set{\ii | \rchi_{\om}\left( \frac{\ii}{\nn}  \right) \ne \rchi_{\om}\left( \frac{\ii}{\nn+\uu}\right) ,  \uu\le \ii\le \nn    }
	\]
	has cardinality negligible when compared to $N(\nn)$, since
	\[
	\#E_n = \rk \left(I_{\nn}(\rchi_\om) -D_{\nn}(\rchi_\om)\right).
	\]
	Note that if $\ii\in E_n$ then there's a point of the boundary $\partial \om$ on the segment connecting the points $\ii/\nn$ and $\ii/(\nn +\uu)$. The distance between the two points is always  bounded and tends to zero when $n$ goes to infinity
	\[
	\left\|  \frac{\ii}{\nn} - \frac{\ii}{\nn+\uu}\right\|_2^2 \le \sum_{k=1}^d \left( \frac{i_k}{n_k(n_k+1)}    \right)^2  \le   \sum_{k=1}^d\frac{1}{n_k^2}=:h_n^2 \xrightarrow{n\to\infty} 0.
	\] 
	It means that for every $\ii\in E_n$ we have  $d(\ii/\nn,\partial \om)\le h_n$, so \autoref{preliminaries_near_border_points} let us conclude that
	\begin{align*}
	\limsup_{n\to \infty} \frac{\#E_n}{N(\nn)} &=
	\limsup_{n\to \infty} \frac{\#\Set{\frac{\ii}{\nn} | \rchi_{\om}\left( \frac{\ii}{\nn}  \right) \ne \rchi_{\om}\left( \frac{\ii}{\nn+\uu}\right) ,  \uu\le \ii\le \nn    }}{N(\nn)}\\
	&\le 	\limsup_{n\to \infty} \frac{\#\Set{\frac{\ii}{\nn} |d\left(\frac{\ii}{\nn},\partial \om\right)\le h_n ,  \uu\le \ii\le \nn    }}{N(\nn)}\\
	&= 	\limsup_{n\to \infty} \frac{  
		\rk(D_{\nn}(\rchi_{K_{h_n}}))
	}{N(\nn)} =0
	\end{align*}
\end{proof}

The latest result shows that the two diagonal sequences 
$	\{ I_{\nn}(\rchi_\om)  \}_n$ and
$	\{ D_\nn(\rchi_\om)  \}_n$
hold essentially the same information about the domain $\om$. We adopt the second one since later it will be fundamental to operate on the grid $\Xi_n$ through  diagonal matrices.
In particular, the quantity
\[
d_n^\om := \rk(I_\nn(\rchi_\om))
\]
is important since it counts the number of grid points inside  $\om$. As a corollary, we find again the same results of \autoref{dimension} and \autoref{preliminaries_near_border_points}, referred to the sequence $	\{ I_{\nn}(\rchi_\om)  \}_n$. We will not prove them, since the arguments are the same we used in the proofs of \autoref{dimension} and \autoref{preliminaries_near_border_points}.
\begin{corollary}\label{dimension_chi}
	If $\om$ is a Peano-Jordan measurable set, then
	\[
	\lim_{n\to\infty} \frac{d_n^\om  }{N(\nn)} = \mu(\om).
	\]	
\end{corollary}
\begin{corollary}\label{preliminaries_near_border_points_chi}
	Given a sequence $h_n$ of real nonnegative numbers converging to zero, and a Peano-Jordan measurable set $\om$, then 
	\[
	\lim_{n\to \infty} \frac{  
		d_n^{K_{h_n}}
	}{N(\nn)} = 0.
	\]
	In particular, if $\mu(\om) >0$, then 
	\[
	\lim_{n\to \infty} \frac{  
		d_n^{K_{h_n}}	
	}{d_n^\om} = 0.
	\]
\end{corollary}

Note that if $h_n=0$ for every $n$, we have $K_{h_n} = K_0=\partial \om$ for every $n$, so $d_n^{\partial\om} = o(d_n^\om) = o(N(\nn))$.
As a corollary, we can also derive  the limits of $d_n^{\ol\om}(N(\nn))^{-1}$ and $d_n^{\om^\circ}(N(\nn))^{-1}$, since we know that $\ol \om$ and $\om^\circ$ differ from $\om$ for a negligible set inside $\partial \om$.
\[
\om \cup \partial\om = \ol\om \supseteq \om \implies d_n^{\om} + d_n^{\partial\om}\ge d_n^{\ol\om}\ge d_n^{\om}\implies \lim_{n\to\infty} \frac{d_n^{\ol\om}}{N(\nn)} = \mu(\om),
\]
\[
\om \setminus \partial\om = \om^\circ \cu \om \implies d_n^{\om} - d_n^{\partial\om}\le d_n^{\om^\circ}\le d_n^{\om}\implies \lim_{n\to\infty} \frac{d_n^{\om^\circ}}{N(\nn)} = \mu(\om).
\]
Notice that   Corollary \autoref{dimension_chi} shows  $\lim_{n\to\infty} d^\om_n = +\infty$ whenever the measure of $\om$ is not zero, so from now on, we  suppose that $\mu(\om)>0$.

\section{Restriction and Expansion Operators}\label{Res_Exp}

If we fix a Peano-Jordan measurable set $\om$, then we can build the map
\[
Z_\om :\{A_{\nn}\}_n \mapsto \{I_{\nn}(\rchi_\om) A_{\nn}I_{\nn}(\rchi_\om)   \}_n.
\] 
From now on, we abuse the notation and write $Z_\om(A_\nn)$ for the matrix $I_{\nn}(\rchi_\om) A_{\nn}I_{\nn}(\rchi_\om) $.
If we call $\mc G_d$ the set of $d$-dimensional GLT sequences, notice that $Z_\om(\mc G_d)\cu \mc G_d$ by \textbf{GLT4}, since it multiplies a GLT sequence with other GLT sequences, as shown in \autoref{preliminaries_equivalence_D_nn_chi}. Some properties of this operation are
\begin{itemize}
	\item $Z_\om$ is linear,
	\item $Z_\om$ is idempotent, 
	\item if $\{A_{\nn}\}_n\GLT k(x,\theta)$, then $Z_\om(\{A_{\nn}\}_n)\GLT k(x,\theta)\rchi_\om(x)$, 
	\item if $\{A_{\nn}\}_n$ is a real sequence, then $Z_\om(\{A_{\nn}\}_n)$ is still real,
	\item  if $\{A_{\nn}\}_n$ is a Hermitian sequence, then $Z_\om(\{A_{\nn}\}_n)$ is still Hermitian.
\end{itemize} 
If we associate  each multi-index $\bm i$ in the matrix $A_{\nn}$  to the point $\frac{\bm i}{\bm{n}+\uu}\in \Xi_n$, then $Z_\om$  sets to zero every row and column corresponding to a point not in $\om$. We can thus try to delete the zero rows and columns in the matrices, and obtain a matrix with size $d_n^\om\times d_n^\om$.

Given a set $\om$ with negligible boundary, we consider $I_{\nn}(\rchi_\om) $ and we enumerate the non-zero rows and the zero rows through two strictly increasing functions
\[
\phi_{\nn} :\{1,2,\dots, d_n^\om\} \to \{ \uu,\dots,\nn \}
\qquad 
\psi_{\nn} :\{d_n^\om + 1,d_n^\om +2,\dots,N(\nn) \} \to \{ \uu,\dots,\nn \}
\]
such that the $\phi_{\nn}(j)$-th row of $I_{\nn}(\rchi_\om) $ is non-zero for every $j$, and the $\psi_{\nn}(j)$-th row of $I_{\nn}(\rchi_\om) $ is zero for every $j$.
In particular, the images of $\phi_\nn$ and $\psi_\nn$ correspond to  the set of points $\ii/(\nn+\uu)$ in $\Xi_n$
respectively belonging and not belonging to $\om$.
Notice that $\phi_\nn$ and $\psi_\nn$ are uniquely determined by their properties.

For every $\nn$, we define a rectangular matrix $\Pi_{\nn,\om}$ of size $d_n^\om\times N(\nn)$ as
\[(
\Pi_{\nn,\om})_{i,\jj} := (I_{\nn}(\rchi_\om) )_{\phi_\nn(i),\jj}
\] 
so that, for any matrix $A_\nn$ of size $N(\nn)\times N(\nn)$, we can delete the rows and columns corresponding to points not belonging to $\om$  
through the  map
\[
R_\om :\{A_{\nn}\}_n \mapsto \{\Pi_{\nn,\om} A_{\nn}(\Pi_{\nn,\om})^T  \}_n
\] 
and add zero rows and columns corresponding to points not belonging to $\om$   to any matrix  $S_\nn^\om$ of size $d_n^\om\times d_n^\om$
through the  map
\[
E_\om : \ms {S^\om} \mapsto \{(\Pi_{\nn,\om})^TS^\om_\nn \Pi_{\nn,\om} \}_n.
\] 
We will use the notation $R_\om(A_\nn)$ for $\Pi_{\nn,\om} A_{\nn}(\Pi_{\nn,\om})^T $ and the notation $E_\om(S^\om_\nn)$ for $(\Pi_{\nn,\om})^TS^\om_\nn \Pi_{\nn,\om}$. Moreover, unless differently specified, we use the exponent $\om$ to distinguish the sequences $\ms {S^\om}$ of size $d_n^\om\times d_n^\om$ from classical sequences $\ms A$ of size $N(\nn)\times N(\nn)$.

\begin{remark}
	Note that the operators $E_\om,R_\om,Z_\om$, the matrices $\Pi_{\nn,\om}, I_{\nn}(\rchi_\om) $  and the quantity $\dno$   can be defined for any measurable set $\om$, even if not Peano-Jordan measurable. 
\end{remark}

\subsection{Effects on the Sequences}

Let us check some basic properties of the matrices $\Pi_{\nn,\om}, I_{\nn}(\rchi_\om) $ and the operators $E_\om,R_\om,Z_\om$.

\begin{lemma}\label{MDREZ_relations}
	For every index $\nn$, we have
	\begin{itemize}
		\item $(\Pi_{\nn,\om})^T\Pi_{\nn,\om} = I_{\nn}(\rchi_\om) $,
		\item $ \Pi_{\nn,\om}(\Pi_{\nn,\om})^T = I_\nn^\om$.
	\end{itemize}
	In particular, given any matrix $A_\nn$ of size $N(\nn)\times N(\nn)$, and any matrix $S^\om_\nn$ of size $d_n^\om\times d_n^\om$, we have
	\begin{itemize}
		\item $R_\om (A_{\nn}) = R_\om \circ Z_\om (A_{\nn})$,
		\item $R_\om(E_\om(S^\om_\nn)) = S^\om_\nn$,
		\item $E_\om(R_\om(A_{\nn})) = Z_\om (A_{\nn})$,
		\item $Z_\om(E_\om(S^\om_\nn)) = E_\om(S^\om_\nn)$.
	\end{itemize}
	Moreover $(E_\om(S^\om_\nn))^* = E_\om((S^\om_\nn)^*)$ and $(R_\om(A_\nn))^* = R_\om(A^*_\nn)$, so 
	\begin{itemize}
		\item $S^\om_\nn$ Hermitian $\implies E_\om(S^\om_\nn)$ Hermitian,
		\item $A_\nn$ Hermitian $\implies R_\om(A_\nn)$ Hermitian.
	\end{itemize}
\end{lemma}
\begin{proof}
	\begin{align*}
	((\Pi_{\nn,\om})^T\Pi_{\nn,\om}  )_{\ii,\jj}
	&=\sum_{k=1}^{d_n^\om}
	(\Pi_{\nn,\om})_{k,\ii} 
	(\Pi_{\nn,\om})_{k,\jj} \\
	&= \sum_{k=1}^{d_n^\om} (I_{\nn}(\rchi_\om) )_{\phi_\nn(k),\ii}
	(I_{\nn}(\rchi_\om) )_{\phi_\nn(k),\jj}\\
	&= 
	\begin{cases}
	1 & \ii=\jj\in Range(\phi_\nn)\\
	0 & \textnormal{otherwise} 
	\end{cases} =
	(I_{\nn}(\rchi_\om) )_{\ii,\jj},
	\end{align*}
	\begin{align*}
	(\Pi_{\nn,\om} (\Pi_{\nn,\om})^T )_{i,j}
	&=\sum_{\bm k=\uu}^{\nn}
	(\Pi_{\nn,\om})_{i,\bm k} 
	(\Pi_{\nn,\om})_{j,\bm k} \\
	&= \sum_{\bm k=\uu}^{\nn} (I_{\nn}(\rchi_\om) )_{\phi_\nn(i),\bm k}
	(I_{\nn}(\rchi_\om) )_{\phi_\nn(j),\bm k}\\
	&= \delta_{\phi_\nn(i),\phi_\nn(j)} =\delta_{i,j},
	\end{align*}
	\begin{align*}
	R_\om (A_{\nn}) &= \Pi_{\nn,\om} A_{\nn}(\Pi_{\nn,\om})^T \\
	&=   I_\nn^\om \Pi_{\nn,\om} A_{\nn}(\Pi_{\nn,\om})^TI_\nn^\om      \\
	&=  \Pi_{\nn,\om} (\Pi_{\nn,\om})^T\Pi_{\nn,\om} A_{\nn}(\Pi_{\nn,\om})^T \Pi_{\nn,\om} (\Pi_{\nn,\om})^T   \\
	&=   \Pi_{\nn,\om} I_{\nn}(\rchi_\om)  A_{\nn}I_{\nn}(\rchi_\om)  (\Pi_{\nn,\om})^T   \\
	&=	R_\om \circ Z_\om ( A_{\nn}  ),
	\end{align*}
	\begin{align*}
	R_\om(E_\om( S_\nn^\om)) &=  \Pi_{\nn,\om} (\Pi_{\nn,\om})^TS^\om_{\nn}\Pi_{\nn,\om}(\Pi_{\nn,\om})^T   \\
	&=  I_\nn^\om S^\om_{\nn}I_\nn^\om    \\
	&=  S_\nn^\om,
	\end{align*}
	\begin{align*}
	E_\om(R_\om(A_\nn)) &=  (\Pi_{\nn,\om})^T\Pi_{\nn,\om} A_{\nn}(\Pi_{\nn,\om})^T\Pi_{\nn,\om}   \\
	&=  I_{\nn}(\rchi_\om) A_{\nn}I_{\nn}(\rchi_\om)    \\
	&= Z_\om(A_\nn),
	\end{align*}
	\[
	R_\om(E_\om( S_\nn^\om)) =  S_\nn^\om\implies E_\om( S_\nn^\om) = E_\om(R_\om(E_\om( S_\nn^\om))) = Z_\om (E_\om( S_\nn^\om)),
	\]
	\[
	(E_\om( S_\nn^\om))^*  = ((\Pi_{\nn,\om})^T S_\nn^\om \Pi_{\nn,\om})^* =  (\Pi_{\nn,\om})^T (S_\nn^\om)^* \Pi_{\nn,\om} = E_\om( (S_\nn^\om)^* ),
	\]
	\[
	(R_\om(A_\nn))^*  = (\Pi_{\nn,\om} A_\nn(\Pi_{\nn,\om})^T)^* =  \Pi_{\nn,\om} A_\nn^* (\Pi_{\nn,\om})^T = R_\om( A_\nn^* ).
	\]
\end{proof}

The operator $R_\om$ has the job to extract a principal minor from the matrices, so it is easy to see that it makes the norm drop.
\begin{lemma}\label{norm_reduction}
	For every $1\le p\le \infty$, 
	\[
	\|R_\om (A)\|_p\le \|A\|_p.
	\] 
\end{lemma}
\begin{proof}
	The matrices $\Pi_{\nn,\om}$ are unitary, so we can apply the Cauchy interlacing theorem and find that
	\[
	\sigma_i(R_\om (A)) \le \sigma_i(A) \qquad \forall 1\le i\le d_n^\om.
	\] 
	The thesis easily follows from the definition of  $p$-Schatten norm.
\end{proof}

The map $R_\om$ applied to $Z_\om(A_{\nn})$ has the effect to delete only rows and columns that are already zero, and we can easily tell the behaviour of their singular values and eigenvalues. 

\begin{lemma}\label{ker_reduction}
	There exists a permutation matrix $P$ of size $N(\nn)\times N(\nn)$  such that 
	for every matrix $A_\nn$ of size $N(\nn)\times N(\nn)$,  
	\[
	P Z_\om(A_\nn) P^T =
	\begin{pmatrix}
	R_\om(A_\nn) & 0\\0 & 0
	\end{pmatrix}.
	\]
	In particular, $ Z_\om(A_\nn)$
	has the same eigenvalues and singular values of the matrix $R_\om(A_\nn)$
	except for $N(\nn) - d_n^\om$ null eigenvalues and singular values.
\end{lemma}
\begin{proof}
	Let $B_\nn = Z_\om(A_\nn)$ and $S^\om_\nn=R_\om(A_\nn)$.
	If we define the permutation matrix $P$ as
	\[
	P_{\ii,\jj} = 
	\begin{cases}
	\delta_{\jj,\phi_\nn(|\ii|)} & |\ii|\le d_n^\om,\\
	\delta_{\jj,\psi_\nn(|\ii|)} & |\ii|> d_n^\om,
	\end{cases}
	\]
	then the matrix $P B_\nn P^T$ can be written as
	\[
	P B_\nn P^T =
	\begin{pmatrix}
	S^\om_\nn & 0\\0 & 0
	\end{pmatrix}.
	\]
	In fact
	\begin{align*}
	(P B_\nn P^T)_{\ii,\jj} &=
	(P I_{\nn}(\rchi_\om)  A_{\nn} I_{\nn}(\rchi_\om)  P^T)_{\ii,\jj} \\
	&=
	\sum_{\bm k=\uu}^{\nn}\sum_{\bm h=\uu}^{\nn}
	P_{\ii,\bm k} (I_{\nn}(\rchi_\om) )_{\bm k,\bm k}
	(A_\nn)_{\bm k,\bm h}
	(I_{\nn}(\rchi_\om) )_{\bm h,\bm h}P_{\jj,\bm h} \\
	&= 
	\begin{cases} 
	(I_{\nn}(\rchi_\om) )_{\phi_\nn(|\ii|),\phi_\nn(|\ii|)}
	(A_\nn)_{\phi_\nn(|\ii|),\phi_\nn(|\jj|)}
	(I_{\nn}(\rchi_\om) )_{\phi_\nn(|\jj|),\phi_\nn(|\jj|)}
	& |\ii|\le d_n^\om, |\jj|\le d_n^\om\\
	(I_{\nn}(\rchi_\om) )_{\psi_\nn(|\ii|),\psi_\nn(|\ii|)}
	(A_\nn)_{\psi_\nn(|\ii|),\phi_\nn(|\jj|)}
	(I_{\nn}(\rchi_\om) )_{\phi_\nn(|\jj|),\phi_\nn(|\jj|)}
	& |\ii|> d_n^\om,|\jj|\le d_n^\om\\
	(I_{\nn}(\rchi_\om) )_{\phi_\nn(|\ii|),\phi_\nn(|\ii|)}
	(A_\nn)_{\phi_\nn(|\ii|),\psi_\nn(|\jj|)}
	(I_{\nn}(\rchi_\om) )_{\psi_\nn(|\jj|),\psi_\nn(|\jj|)}
	& |\ii|\le d_n^\om,|\jj|> d_n^\om\\
	(I_{\nn}(\rchi_\om) )_{\psi_\nn(|\ii|),\psi_\nn(|\ii|)}
	(A_\nn)_{\psi_\nn(|\ii|),\psi_\nn(|\jj|)}
	(I_{\nn}(\rchi_\om) )_{\psi_\nn(|\jj|),\psi_\nn(|\jj|)}
	& |\ii|> d_n^\om,|\jj|> d_n^\om
	\end{cases}\\
	&= 
	\begin{cases} 
	(A_\nn)_{\phi_\nn(|\ii|),\phi_\nn(|\jj|)}
	& |\ii|\le d_n^\om, |\jj|\le d_n^\om\\
	0& |\ii|> d_n^\om,|\jj|\le d_n^\om\\
	0& |\ii|\le d_n^\om,|\jj|> d_n^\om\\
	0& |\ii|> d_n^\om,|\jj|> d_n^\om
	\end{cases}
	\end{align*}
	and
	\begin{align*}
	(S^\om_\nn)_{i,j} = R_\om(A_\nn)&=
	(\Pi_{\nn,\om} A_{\nn} (\Pi_{\nn,\om})^T )_{i,j} \\
	&=
	\sum_{\bm k=\uu}^{\nn}\sum_{\bm h=\uu}^{\nn}
	(\Pi_{\nn,\om})_{i,\bm k} 
	(A_\nn)_{\bm k,\bm h}
	(\Pi_{\nn,\om})_{j,\bm h} \\
	&= \sum_{\bm k=\uu}^{\nn}\sum_{\bm h=\uu}^{\nn} (I_{\nn}(\rchi_\om) )_{\phi_\nn(i),\bm k}
	(A_\nn)_{\bm k,\bm h}
	(I_{\nn}(\rchi_\om) )_{\phi_\nn(j),\bm h}\\
	&= (A_\nn)_{\phi_\nn(i),\phi_\nn(j)}.
	\end{align*}
	The proof is thus concluded, since $S^\om_\nn$ has the same eigenvalues and singular values of $B_\nn$ except for $N(\nn) -d_n^\om$ zeros.
\end{proof}

\begin{corollary}\label{ker_expansion}
	There exists a permutation matrix $P$ of size $N(\nn)\times N(\nn)$ such that for every  matrix $S^\om_\nn$ of size $d_n^\om\times d_n^\om$,
	\[
	P E_\om(S^\om_\nn) P^T =
	\begin{pmatrix}
	S^\om_\nn & 0\\0 & 0
	\end{pmatrix}.
	\]
	In particular, $E_\om(S^\om_\nn)$
	has the same eigenvalues and singular values of the matrix $S^\om_\nn$
	except for $N(\nn) - d_n^\om$ null eigenvalues and singular values.
\end{corollary}
\begin{proof}Let $A_\nn = E_\om(S^\om_\nn)$.
	Using \autoref{MDREZ_relations}, we get 
	\[
	S^\om_\nn  = R_\om(E_\om( S^\om_\nn )) = R_\om( A_\nn ), \qquad 
	Z_\om( A_\nn ) = Z_\om (E_\om( S^\om_\nn )) =E_\om( S^\om_\nn ) =  A_\nn .
	\]
	As a consequence, we can apply \autoref{ker_reduction} on $ A_\nn $ to find a permutation matrix $P$ such that 
	\[
	P  A_\nn P^T = 
	\begin{pmatrix}
	S^\om_\nn & 0\\0 & 0
	\end{pmatrix}
	\]
	so $S^\om_\nn$ has the same eigenvalues and singular values of $A_\nn$ except for $N(\nn) -d_n^\om$ zeros.
\end{proof}

\begin{corollary}\label{R_principal_submatrix}
	There exists a permutation matrix $P$ of size $N(\nn)\times N(\nn)$  such that 
	for every matrix $A_\nn$ of size $N(\nn)\times N(\nn)$,  
	\[
	P A_\nn P^T =
	\begin{pmatrix}
	R_\om(A_\nn) & *\\ * & *
	\end{pmatrix}.
	\]
\end{corollary}
\begin{proof}
	using \autoref{MDREZ_relations},
	\[
	Z_\om (I_{\nn}(\rchi_\om) ) = I_{\nn}(\rchi_\om) I_{\nn}(\rchi_\om) I_{\nn}(\rchi_\om) =I_{\nn}(\rchi_\om) ,\]\[
	R_\om(I_{\nn}(\rchi_\om) ) = \Pi_{\nn,\om} I_{\nn}(\rchi_\om) (\Pi_{\nn,\om})^T=
	\Pi_{\nn,\om}(\Pi_{\nn,\om})^T\Pi_{\nn,\om}(\Pi_{\nn,\om})^T=
	I_\nn^\om I_\nn^\om    = I_\nn^\om ,
	\]
	so \autoref{ker_reduction} shows that there exists $P$ such that
	\[
	PI_{\nn}(\rchi_\om) P^T = PZ_\om(I_{\nn}(\rchi_\om) ) P^T = \begin{pmatrix}
	R_\om(I_{\nn}(\rchi_\om) ) &0\\0&0
	\end{pmatrix}
	=
	\begin{pmatrix}
	I_\nn^\om  &0\\0&0
	\end{pmatrix}.
	\]
	As a consequence, we have that
	\begin{align*}
	\begin{pmatrix}
	R_\om(A_\nn) &0\\0&0
	\end{pmatrix}
	&=
	PZ_\om(A_\nn) P^T \\
	&= P I_{\nn}(\rchi_\om) A_\nn I_{\nn}(\rchi_\om) P^T\\
	&= P I_{\nn}(\rchi_\om) P^T PA_\nn P^TP I_{\nn}(\rchi_\om) P^T\\
	&=
	\begin{pmatrix}
	I_\nn^\om  &0\\0&0
	\end{pmatrix}
	PA_\nn P^T
	\begin{pmatrix}
	I_\nn^\om  &0\\0&0
	\end{pmatrix}\\
	\implies 
	\begin{pmatrix}
	R_\om(A_\nn) &*\\ *&* 
	\end{pmatrix}&= PA_\nn P^T
	\end{align*}
	
\end{proof}

\subsection{Effects on the Symbols}

We have seen how $R_\om,E_\om$ modify the sequences of matrices. Now we focus on how the symbols change though these operators. 
Let us start with the reduction operator $R_\om$.

\begin{lemma}\label{general_reduced_symbol}
	Let $\{A_\nn\}_n$ be a sequence with $A_\nn$ of size $N(\nn)\times N(\nn)$ that is a fixed point for the operator $Z_\om$, and 
	let $k:[0,1]^d\times[-\pi,\pi]^d\to \f C$ be a measurable function 
	with $k(x,\theta)|_{x\not\in \om} = 0$. 
	If $\ms A\sim_\sigma k$, then
	\[
	R_\om(\{A_\nn\}_n)\sim_\sigma k(x,\theta)|_{x\in \om}.
	\] 
	If $\ms A\sim_\lambda k$, then
	\[
	R_\om(\{A_\nn\}_n)\sim_\lambda k(x,\theta)|_{x\in \om}.
	\] 
\end{lemma}
\begin{proof}
	Suppose that $\ms A\sim_\sigma k$. Consider any continuous function $F:\f R\to \f C$ with compact support, and call $S^\om_\nn=R_\om( A_\nn)$. By hypothesis $ A_\nn=Z_\om(A_\nn)$, so we can use \autoref{ker_reduction}  and obtain
	\begin{align*} 
	\frac{1}{d_{n}^\om} \sum_{i=1}^{d_n^\om} F(\sigma_i(S^\om_\nn))	 
	&= 	 \frac{N(\nn)}{d_{n}^\om}\frac 1{N(\nn)} \sum_{i=1}^{N(\nn)} F(\sigma_i( A_\nn))
	- 
	\frac{N(\nn)-d_n^\om}{d_{n}^\om} F(0) .
	\end{align*}
	Notice that $\ms A\sim_\sigma k(x,\theta) = k(x,\theta)\rchi_\om(x)$, so  \autoref{dimension_chi} shows that 
	\begin{align*}
	\lim_{n\to\infty}\frac{1}{d_{n}^\om} \sum_{i=1}^{d_n^\om} F(\sigma_i(S^\om_\nn)) &=\lim_{n\to\infty}	 \frac{N(\nn)}{d_{n}^\om}\frac 1{N(\nn)} \sum_{i=1}^{N(\nn)} F(\sigma_i(A_\nn))
	- \lim_{n\to\infty}
	\frac{N(\nn)-d_n^\om}{d_{n}^\om} F(0)\\
	&=\frac 1{\mu(\om)}
	\frac{1}{(2\pi)^d}\int_{[0,1]^d\times [-\pi,\pi]^d} F(|k(x,\theta)|) d(x,\theta)
	- \frac {1-\mu(\om)}{\mu(\om)} F(0) \\
	&= \frac 1{\mu(\om)(2\pi)^d}
	\int_{\om\times [-\pi,\pi]^d} F(|k(x,\theta)|) d(x,\theta)
	+
	\frac {\mu(\om^C)}{\mu(\om)}F(0)
	- \frac {1-\mu(\om)}{\mu(\om)}F(0)\\
	&= 
	\frac 1{\mu(\om\times [-\pi,\pi]^d)}
	\int_{\om\times [-\pi,\pi]^d} F(|k(x,\theta)|) d(x,\theta).
	\end{align*}
	The last formula holds for every continuous function $F$ with compact support, so
	\[
	R_\om(\{A_\nn\}_n)=\ms {S^\om}\sim_\sigma k(x,\theta)|_{x\in \om}.
	\] 
	If we suppose $\ms A\sim_\lambda k$, the proof is analogous. Consider any continuous and compact supported function $F:\f C\to \f C$ and use \autoref{ker_reduction} to show that 
	\begin{align*}
	\frac{1}{d_{n}^\om} \sum_{i=1}^{d_n^\om} F(\lambda_i(S^\om_\nn ))
	&= 	 \frac{N(\nn)}{d_{n}^\om}\frac 1{N(\nn)} \sum_{i=1}^{N(\nn)} F(\lambda_i( A_\nn)))
	- 
	\frac{N(\nn)-d_n^\om}{d_{n}^\om} F(0),
	\end{align*}
	and exploiting  $\ms B\sim_\lambda k(x,\theta) = k(x,\theta)\rchi_\om(x)$ and   \autoref{dimension_chi}, we conclude that
	\begin{align*}
	\lim_{n\to\infty}\frac{1}{d_{n}^\om} \sum_{i=1}^{d_n^\om} F(\lambda_i(S^\om_\nn )) &=\lim_{n\to\infty}	 \frac{N(\nn)}{d_{n}^\om}\frac 1{N(\nn)} \sum_{i=1}^{N(\nn)} F(\lambda_i( A_\nn))
	- \lim_{n\to\infty}
	\frac{N(\nn)-d_n^\om}{d_{n}^\om} F(0)\\
	&=\frac 1{\mu(\om)}
	\frac{1}{(2\pi)^d}\int_{[0,1]^d\times [-\pi,\pi]^d} F(k(x,\theta)) d(x,\theta)
	- \frac {1-\mu(\om)}{\mu(\om)} F(0) \\
	&= \frac 1{\mu(\om)(2\pi)^d}
	\int_{\om\times [-\pi,\pi]^d} F(k(x,\theta)) d(x,\theta)
	+
	\frac {\mu(\om^C)}{\mu(\om)}F(0)
	- \frac {1-\mu(\om)}{\mu(\om)}F(0)\\
	&= 
	\frac 1{\mu(\om\times [-\pi,\pi]^d)}
	\int_{\om\times [-\pi,\pi]^d} F(k(x,\theta)) d(x,\theta).
	\end{align*}
	The last formula holds for every continuous function $F$ with compact support, so
	\[
	R_\om(\{A_\nn\}_n)=\ms {S^\om}\sim_\lambda k(x,\theta)|_{x\in \om}.
	\] 
\end{proof}

On the contrary let us analyse the effects of the extension operator $E_\om$.

\begin{lemma}\label{general_extended_symbol}
	Let $\{S^\om_\nn\}_n$ be a sequence with $S^\om_\nn$ of size $d_n^\om\times d_n^\om$, let $\kappa:\om\times[-\pi,\pi]^d\to \f C$ be a measurable function, and define
	\[
	\kappa'(x,\theta) = 
	\begin{cases}
	\kappa(x,\theta) & x\in \om,\\
	0 & x\in [0,1]^d\setminus\om.
	\end{cases}
	\]
	If $\ms {S^\om}\sim_\sigma k$, then
	\[
	E_\om(\{S^\om_\nn\}_n)\sim_\sigma \kappa'(x,\theta).
	\] 
	If $\ms {S^\om}\sim_\lambda \kappa$, then
	\[
	E_\om(\{S^\om_\nn\}_n)\sim_\lambda \kappa'(x,\theta).
	\] 
\end{lemma}
\begin{proof}
	Suppose that $\ms {S^\om}\sim_\sigma \kappa$, and denote $\ms A=E_\om(\ms {S^\om})$. If we consider any continuous function $F:\f R\to \f C$ with compact support, then we can use  \autoref{ker_expansion} on $\ms {S^\om}$ to obtain
	\begin{align*} 
	\frac{1}{N(\nn)} \sum_{i=1}^{N(\nn)} F(\sigma_i(A_\nn))	 
	&= 	 \frac{d_{n}^\om}{N(\nn)}\frac 1{d_{n}^\om} \sum_{i=1}^{d_{n}^\om} F(\sigma_i( S^\om_\nn))
	+ 
	\frac{N(\nn)-d_n^\om}{N(\nn)} F(0) .
	\end{align*}
	As a consequence of  \autoref{dimension_chi}, we can show that 
	\begin{align*}
	\lim_{n\to\infty}\frac{1}{N(\nn)} \sum_{i=1}^{N(\nn)} F(\sigma_i(A_\nn)) &=\lim_{n\to\infty}	 \frac{d_{n}^\om}{N(\nn)}\frac 1{d_{n}^\om} \sum_{i=1}^{d_{n}^\om} F(\sigma_i( S^\om_\nn))
	+ \lim_{n\to\infty}	
	\frac{N(\nn)-d_n^\om}{N(\nn)} F(0)\\
	&=\mu(\om)
	\frac{1}{\mu(\om\times[-\pi,\pi]^d)}\int_{\om\times [-\pi,\pi]^d} F(|\kappa(x,\theta)|) d(x,\theta)
	+ (1-\mu(\om)) F(0) \\
	&= \frac 1{(2\pi)^d}
	\int_{[0,1]^d\times [-\pi,\pi]^d} F(|\kappa'(x,\theta)|) d(x,\theta)
	-
	\frac {\mu(\om^C\times[-\pi,\pi]^d)}{(2\pi)^d}F(0)
	+ (1-\mu(\om)) F(0)\\
	&= 
	\frac 1{\mu([0,1]^d\times [-\pi,\pi]^d)}
	\int_{\om\times [-\pi,\pi]^d} F(|\kappa'(x,\theta)|) d(x,\theta).
	\end{align*}
	The last formula holds for every continuous function $F$ with compact support, so
	\[
	E_\om(\{S^\om_\nn\}_n)=\ms A\sim_\sigma \kappa'(x,\theta).
	\] 
	If we suppose $\ms {S^\om}\sim_\lambda \kappa$, the proof is analogous.  If we consider any continuous function $F:\f C\to \f C$ with compact support, then we can use  \autoref{ker_expansion} on $\ms {S^\om}$ to obtain
	\begin{align*} 
	\frac{1}{N(\nn)} \sum_{i=1}^{N(\nn)} F(\lambda_i(A_\nn))	 
	&= 	 \frac{d_{n}^\om}{N(\nn)}\frac 1{d_{n}^\om} \sum_{i=1}^{d_{n}^\om} F(\lambda_i( S^\om_\nn))
	+ 
	\frac{N(\nn)-d_n^\om}{N(\nn)} F(0) .
	\end{align*}
	As a consequence of  \autoref{dimension_chi}, we can show that 
	\begin{align*}
	\lim_{n\to\infty}\frac{1}{N(\nn)} \sum_{i=1}^{N(\nn)} F(\lambda_i(A_\nn)) &=\lim_{n\to\infty}	 \frac{d_{n}^\om}{N(\nn)}\frac 1{d_{n}^\om} \sum_{i=1}^{d_{n}^\om} F(\lambda_i( S^\om_\nn))
	+ \lim_{n\to\infty}	
	\frac{N(\nn)-d_n^\om}{N(\nn)} F(0)\\
	&=\mu(\om)
	\frac{1}{\mu(\om\times[-\pi,\pi]^d)}\int_{\om\times [-\pi,\pi]^d} F(\kappa(x,\theta)) d(x,\theta)
	+ (1-\mu(\om)) F(0) \\
	&= \frac 1{(2\pi)^d}
	\int_{[0,1]^d\times [-\pi,\pi]^d} F(\kappa'(x,\theta)) d(x,\theta)
	-
	\frac {\mu(\om^C\times[-\pi,\pi]^d)}{(2\pi)^d}F(0)
	+ (1-\mu(\om)) F(0)\\
	&= 
	\frac 1{\mu([0,1]^d\times [-\pi,\pi]^d)}
	\int_{\om\times [-\pi,\pi]^d} F(\kappa'(x,\theta)) d(x,\theta).
	\end{align*}
	The last formula holds for every continuous function $F$ with compact support, so
	\[
	E_\om(\{S^\om_\nn\}_n)=\ms A\sim_\lambda \kappa'(x,\theta).
	\] 
\end{proof}

\subsection{Effects on the Convergence}

When dealing with the space of matrix sequences, we already know that it is a complete pseudometric space, equipped with the a.c.s. convergence and the distance
\[
\dacs{\ms{A}}{\ms{B}} = \limsup_{n\to \infty} p(A_\nn-B_\nn), \qquad
p(C_\nn)= \min_{i=1,\dots,N(\nn)+1}\left\{ \frac{i-1}{N(\nn)} + \sigma_i(C_\nn) \right\}.
\]
The operators $R_\om$ and $E_\om$ link two different matrix sequence spaces, so we can analyse how they affect the metrics and the convergences. 

\begin{lemma}\label{zero_distributed_RE}
	Given a sequence $\ms {S^\om}$ with $S^\om_\nn$ of size $d_n^\om\times d_n^\om$ and a sequence $\ms A$ with $A_\nn$ of size $\snnn$, we have
	\[
	\ms {S^\om}\sim_\sigma 0\implies E_\om(\ms {S^\om})\sim_\sigma 0,\qquad \ms {S^\om}\sim_\lambda 0\implies E_\om(\ms {S^\om})\sim_\lambda 0,
	\]
	\[
	\ms A\sim_\sigma 0\implies R_\om(\ms A)\sim_\sigma 0,\qquad \ms A\sim_\lambda 0\implies R_\om(\ms A)\sim_\lambda 0.
	\]
\end{lemma}
\begin{proof}
	Easy corollary of \autoref{general_extended_symbol} and \autoref{general_reduced_symbol}.
\end{proof}

\begin{lemma}\label{metric_reduction}
	Given two sequences $\ms A$ and $\ms B$ with matrices of size $\snnn$,
	\[
	\dacs{\ms A}{\ms B} \ge \mu(\om) \dacs{R_\om(\ms A)}{R_\om(\ms B)}.
	\]
	In particular, 
	\[
	\{B_{\nn,m}\}_n \acs\ms A \implies R_\om(\{B_{\nn,m}\}_n) \acs R_\om(\ms A).
	\]
\end{lemma}
\begin{proof}
	Let $P$ be the permutation matrix in Corollary \autoref{R_principal_submatrix}, so that 
	\[
	P(A_\nn - B_\nn)P^T = \begin{pmatrix}
	R_\om(A_\nn-B_\nn) & *\\
	* &*
	\end{pmatrix}.
	\]
	Using the Cauchy interlacing theorem 
	for singular values, we get that $\sigma_i(R_\om(A_\nn-B_\nn))\le \sigma_i(P(A_\nn-B_\nn)P^T) =\sigma_i(A_\nn-B_\nn) $ for every $1\le i\le d_n^\om$. We can thus use the definition of $d_{acs}$ and  \autoref{dimension_chi} to obtain
	\begin{align*}
	\dacs{\ms A}{\ms B} &= \limsup_{n\to \infty} \min_{i=1,\dots,N(\nn)+1}\left\{ \frac{i-1}{N(\nn)} + \sigma_i(A_\nn-B_\nn) \right\}\\
	&= 
	\limsup_{n\to \infty} \min_{i=1,\dots,N(\nn)+1}\left\{ \frac{i-1}{d_n^\om}\frac{d_n^\om}{N(\nn)} + \sigma_i(A_\nn-B_\nn)\right\}\\
	&\ge \limsup_{n\to \infty}\frac{d_n^\om}{N(\nn)} \min_{i=1,\dots,N(\nn)+1}\left\{ \frac{i-1}{d_n^\om} + \sigma_i(A_\nn-B_\nn)\right\}\\
	&\ge \mu(\om)\limsup_{n\to \infty}\min\left\{
	\min_{i=1,\dots,d_n^\om}\left\{ \frac{i-1}{d_n^\om} + \sigma_i(A_\nn-B_\nn)\right\},
	1
	\right\}\\
	&\ge \mu(\om)\limsup_{n\to \infty}\min\left\{
	\min_{i=1,\dots,d_n^\om}\left\{ \frac{i-1}{d_n^\om} + \sigma_i(R_\om(A_\nn-B_\nn))\right\},
	1
	\right\}\\
	&= 	\mu(\om)\limsup_{n\to \infty}
	\min_{i=1,\dots,d_n^\om+1}\left\{ \frac{i-1}{d_n^\om} + \sigma_i(R_\om(A_\nn-B_\nn))\right\}\\
	&=\mu(\om) \dacs{R_\om(\ms A)}{R_\om(\ms B)}.
	\end{align*}
	Consequentially,
	\begin{align*}
	\{B_{\nn,m}\}_n \acs\ms A &\iff  \dacs{\{B_{\nn,m}\}_n}{\ms A }\to 0\\
	&\implies \dacs{R_\om(\{B_{\nn,m}\}_n)}{R_\om(\ms A) }\to 0\\
	&\iff
	R_\om(\{B_{\nn,m}\}_n) \acs R_\om(\ms A).
	\end{align*}
\end{proof}

\begin{lemma}\label{metric_extension}
	Given two sequences $\mso A$ and $\mso B$ with matrices of size $\sdno$,
	\[
	\dacs{\mso A}{\mso B} \ge \dacs{E_\om(\mso A)}{E_\om(\mso B)}\ge  \mu(\om)\dacs{\mso A}{\mso B}.
	\]
	In particular, 
	\[
	\{B^\om_{\nn,m}\}_n \acs\mso A \iff  E_\om(\{B^\om_{\nn,m}\}_n) \acs E_\om(\mso A).
	\]
\end{lemma}
\begin{proof}
	Thanks to \autoref{MDREZ_relations}, we know that $R_\om(E_\om(\mso A)) =\mso A$, and the same happens to $\mso B$, so we can apply \autoref{metric_extension} and obtain
	\[
	\dacs{E_\om(\mso A)}{E_\om(\mso B)}\ge \mu(\om) \dacs{\mso A}{\mso B}.
	\]
	On the other hand, since $Z_\om(E_\om(\mso A-\mso B)) = E_\om(\mso A-\mso B)$, \autoref{ker_expansion} assures us that the singular values of $\mso A-\mso B$ are the same of the singular values of $E_\om(\mso A-\mso B)$ except $N(\nn)-d_n^\om$ for zeros. It means that
	\begin{align*}
	\dacs{E_\om(\mso A)}{E_\om(\mso B)} &= \limsup_{n\to \infty} \min_{i=1,\dots,N(\nn)+1}\left\{ \frac{i-1}{N(\nn)} + \sigma_i(E_\om(A^\om_\nn-B^\om_\nn)) \right\}\\
	&= \limsup_{n\to \infty} \min_{i=1,\dots,d_n^\om+1}\left\{ \frac{i-1}{N(\nn)} + \sigma_i(E_\om(A^\om_\nn-B^\om_\nn)) \right\}
	\\
	&= \limsup_{n\to \infty} \min_{i=1,\dots,d_n^\om+1}\left\{ \frac{i-1}{N(\nn)} + \sigma_i(A^\om_\nn-B^\om_\nn) \right\}
	\\
	&\le  \limsup_{n\to \infty} \min_{i=1,\dots,d_n^\om+1}\left\{ \frac{i-1}{d_n^\om} + \sigma_i(A^\om_\nn-B^\om_\nn) \right\}
	\\
	&= \dacs{\mso A}{\mso B}.
	\end{align*}
	Consequentially,
	\begin{align*}
	\{B^\om_{\nn,m}\}_n \acs\mso A &\iff  \dacs{\{B^\om_{\nn,m}\}_n}{\mso A }\to 0\\
	&\iff \dacs{E_\om(\{B^\om_{\nn,m}\}_n)}{E_\om(\mso A) }\to 0\\
	&\iff
	E_\om(\{B^\om_{\nn,m}\}_n) \acs E_\om(\mso A).
	\end{align*}
\end{proof}


\subsection{Different Grids}

The operators $Z_\om,R_\om,E_\om$ are always referred to a measurable set $\om$ that tells us which rows and columns to add or remove from the matrices depending on the points of the regular grid $\Xi_n$ inside $\om$. Suppose now that we want to choose a slight different set of points for every $n$, and we ask whether the resulting sequence of matrices still enjoys a symbol. 
Remember that the symmetric difference $\triangle$ between two sets is the set of elements belonging to only one of the two sets. In symbols, $A\triangle B = (A\setminus B)\cup (B\setminus A)$. 

\begin{lemma}\label{Different_Grids}
	Let $\Gamma_n$ be a measurable set  in $[0,1]^d$  (not necessarily Peano-Jordan measurable) and let $\om$ be a  Peano-Jordan measurable set with positive measure  in $[0,1]^d$. Suppose that
	\[
	d_n^{\om\triangle\Gamma_n  }  = o(N(\nn)).
	\]
	Given a sequence $\{A_\nn\}_n$ with $A_\nn$ of size $\snnn$, and a measurable function $k$, we have that
	\[
	R_{\om}(\{A_\nn\}_n)\sim_\sigma k \iff \{R_{\Gamma_n}(A_\nn)\}_n\sim_\sigma k.
	\] 
	Moreover, if $A_\nn$ are Hermitian, then
	\[
	R_{\om}(\{A_\nn\}_n)\sim_\lambda k \iff \{R_{\Gamma_n}(A_\nn)\}_n\sim_\lambda k.
	\] 
\end{lemma}
\begin{proof}
	Consider 
	the difference \[
	R_{\om\cup \Gamma_n}(Z_{\om \cap \Gamma_n} (A_\nn)) - R_{\om\cup \Gamma_n}(E_{\om} (R_{\om}(A_\nn))) = 
	R_{\om\cup \Gamma_n}(Z_{\om \cap \Gamma_n} (A_\nn)-Z_{\om}(A_\nn))
	.
	\]
	The matrix has at most $d_n^{\om\setminus \Gamma_n}\le d_n^{\om\triangle\Gamma_n  }  = o(N(\nn))$ non-zero rows and columns, and from  \autoref{dimension_chi}, we infer also that $d_n^{\om\setminus \Gamma_n} = o(d_n^\om)$. Consequently, $d_n^{\om\setminus \Gamma_n} = o(d_n^{\om\cup\Gamma_n})$, so the sequence is zero-distributed. Moreover, the matrix $B^{\om\cup \Gamma_n}_\nn := R_{\om\cup \Gamma_n}(E_{\om} (R_{\om}(A_\nn)))$ is actually $R_{\om}(A_\nn)$ with additional $d_n^{\Gamma_n\setminus \om}\le d_n^{\om\triangle\Gamma_n  }  = o(N(\nn))$ zero columns and rows, so we just added few zero singular values, for which holds again $d_n^{ \Gamma_n\setminus\om} = o(d_n^{\om\cup\Gamma_n})$. In particular,
	if we consider any continuous function $F:\f C\to \f C$ with compact support, then
	\begin{align*}
	\frac{1}{d_n^{\Gamma_n\cup \om}} \sum_{i=1}^{d_n^{\Gamma_n\cup \om}} F(\sigma_i(B^{\om\cup \Gamma_n}_\nn))	&=
	\frac{d_n^{\Gamma_n\setminus \om}}{d_n^{\Gamma_n\cup \om}}  F(0)   +  \frac{d_n^{\Gamma_n\cup \om} - d_n^{\Gamma_n\setminus \om}}{d_n^{\Gamma_n\cup \om}} \frac{1}{d_n^\om}\sum_{i=1}^{d_n^{\om}} F(\sigma_i(R_{\om}(A_\nn)))
	\end{align*}
	and asymptotically we have
	\begin{align*}
	\lim_{n\to \infty}\frac{1}{d_n^{\Gamma_n\cup \om}} \sum_{i=1}^{d_n^{\Gamma_n\cup \om}} F(\sigma_i(B^{\om\cup \Gamma_n}_\nn))	&=
	\lim_{n\to \infty} \frac{1}{d_n^\om}\sum_{i=1}^{d_n^{\om}} F(\sigma_i(R_{\om}(A_\nn))).
	\end{align*}
	It leads to 
	\[
	R_{\om}(\{A_\nn\}_n)\sim_\sigma k \iff
	\{R_{\om\cup \Gamma_n}(E_{\om} (R_{\om}(A_\nn)))\}_n\sim_\sigma k \iff
	\{R_{\om\cup \Gamma_n}(Z_{\om \cap \Gamma_n} (A_\nn))\}_n \sim_\sigma k
	\]
	and the same argument can be applied to $R_{\Gamma_n}(A_\nn)$, so we can conclude that 
	\[
	R_{\om}(\{A_\nn\}_n)\sim_\sigma k \iff \{R_{\Gamma_n}(A_\nn)\}_n\sim_\sigma k.
	\] 
	If $A_\nn$ are hermitian, then all the matrices considered until now are also Hermitian, so the same results apply to the spectral symbols and 
	\[
	R_{\om}(\{A_\nn\}_n)\sim_\lambda k \iff \{R_{\Gamma_n}(A_\nn)\}_n\sim_\lambda k.
	\] 
\end{proof}

This result is quite powerful since it tells us that we can add and remove a small number of rows and columns without changing the symbol of the sequence. It will be useful in applications when dealing with near-boundary conditions.

\section{Reduced GLT}\label{RedGLT}

In the following propositions, we denote the image of $R_\om$ when applied to GLT sequences as $\mc G^\om_d := R_\om(\mc G_d)$, and we call it the space of \textit{Reduced GLT} with respect to $\om$.
\subsection{Reduced GLT Symbol}

\begin{lemma}\label{reduced_symbol}
	Given a GLT sequence $\{A_\nn\}_n\GLT k(x,\theta)$ with $k:[0,1]^d\times[-\pi,\pi]^d\to \f C$, then 
	\[
	R_\om(\{A_\nn\}_n)\sim_\sigma k(x,\theta)|_{x\in \om}.
	\] 
	If $A_\nn$ are also Hermitian matrices, then
	\[
	R_\om(\{A_\nn\}_n)\sim_\lambda k(x,\theta)|_{x\in \om}.
	\] 
\end{lemma}
\begin{proof}
	Thanks to \autoref{MDREZ_relations}, we have 
	$
	R_\om(\ms A) = R_\om(Z_\om(\ms A))
	$
	and if we call $\ms B =Z_\om(\ms A)$, then $Z_\om(\ms B)= \ms B$ since $Z_\om$ is an idempotent operator. Moreover,  $\ms B\GLT k(x,\theta)\rchi_\om(x)$, so in particular $\ms B\sim_\sigma k(x,\theta)\rchi_\om(x)$ due to \textbf{GLT1}. We can thus use \autoref{general_reduced_symbol} and obtain that 
	\[
	R_\om(\ms A) =R_\om(\ms B)\sim_\sigma  k(x,\theta)\rchi_\om(x)|_{x\in\om}=k(x,\theta)|_{x\in\om}. 
	\]
	If $A_\nn$ are Hermitian matrices, then also $\ms B=Z_\om(\ms A)$ is a Hermitian sequence, since $Z_\om$ preserves the Hermitianity, so $\ms B\GLT k(x,\theta)\rchi_\om(x)\implies \ms B\sim_\lambda k(x,\theta)\rchi_\om(x)$ due to \textbf{GLT1}. As before, $Z_\om(\ms B)= \ms B$ and \autoref{general_reduced_symbol} assure us that 
	\[
	R_\om(\ms A) =R_\om(\ms B)\sim_\lambda  k(x,\theta)\rchi_\om(x)|_{x\in\om}=k(x,\theta)|_{x\in\om}. 
	\]
\end{proof}

Notice that the map $R_\om$ is not injective, but one can prove that all the GLT sequences with the same image have  symbols that coincide on $\om\times [-\pi,\pi]$.

\begin{lemma}\label{uniqueness}
	Given $\ms A\GLT k$, $\ms B\GLT h$ such that $R_\om(\ms A)=R_\om(\ms B)= \ms {S^\om}\in \mc G^\om_d$, the symbols $k,h$ coincide on $\om\times [-\pi,\pi]^d$. 
\end{lemma}
\begin{proof}
	Since $R_\om$ is linear, we can use \autoref{reduced_symbol} and \textbf{GLT4} and obtain
	\[
	\ms A -\ms B\GLT k-h \implies 
	R_\om(\ms A -\ms B) = \ms {S^\om}-\ms {S^\om} =\mso 0
	\sim_\sigma (k-h)|_{x\in \om}=\kappa.
	\]
	Notice that if the set where $0<L<|\kappa|<M$ has non-zero measure, then we can consider a nonnegative continuous function $F:\f R\to \f C$ with compact support such that $F(0)=0$ and $F(x)>\delta>0$ for every $x\in (L,M)$ to get an absurd
	\[
	0=\lim_{n\to \infty}\frac{1}{d_n^\om}F(\sigma_i(0_\nn)) = \frac{1}{\mu(\om)(2\pi)^d}
	\int_{\om\times[-\pi,\pi]^d} F(|\kappa(x,\theta)|) d(x,\theta) \ge \frac{1}{\mu(\om)(2\pi)^d}\delta \mu\{|\kappa|>0\}>0.
	\]
	We conclude that $\kappa=0$, and so $k,h$ coincide on $\om\times [-\pi,\pi]^d$.
\end{proof}
As a corollary,  every GLT sequence mapped into $\ms {S^\om}$ possesses a symbol with a fixed value on $\om\times [-\pi,\pi]^d$, so
we can associate to each reduced GLT sequence $\ms {S^\om}$ an unique symbol, called \textit{Reduced GLT Symbol}, obtained as the restriction of any GLT symbol of the sequences in the counter-image $R_\om^{-1}(\ms {S^\om})\cap \mc G_d$. From now on, we will use the notation $\ms {S^\om}\GLT^\om s$ to indicate that $s:\om\times[-\pi,\pi]^d\to \f C$ is the restriction of a symbol $k:[0,1]^d\times [-\pi,\pi]^d\to \f C$ such that $\ms A\GLT k$ and $R_\om(\ms A)=\ms {S^\om}$. \\

Given any reduced GLT sequence $\ms {S^\om}$, it is easy to produce a GLT sequence $\ms A$ such that $R_\om(\ms A)=\ms {S^\om}$ using the operator $E_\om$.  We can thus  reverse \autoref{reduced_symbol}. 
\begin{lemma}\label{symbol_extension}
	If  $\ms {S^\om}\GLTom \kappa$, then
	$E_\om(\ms {S^\om})\GLT k(x,\theta)=
	\begin{cases}
	\kappa(x,\theta) & x\in \om,\\
	0 & x\not\in \om.
	\end{cases}$ and $R_\om(E_\om(\ms {S^\om}))=\ms {S^\om}$.
\end{lemma}
\begin{proof}
	Since $\ms {S^\om}\in \mc G_d^\om = R_\om(\mc G_d)$, there exists a GLT  sequence $\ms A\GLT h$ such that $R_\om(\ms A) = \ms {S^\om}$, but thanks to \autoref{MDREZ_relations} we know that also $R_\om(Z_\om(\ms A)) =\ms {S^\om}$ and
	\[
	Z_\om(\ms A)\GLT h(x,\theta)\rchi_\om (x) =k(x,\theta)=
	\begin{cases}
	\kappa(x,\theta) & x\in \om,\\
	0 & x\not\in \om.
	\end{cases}
	\]
	Using again \autoref{MDREZ_relations}, we can conclude, since
	\[
	Z_\om(\ms A) = E_\om(R_\om(\ms A)) = E_\om(\ms {S^\om}).
	\]
\end{proof}

\subsection{Axioms of Reduced GLT}
Using the connection between $\mc G_d$ and $\mc G_d^\om$, we can prove that many properties of the first space transfer to the second. 
\begin{theorem}
	Suppose $\mso A$, $\mso B$ are reduced GLT sequences and $\mso X,\mso Y$ are sequences with $X^\om_\nn,Y^\om_\nn\in\f C^{d_n^\om\times d_n^\om}$. 
	\begin{enumerate} [leftmargin=39pt]
		\item[\textbf{GLT\,1.}] If $\mso A\GLTom\kappa$ then $\mso A\sim_{\sigma}\kappa$. If $\mso A\GLTom\kappa$ and each $A^\om_\nn$ is Hermitian then $\mso A\sim_{\lambda}\kappa$.
		\item[\textbf{GLT\,2.}] If $\mso A\GLTom\kappa$ and $\mso A=\mso X+\mso Y$, where
		\begin{itemize}
			\item every $X^\om_\nn$ is Hermitian,
			\item $(d_n^\om)^{-1}\|Y^\om_\nn\|_2^2\to 0$,
		\end{itemize}
		then $\mso A\sim_{\lambda}\kappa$.
		\item[\textbf{GLT\,3.}] Here we list three important examples of reduced GLT sequences. 
		\begin{itemize}
			\item Given a function $f$ in $L^1([-\pi,\pi]^d)$, its associated reduced Toeplitz sequence is $\{
			T^\om_\nn(f)\}_n= R_\om (\{T_\nn(f)\}_n)$, where the elements are multidimensional Fourier coefficients of $f$:
			\[ T_\nn( f ) = [ f_{\ii-\jj} ]^\nn_{\ii, \jj={\bm 1}}, \qquad f_{\bm k} = \frac{1}{(2\pi)^d} \int_{-\pi}^{\pi} f(\theta) e^{-\textnormal i \bm k\cdot\theta} d\theta. \]
			$\{
			T^\om_\nn(f)\}_n$ is a reduced GLT sequence with symbol $\kappa( x,\theta)=f(\theta)$.
			\item Given an almost everywhere continuous function, $\wt a:[0,1]^d\to\mathbb C$ and its restriction $a=\wt a|_{\om}$, its associated diagonal sampling sequence $\{D_\nn^\om(a)\}_n$ is defined as
			\[ D_\nn^\om(a) = \textnormal{diag}\left(\left\{a\left(\frac {\phi(i)}{\nn+\uu}\right) \right\}_{i=1}^{d_n^\om}\right). \]
			$\{D_\nn^\om(a)\}_n$  is a reduced GLT sequence with symbol $\kappa( x,\theta)=a( x)$.
			\item Any zero-distributed sequence $\{Y^\om_\nn\}_n\sim_\sigma0$ is a reduced GLT sequence with symbol $\kappa( x,\theta)=0$.
		\end{itemize}
		\item[\textbf{GLT\,4.}] If $\mso A\GLTom\kappa$ and $\mso B\GLTom\xi$, then
		\begin{itemize}
			\item $\{ (A_\nn^\om)^*\}_n\GLTom\overline\kappa$, where $(A^\om_\nn)^*$ is the conjugate transpose of $A^\om_\nn$,
			\item $\{\alpha A^\om_\nn+\beta B^\om_\nn\}_n\GLTom\alpha\kappa+\beta\xi$ for all $\alpha,\beta\in\mathbb C$,
			\item $\{A^\om_\nn B^\om_\nn\}_n\GLTom\kappa\xi$.
		\end{itemize}
		\item[\textbf{GLT\,5.}] If $\mso A\GLTom\kappa$ and $\kappa\ne0$ a.e., then $\{(A_\nn^\om)^\dag \}_n\GLTom\kappa^{-1}$, where $(A^\om_\nn)^\dag$ is the Moore--Penrose pseudoinverse of $A_\nn^\om$. 
		\item[\textbf{GLT\,6.}] If $\mso A\GLTom\kappa$ and each $A_\nn^\om$ is Hermitian, then $\{f(A^\om_\nn)\}_n\GLTom f(\kappa)$ for all continuous functions $f:\mathbb C\to\mathbb C$.
		\item[\textbf{GLT\,7.}] $\mso A\GLTom\kappa$ if and only if there exist GLT sequences $\{B_{\nn,m}\}_n\GLTom\kappa_m$ such that $\kappa_m$ converges to $\kappa$ in measure and $\{B_{\nn,m}\}_n\xrightarrow{\textnormal{a.c.s.}}\mso A$ as $m\to\infty$.
		\item[\textbf{GLT\,8.}] Suppose $\mso A\GLTom\kappa$ and $\{B^\om_{\nn,m}\}_n\GLTom\kappa_m$, where both $A^\om_\nn$ and $B^\om_{\nn,m}$ have the same size $\sdno$. Then, $\{B^\om_{\nn,m}\}_n\xrightarrow{\textnormal{a.c.s.}}\mso A$ as $m\to\infty$ if and only if $\kappa_m$ converges to $\kappa$ in measure.
		\item[\textbf{GLT\,9.}]  If $\mso A\GLTom\kappa$ 
		then there exist functions  $a_{i,m},f_{i,m},\ i=1,\ldots,N_m$, such~that
		\begin{itemize}
			\item $a_{i,m}\in C^\infty(\om)$ and $f_{i,m}$ is a trigonometric polynomial,
			\item $\sum_{i=1}^{N_m}a_{i,m}(x)f_{i,m}(\theta)$ converges to $\kappa(x,\theta)$ a.e.,
			\item  $\bigl\{\sum_{i=1}^{N_m}D_\nn^\om(a_{i,m})T_{\nn}^\om(f_{i,m})\bigr\}_n\xrightarrow{\rm a.c.s.}\mso A$ as $m\to\infty$. 
		\end{itemize}
	\end{enumerate}
\end{theorem}
\begin{proof}
	Given $\mso A\GLTom \kappa$, $\mso B\GLTom \xi$,
	call $\ms {A}=E_\om(\mso A)\GLT \kappa'$ and $\ms {B}=E_\om(\mso B)\GLT \xi'$, where $\kappa'$ and $\xi'$ are the extension of $\kappa$ and $\xi$ as specified in \autoref{symbol_extension}. We know that $\kappa'|_{x\in \om} =\kappa$, $\xi'|_{x\in \om} =\xi$ and $R_\om(\ms{A})=\mso A$, $R_\om(\ms{B})=\mso B$.
	Notice that in every proof we use the axioms \textbf{GLT1-9} referred to the regular multilevel GLT.
	\begin{enumerate} [leftmargin=39pt]
		\item[\textbf{GLT\,1.}] 
		Using \autoref{reduced_symbol}, we know that $\mso A\sim_{\sigma}\kappa'|_{x\in \om} =\kappa$. If  $\mso A$ is Hermitian, then $\ms {A}$ is Hermitian by \autoref{MDREZ_relations}, so \autoref{reduced_symbol} let us conclude that $\mso A\sim_{\lambda}\kappa'|_{x\in \om} =\kappa$.
		\item[\textbf{GLT\,2.}] Let $\ms {X}=E_\om(\mso X)$ and $\ms {Y}=E_\om(\mso Y)$. The operator $E_\om$ is linear, so $\ms{ A} = \ms{X}+\ms {Y}$, where $\ms {A}\GLT \kappa'$. Using 
		Corollary \autoref{ker_expansion}, we know that the singular values of $Y_\nn$ are the same of $Y^\om_\nn$ except for zero singular values. As a consequence,
		\[ 
		\lim_{n\to+\infty}(N(\nn))^{-1}\|Y_\nn\|_2^2 = 
		\lim_{n\to+\infty}\frac{d_n^\om}{N(\nn)}(d_n^\om)^{-1}\|Y^\om_\nn\|_2^2 = \mu(\om)\cdot 0=0.
		\]
		We can thus assert that $\ms{A}\sim_\lambda \kappa'$, but $\kappa'|_{x\not\in\om} = 0$ and $R_\om(\ms{A})=\mso{A}$, so we can apply \autoref{general_reduced_symbol} and conclude that
		\[
		\mso A=R_\om(\ms {A})\sim_\lambda \kappa'|_{x\in \om} = \kappa.
		\] 
		\item[\textbf{GLT\,3.}] 
		We know that $\{T_\nn(f)\}_n\GLT f$, so \autoref{reduced_symbol} assures us that
		\[
		\{
		T^\om_\nn(f)\}_n= R_\om (\{T_\nn(f)\}_n)\GLTom f(\theta).\]
		Analogously, \autoref{preliminaries_equivalence_D_nn_a} shows that $\{I_\nn(\wt a)  \}_n\GLT \wt a$ and it is easy to check that $\{D_\nn^\om(a) \}_n = R_\om(\{I_\nn(\wt a)  \}_n)$, so 
		\[
		\{D_\nn^\om(a) \}_n \GLTom a.
		\]
		Moreover,  \autoref{zero_distributed_RE}, shows that 
		\[
		\mso Y\sim_\sigma 0
		\implies E_\om(\mso Y)\sim_\sigma 0
		\implies  E_\om(\mso Y)\GLT 0
		\implies \mso Y = R_\om(E_\om(\mso Y))\GLTom 0.
		\]
		\item[\textbf{GLT\,4.}] 
		Using \autoref{MDREZ_relations} and \autoref{reduced_symbol}, we know that
		\[
		\{(A_\nn^\om)^*\}_n = (R_\om(\ms {A}))^*= R_\om(\ms {A^*})\GLTom \ol \kappa'|_{x\in \om} =\ol \kappa.
		\]
		Moreover, $R_\om$ is linear, so we can apply \autoref{reduced_symbol} on $\alpha\ms {A} +\beta\ms{B} \GLT \alpha \kappa'+\beta\xi'$ and obtain
		\[
		\{\alpha A^\om_\nn+\beta B^\om_\nn\}_n = R_\om(\alpha\ms {A} +\beta\ms{B}) \GLTom \alpha \kappa'+\beta\xi'|_{x\in \om} = \alpha\kappa+\beta\xi.
		\]
		In order to prove the last point, remember that $Z_\om(A_\nn)=A_\nn$, so we can use \autoref{MDREZ_relations} and obtain the relation
		\begin{align*}
		R_\om(A_{\nn}B_\nn) &=
		\Pi_{\nn,\om} A_{\nn}B_\nn (\Pi_{\nn,\om})^T \\
		&=
		\Pi_{\nn,\om} D_\nn(\rchi_\om) A_{\nn}D_\nn(\rchi_\om) B_\nn (\Pi_{\nn,\om})^T  \\
		&=
		\Pi_{\nn,\om} D_\nn(\rchi_\om) A_{\nn}(D_\nn(\rchi_\om))^2 B_\nn (\Pi_{\nn,\om})^T\\
		&=
		\Pi_{\nn,\om} A_{\nn}D_\nn(\rchi_\om) B_\nn (\Pi_{\nn,\om})^T\\
		&= \Pi_{\nn,\om} A_{\nn}(\Pi_{\nn,\om})^T\Pi_{\nn,\om} B_\nn (\Pi_{\nn,\om})^T\\
		&= R_\om(A_{\nn})R_\om(B_\nn).
		\end{align*}
		Using \autoref{reduced_symbol}, we conclude that
		\[
		\mso A\mso B = R_\om(\ms {A})R_\om(\ms{B}) = 
		R_\om(\ms {A}\ms{B}) \GLTom \kappa'\xi'|_{x\in \om} = \kappa\xi.
		\]
		\item[\textbf{GLT\,5.}]
		Notice that $\partial\om=\partial(\om^C)$, so $\{D_\nn(\rchi_{\om^C}) \}_n\GLT \rchi_{\om^C}$. 
		If we define $\ms C = \ms{A} + \{D_\nn(\rchi_{\om^C}) \}_n$, then 
		\[
		\ms C\GLT \kappa'(x,\theta) +  \rchi_{\om^C}(x)= 
		\begin{cases}
		\kappa & x\in \om,\\
		1 & x	\not\in\om,
		\end{cases} 
		\]
		so $\kappa'(x,\theta) +  \rchi_{\om^C}(x) = 0$ if and only if $x\in \om$ and $\kappa(x,\theta) =0$. In particular it is different from zero a.e., so 
		\[
		\ms {C^\dag}\GLT (\kappa'(x,\theta) +  \rchi_{\om^C}(x))^{-1}= 
		\begin{cases}
		\kappa^{-1} & x\in \om,\\
		1 & x	\not\in\om.
		\end{cases} 
		\]
		We know that $Z_\om(\ms{A})=\ms{A}$ and
		using \autoref{MDREZ_relations},
		\[
		Z_\om (D_\nn(\rchi_{\om})) = D_\nn(\rchi_{\om})D_\nn(\rchi_{\om})D_\nn(\rchi_{\om})=D_\nn(\rchi_{\om}),\]\[
		R_\om(D_\nn(\rchi_{\om})) = \Pi_{\nn,\om} D_\nn(\rchi_{\om})(\Pi_{\nn,\om})^T=
		\Pi_{\nn,\om}(\Pi_{\nn,\om})^T\Pi_{\nn,\om}(\Pi_{\nn,\om})^T=
		I_\nn^\om I_\nn^\om    = I_\nn^\om .
		\]
		Let $P$ be the permutation matrix in \autoref{ker_reduction}, so that
		\[
		PA_\nn P^T=
		\begin{pmatrix}
		A_\nn^\om&0\\0&0
		\end{pmatrix},
		\qquad 
		PD_\nn(\rchi_{\om^C})P^T = P(I_\nn - D_\nn(\rchi_{\om}))P^T = I_\nn  - \begin{pmatrix}
		I_\nn^\om &0\\0&0
		\end{pmatrix} =
		\begin{pmatrix}
		0&0\\0&
		I_\nn^{\om^C}
		\end{pmatrix},
		\]
		\[
		PC_\nn P^T = P(A_\nn + D_\nn(\rchi_{\om^C}))P^T= \begin{pmatrix}
		A_\nn^\om&0\\0&I_\nn^{\om^C}
		\end{pmatrix} \implies PC_\nn^\dag P^T =  \begin{pmatrix}
		(A_\nn^\om)^\dag&0\\0&I_\nn^{\om^C}
		\end{pmatrix}.
		\]
		Consequentially,
		\begin{align*}
		\begin{pmatrix}
		R_\om(C^\dag_\nn)&0\\0&0
		\end{pmatrix} &= P Z_\om(C^\dag_\nn)  P^T\\
		&= P D_\nn(\rchi_\om) C_\nn^\dag D_\nn(\rchi_\om) P^T\\
		&= P D_\nn(\rchi_\om) P^T
		\begin{pmatrix}
		(A_\nn^\om)^\dag&0\\0&I_\nn^{\om^C}
		\end{pmatrix}
		P D_\nn(\rchi_\om) P^T\\
		&= \begin{pmatrix}
		I_\nn^\om &0\\0&0
		\end{pmatrix} 
		\begin{pmatrix}
		(A_\nn^\om)^\dag&0\\0&I_\nn^{\om^C}
		\end{pmatrix}
		\begin{pmatrix}
		I_\nn^\om &0\\0&0
		\end{pmatrix} \\
		&= \begin{pmatrix}
		(A_\nn^\om)^\dag&0\\0&0
		\end{pmatrix}
		\end{align*}
		and \autoref{reduced_symbol} let us conclude that 
		\[
		\{(A_\nn^\om)^\dag \}_n =R_\om(\ms {C^\dag})\GLTom (\kappa'(x,\theta) +  \rchi_{\om^C}(x))^{-1}|_{x\in\om} = \kappa^{-1}.
		\]
		\item[\textbf{GLT\,6.}]
		If $A^\om_\nn$ is Hermitian, then $A_\nn=E_\om(A^\om_\nn)$ is also Hermitian and $\ms{A}\GLT \kappa'$, so 
		\[
		\{f(A_\nn)\}_n\GLT f(\kappa')
		=
		\begin{cases}
		f(\kappa(x,\theta)) & x\in \om,\\
		f(0) & x\not\in\om.
		\end{cases}
		\]
		Notice that, using \autoref{ker_reduction},
		\[
		PA_\nn P^T =
		\begin{pmatrix}
		A^\om_\nn&0\\0&0
		\end{pmatrix}
		\implies
		Pf(A_\nn) P^T =f(PA_\nn P^T)=  
		\begin{pmatrix}
		f(A^\om_\nn)&0\\0&f(0)I_\nn^{\om^C}
		\end{pmatrix},
		\]
		so one can prove that
		\begin{align*}
		\begin{pmatrix}
		R_\om(f(A_\nn))&0\\0&0
		\end{pmatrix} &= P Z_\om(f(A_\nn))  P^T\\
		&= P D_\nn(\rchi_\om) f(A_\nn) D_\nn(\rchi_\om) P^T\\
		&= P D_\nn(\rchi_\om) P^T
		\begin{pmatrix}
		f(A^\om_\nn)&0\\0&f(0)I_\nn^{\om^C}
		\end{pmatrix}
		P D_\nn(\rchi_\om) P^T\\
		&= \begin{pmatrix}
		I_\nn^\om &0\\0&0
		\end{pmatrix} 
		\begin{pmatrix}
		f(A^\om_\nn)&0\\0&f(0)I_\nn^{\om^C}
		\end{pmatrix}
		\begin{pmatrix}
		I_\nn^\om &0\\0&0
		\end{pmatrix} \\
		&= \begin{pmatrix}
		f(A^\om_\nn)&0\\0&0
		\end{pmatrix}
		\end{align*}
		and consequentially \autoref{reduced_symbol} let us conclude
		\[
		\{f(A^\om_\nn)\}_n = 	R_\om(\{f(A_\nn)\}_n)\GLTom f(\kappa')|_{x\in \om} = f(\kappa).
		\]
		\item[\textbf{GLT\,7.}]
		Notice that if $\mso A\GLTom\kappa$ and $A^\om_\nn = B^\om_{\nn,m}$ for every $m$, then $\{B^\om_{\nn,m}\}_n\GLTom\kappa_m=\kappa$, $\kappa_m$ converges to $\kappa$ and $\{B^\om_{\nn,m}\}_n\xrightarrow{\textnormal{a.c.s.}}\mso A$. 
		
		On the opposite, assume there exist reduced GLT sequences $\{B^\om_{\nn,m}\}_n\GLTom\kappa_m$ such that $\kappa_m$ converges to $\kappa$ in measure and $\{B^\om_{\nn,m}\}_n\xrightarrow{\textnormal{a.c.s.}}\mso A$. In this case, let $B_{\nn,m}=E_\om(B^\om_{\nn,m})$ and let $\kappa'_m$ be the extension of $\kappa$ given by \autoref{symbol_extension}, so that $\{B_{\nn,m} \}_n\GLT \kappa'_m$. Using \autoref{metric_extension}, we know that  $\{B_{\nn,m}\}_n\xrightarrow{\textnormal{a.c.s.}}E_\om(\mso A)$, and moreover
		\[
		\kappa'_m = \begin{cases}
		\kappa_m(x,\theta) &x\in \om,\\
		0 & x\not\in\om,
		\end{cases}
		\to \kappa' = \begin{cases}
		\kappa(x,\theta) &x\in \om,\\
		0 & x\not\in\om,
		\end{cases}
		\]
		so $E_\om(\mso A)\GLT \kappa'$ and \autoref{reduced_symbol} let us conclude that 
		\[
		R_\om(E_\om(\mso A)) =\mso A \GLTom \kappa'|_{x\in\om} = \kappa.
		\]
		\item[\textbf{GLT\,8.}]
		Let $B_{\nn,m}=E_\om(B^\om_{\nn,m})$ and let $\kappa'_m$ be the extension of $\kappa$ given by \autoref{symbol_extension}, so that $\{B_{\nn,m} \}_n\GLT \kappa'_m$.
		Using \autoref{metric_extension}, we know that
		\[
		\{B^\om_{\nn,m}\}_n\xrightarrow{\textnormal{a.c.s.}}\mso A
		\iff 
		\{B_{\nn,m}\}_n\xrightarrow{\textnormal{a.c.s.}}E_\om(\mso {A})
		\iff
		\kappa'_m\to \kappa'.
		\]
		All the functions $\kappa'_m$ and $\kappa'$ are zero outside $\om$, and $\om$ has positive measure, so
		\[
		\kappa'_m-\kappa' \to 0 \iff \kappa'_m-\kappa'|_{x\in\om} \to 0\iff \kappa_m -\kappa \to 0\iff \kappa_m\to\kappa.
		\]
		\item[\textbf{GLT\,9.}] 
		The functions in $ C^\infty(\om)$ are restrictions of functions in $C^\infty([0,1]^d)$, so, given $\kappa$, we can consider $E_\om(\mso A)\GLT\kappa'$ and find $a'_{i,m}\in C^\infty([0,1]^d)$ and trigonometric polynomials $f_{i,m}$ such that 
		$\sum_{i=1}^{N_m}a'_{i,m}(x)f_{i,m}(\theta)$ converges to $\kappa'(x,\theta)$ a.e., and if $a'_{i,m}|_{x\in \om} = a_{i,m}$, then $\sum_{i=1}^{N_m}a'_{i,m}(x)f_{i,m}(\theta)|_{x\in \om} =\sum_{i=1}^{N_m}a_{i,m}(x)f_{i,m}(\theta) $ converges to $\kappa'|_{x\in \om} =\kappa$ almost everywhere.
		Thanks to \textbf{GLT\,3} we know that $\{D_\nn^\om(a_{i,m})\}_n\GLTom a_{i,m}$ and $\{T_{\nn}^\om(f_{i,m})\}_n\GLTom f_{i,m}$, so we can apply \textbf{GLT 4} and obtain
		\[
		\left\{ \sum_{i=1}^{N_m}D_\nn^\om(a_{i,m})T_{\nn}^\om(f_{i,m})\right\}_n\GLTom \sum_{i=1}^{N_m}a_{i,m}(x)f_{i,m}(\theta)\to \kappa,
		\]
		so that \textbf{GLT\,8} let us conclude that 
		\[ \left\{ \sum_{i=1}^{N_m}D_\nn^\om(a_{i,m})T_{\nn}^\om(f_{i,m})\right\}_n
		\xrightarrow{\rm a.c.s.}\mso A.\]
	\end{enumerate}
\end{proof}

\subsection{Isometry with Measurable Functions}

It has been proved that the space of GLT sequences, up to zero-distributed sequences, is actually isomorphic as an algebra and isometric as a complete pseudometric space to the space of measurable functions on an opportune domain. 
In particular, every measurable function with domain $[0,1]^d\times [-\pi,\pi]^d$ is a GLT symbol for some multilevel GLT sequence. The same can be said for the space of reduced GLT sequences.

Let $\hat {\mc S}^\om$ be the map connecting each reduced GLT sequence with its symbol
\[
\hat{\mc S}^\om:\mc G_d^\om\to\mc M_\om
\]
where $\mc M_\om$ is the space of measurable functions from $\om \times [-\pi,\pi]^d$ to $\f C$, equipped with the metric of the convergence in measure $d_m$. \textbf{GLT 4} assures us that $\hat {\mc S}^\om$ is a linear map, and \textbf{GLT 1,3} identify the kernel as the  set $\mc Z$ of zero-distributed sequences. We can thus define the map  
\[
\mc S^\om:\faktor {\mc G_d^\om}{\mc Z}\to\mc M_\om
\]
and prove it is an isomorphism and an isometry. 
\begin{lemma}
	The map $\mc S^\om$ is an isomorphism of algebras. 
\end{lemma}
\begin{proof}
	By construction, we already know that $\mc S^\om$ is a linear injective map. Given now any $\kappa\in\mc M_\om$, let $\kappa'$ be the extension of $\kappa$ to $[0,1]^d$ obtained by setting $\kappa'=0$ outside $\om$. Let $\ms A\GLT \kappa'$, and notice that $\mc S^\om (R_\om(\ms A))=\kappa$, proving that $S^\om$ is also surjective. 
\end{proof}

\begin{theorem}\label{isometry}
	The map $\mc S^\om$ is an isometry of pseudometric spaces.
\end{theorem}
\begin{proof}
	Let $\ms {S^\om}\GLTom \kappa$ and notice that 
	\[
	\dacs{\ms {S^\om}}{\mso 0} =\rho(\ms {S^\om}) =\limsup_{n\to \infty} \min_{i=1,\dots,d_n^\om+1} \left\{ \frac{i-1}{d_n^\om} + \sigma_i(S^\om_\nn) \right\},
	\]\[
	d_{\rm mea}(\kappa, 0) = \rho_{\rm mea}(\kappa) =\inf_{E\cu \om\times[-\pi,\pi]^d} \left\{ \frac{\mu(E^C)}{\mu(\om)(2\pi)^d} + \ess\sup_E |\kappa|  \right\}.
	\]
	Call $L := \rho_{\rm mea}(\kappa)$. 
	By the definition of the infimum, if we set $\ve>0$, we can always find a measurable set  $H$ such that 
	\[
	\frac{\mu(H^C)}{\mu(\om)(2\pi)^d}+ \ess\sup_H |\kappa| \le L + \ve.
	\]
	From now on, let us call $M=\ess\sup_H |\kappa|$.
	Let $F:\f R\to \f R$  be a continuous and compact supported function such that $\rchi_{[-\ve,M+\ve]}\ge F\ge \rchi_{[0,M]}$. 
	\[
	\frac{1}{d_n^\om} \sum_{i=1}^{d_n^\om} F(\sigma_i(S^\om_\nn)) \le\frac{\#\left\{ i:\sigma_i(S^\om_\nn) \le M+\ve  \right\}}{d_n^\om},
	\]
	\[
	\frac{1}{\mu(\om)(2\pi)^d}\int_{\om\times[-\pi,\pi]^d} F(|\kappa(x)|) dx \ge  \frac{\mu(|\kappa|\le M)}{\mu(\om)(2\pi)^d} \ge \frac{\mu(H)}{\mu(\om)(2\pi)^d}.
	\]
	Since $\ms {S^\om} \sim_\sigma \kappa$, we know that
	\[
	\liminf_{n\to\infty}
	\frac{\#\left\{ i:\sigma_i(S^\om_\nn) \le M+\ve  \right\}}{d_n^\om} \ge 
	\lim_{n\to\infty}
	\frac{1}{d_n^\om} \sum_{i=1}^{d_n^\om} F(\sigma_i(S^\om_\nn)) 
	= \frac{1}{\mu(\om)(2\pi)^d}\int_{\om\times[-\pi,\pi]^d} F(|\kappa(x)|) dx  \ge \frac{\mu(H)}{\mu(\om)(2\pi)^d}
	\]
	\[
	\implies 
	\limsup_{n\to\infty}
	\frac{\#\left\{ i:\sigma_i(S^\om_\nn) > M+\ve  \right\}}{d_n^\om} \le 
	\frac{\mu(H^C)}{\mu(\om)(2\pi)^d}\le L+\ve -M,
	\]
	but
	\begin{align*}
	\min_{i=1,\dots,d_n^\om+1} \left\{ \frac{i-1}{d_n^\om} + \sigma_i(S^\om_\nn)\right\} &\le 
	\frac{\#\left\{ i:\sigma_i(S^\om_\nn) > M+\ve  \right\}}{d_n^\om} + M+\ve \\
	\implies 
	\rho(\ms {S^\om}) &=\limsup \min_{i=1,\dots,d_n^\om+1} \left\{ \frac{i-1}{d_n^\om} + \sigma_i(S^\om_\nn) \right\}\\
	&\le \limsup \frac{\#\left\{ i:\sigma_i(S^\om_\nn) > M+\ve  \right\}}{d_n^\om} + M+\ve\\
	&\le L+2\ve = \rho_{\rm mea}(\kappa) + 2\ve.
	\end{align*}
	For the converse, let $j_n$ be the sequence of indices that satisfies
	\[
	r_n:= \min_{i=1,\dots,d_n^\om+1}\left\{\frac{i-1}{d_n^\om} + \sigma_i(S^\om_\nn)\right\} = 
	\frac{j_n-1}{d_n^\om} + \sigma_{j_n}(S^\om_\nn).
	\] 
	The sequence $r_n$ is bounded by $L+\ve$ definitively, and $\frac{j_n-1}{d_n^\om}\le 1$, so $\sigma_{j_n}(S^\om_\nn)$ is also bounded and  admits a subsequence $j_{n_k}$ that converges to a value $N$. 
	Consequently,
	\[
	\rho(\mso{S}) = \limsup_{n\to\infty}\frac{j_n-1}{d_n^\om} + \sigma_{j_n}(S^\om_\nn) \ge N+ \limsup_{n\to\infty }\frac{j_{n_k}-1}{d_{n_k}^\om}
	.
	\]
	Let $F:\f R\to \f R$  be a continuous and compact supported function such that $\rchi_{[-\ve,N+2\ve]}\ge F\ge \rchi_{[0,N+\ve]}$. 
	\[
	\frac{1}{d_n^\om} \sum_{i=1}^{d_n^\om} F(\sigma_i(S^\om_\nn)) \ge\frac{\#\left\{ i:\sigma_i(S^\om_\nn) \le N+\ve  \right\}}{d_n^\om},
	\]
	\[
	\frac{1}{\mu(\om)(2\pi)^d}\int_{\om\times[-\pi,\pi]^d} F(|\kappa(x)|) dx \le  \frac{\mu(|\kappa|\le N+2\ve)}{\mu(\om)(2\pi)^d}.
	\]
	Since $\ms {S^\om} \sim_\sigma \kappa$, we know that
	\[
	\limsup_{n\to\infty}
	\frac{\#\left\{ i:\sigma_i(S^\om_\nn) \le N+\ve  \right\}}{d_n^\om} \le 
	\lim_{n\to\infty}
	\frac{1}{d_n^\om} \sum_{i=1}^{d_n^\om} F(\sigma_i(S^\om_\nn)) 
	= \frac{1}{\mu(\om)(2\pi)^d}\int_{\om\times[-\pi,\pi]^d} F(|\kappa(x)|) dx  \le \frac{\mu(|\kappa|\le N+2\ve)}{\mu(\om)(2\pi)^d}
	\]
	\[
	\implies 
	\liminf_{n\to\infty}
	\frac{\#\left\{ i:\sigma_i(S^\om_\nn) > N+\ve  \right\}}{d_n^\om} \ge 
	\frac{\mu(|\kappa|> N+2\ve)}{\mu(\om)(2\pi)^d},
	\]
	Notice that definitively in $k$, $\#\left\{ i:\sigma_i(S^\om_\nn) > N+\ve  \right\}\le \#\left\{ i:\sigma_i(S^\om_\nn) > \sigma_{j_k} (S^\om_\nn) \right\}\le j_{n_k}-1$, so
	\begin{align*}
	\rho(\mso{S}) &\ge N+ \limsup_{n\to\infty }\frac{j_{n_k}-1}{d_{n_k}^\om}\\
	&\ge N + \liminf_{n\to\infty }\frac{j_{n_k}-1}{d_{n_k}^\om}\\
	&\ge N + \liminf_{k\to\infty}\frac{\#\left\{ i:\sigma_i(S^\om_\nn) > N+\ve  \right\}}{d_{n_k}^\om}\\
	&\ge N+2\ve  + \frac{\mu(|\kappa|> N+2\ve)}{\mu(\om)(2\pi)^d} -2\ve\\
	&\ge \rho_{\rm mea}(\kappa) -2\ve.
	\end{align*}
	Since we proved that $\rho_{\rm mea}(\kappa) +2\ve\ge \rho(\mso{S})\ge \rho_{\rm mea}(\kappa) -2\ve$ for every $\ve >0$, we conclude that $\rho(\mso{S})= \rho_{\rm mea}(\kappa)$.
	Now the proof is finished, since if we take $\mso A\GLTom \kappa$ and $\mso B\GLTom \xi$, then we have by \textbf{GLT 4} that $\mso A-\mso B\GLTom \kappa-\xi$, so
	\[
	\dacs{\mso A}{\mso B} = \rho(\mso A-\mso B) =\rho_{\rm mea}(\kappa-\xi) = d_{\rm mea}(\kappa,\xi).
	\]
\end{proof}

\begin{corollary}
	The space $\mc G_d^\om$ is a complete pseudometric space when equipped with the acs distance.
\end{corollary}
\begin{proof}
	Suppose that $\{B^\om_{\nn,m}\}_n$ is a Cauchy sequence in the acs metric and $\{B^\om_{\nn,m}\}_n\GLTom \kappa_m$. By Theorem \autoref{isometry}, also $\kappa_m$ is a Cauchy sequence for the convergence in measure. Both the spaces of matrix sequences and measurable functions are complete spaces, so $\{B^\om_{\nn,m}\}_n\acs\mso A$ and $\kappa_m\to \kappa$. \textbf{GLT 7} let us conclude that $\mso A\GLTom \kappa$, so any Cauchy sequence in $\mc G_d^\om$ converges in $\mc G_d^\om$.
\end{proof}

Let us now show how the theory of reduced GLT is useful in the context of linear PDE and their discretization.

\section{Shortley-Weller Approximation}\label{SW}

Consider a linear partial differential equation
\[
\mc L(u)(x) = b(x) \qquad x\in \om^\circ
\]
equipped with some boundary conditions (Dirichlet, Neumann, etc.) when $x\in\partial\om$.
Suppose that  $\om\cu [0,1]^d$ is a  closed Peano-Jordan measurable set and $b$ is a function defined over $\om$.

We can try to discretize the equation by considering the $d$-dimensional grid $\Xi_n$ over $[0,1]^d$ and by applying a Finite Difference method only on the points of the grid inside $\om$.
Notice that the union of $\Xi_n$ for every $n$ is  the set $\f Q^d\cap [0,1]^d$, that  is dense in $[0,1]^d$, and consequently even the set
\[
\bigcup_{n\in\f N} (\Xi_n\cap \om^\circ)  = \f Q^d\cap [0,1]^d \cap \om^\circ
\]
is dense in $\om^\circ$. The grids are hence bound to describe well the interior of $\om$, but the same cannot be said about the border. In fact, it may happen that
\[
\f Q^d\cap \delta \om = \emptyset
\]
and in this case no point from $\Xi_n$ belongs to $\partial\om$, hence the discretization does not take in account the boundary conditions of the problem. When dealing with hyper-tetrahedrons, one can build  regular grids whose points on the border  are dense through an affinity. 
Otherwise, we need to use non regular grids shaped accordingly to the boundary, like the ones that arise from the \textit{Shortley-Weller Approximation} for a convection-diffusion-reaction linear PDE.

\subsection{Convection-Diffusion-Reaction PDE}

Let us consider the problem 
\begin{equation}\label{problem.dFD}
\left\{\begin{array}{ll}
\displaystyle-\sum_{i=1}^d\frac{\partial}{\partial x_i}\Bigl(a_{i}\frac{\partial u}{\partial x_i}\Bigr)+\sum_{i=1}^db_i\frac{\partial u}{\partial x_i}+cu=f, & \mbox{in }\om^\circ,\\
u=0, & \mbox{on }\partial(\om).
\end{array}\right. 
\end{equation}
where $a_{i}$, $b_i$, $c$ and $f$ are given real-valued continuous functions defined on $\om$ 
and $a_{i}\in C^1(\om)$. Moreover, suppose that $\om$ is a closed Peano-Jordan measurable set inside $[0,1]^d$ with positive measure.
We set $\hh=\frac\uu{\nn+\uu}$, so that $x_\jj=\jj \hh$ for $\jj=\mathbf0,\ldots,\nn+\uu$ are the points of the grid $\Xi_{n}$.
It is also natural to assume that $\nn +\uu= n{\bm c}$, where $\bm c$ is a vector of rational constants.
Let $\mathbf e_i$ be the vectors of the canonical basis of $\mathbb R^d$ and notice that 
$x_\jj + sh_i\bm e_i = x_{\jj + s\bm e_i}$. Then, for $\jj=\uu,\ldots,\nn$, we try to approximate the terms appearing in \eqref{problem.dFD} 
according to the classical central FD discretizations on $[0,1]^d$
as follows:
\begin{align}
\notag\left.\frac{\partial}{\partial x_i}\left(a_{i}\frac{\partial u}{\partial x_i}\right)\right|_{x=x_\jj}
&
\approx 
\frac
{a_{i}
	\dfrac{\partial u}{\partial x_i}(x_{\jj + \mathbf e_i/2})
	-
	a_{i}
	\dfrac{\partial u}{\partial x_i}(x_{\jj - \mathbf e_i/2})
}
{h_i}
\\
\label{FD.jj}
&
\approx 
a_{i}(x_{\jj + \mathbf e_i/2})
\frac
{u(x_{\jj + \mathbf e_i}) - u(x_\jj)}
{h_i^2}-
a_{i}(x_{\jj - \mathbf e_i/2})
\frac
{u(x_\jj) - u(x_{\jj - \mathbf e_i})}
{h_i^2}
\\
\label{FD.j}\left.b_i\frac{\partial u}{\partial x_i}\right|_{x=x_\jj} &\approx b_i(x_\jj)\frac{u(x_{\jj+\mathbf e_i})-u(x_{\jj-\mathbf e_i})}{2h_i},\\ 
\label{FD.}cu|_{x=x_\jj}&=c(x_\jj)u(x_\jj),
\end{align}
for $i=1,\ldots,d$. This approach requires that all the segments connecting the points $x_\jj$  with $\jj=\uu,\ldots,\nn$,   to their neighbours $x_{\jj \pm \bm e_i}$   still lie inside the domain of the problem. It always happens if the domain is $[0,1]^d$, but when we 
consider $\om$, we need to modify the scheme by adding some points.
In particular, we define a new set of neighbours for every point in $\Xi_n':=\om^\circ\cap \Xi_n$.
Given $x_\jj\in \Xi'_n$ and a direction $\bm e_i$, we can set the numbers $s_i^+(\jj),s_i^-(\jj)$ as
\[
s_i^\pm(\jj) = 
\sup\Set{t\in[0,1]|  x_\jj \pm rh_i\bm e_i\in \om^\circ \quad\forall\, 0\le r\le t       }
\]
that is the size of the biggest connected line contained in the segment connecting $x_\jj$ to $x_{\jj + \bm e_i}$ and containing $x_\jj$. 
We can thus call $x_\jj+s_i^\pm(\jj)h_i\bm e_i = x_{\jj + s_i^\pm(\jj)\bm e_i}$ the right/left neighbour of $x_\jj$ along the direction $\bm e_i$.
The values $s_i^\pm(\jj)$ depend on the point $x_\jj$, but when it is evident, we can omit the index and write simply $s_i^\pm$. 

As we can see in \autoref{fig:neighbours_FD}, even if $x_\jj$ and $x_{\jj + \bm e_i}$ belong to $\Xi'_n$, it doesn't mean that $s_i^+(\jj)=1$, because the segment connecting  $x_\jj$  to $x_{\jj + \bm e_i}$ may not be contained entirely in $\om^\circ$ (this happens often, for example, when $\om$ is not convex).

\begin{figure}[t]
	\centering
	\begin{minipage}[b]{0.45\textwidth}
		\centering
		\begin{tikzpicture}
		\clip (-.2,-.2) rectangle (5.2,5.2);
		\draw[-] (0,0) -- (5,1) -- (2,5) -- (0,0);
		\foreach \x in {0,...,10}{
			\foreach \y in {0,...,10}{
				\ifthenelse{\cnttest{5*\x-2*\y}>{0} \AND \cnttest{\x-5*\y}<{0}\AND \cnttest{4*\x+3*\y}<{46}}
				{\node[circle,fill=black,scale=.3] (a) at (\x/2,\y/2) {};}
				{\ifthenelse{\cnttest{5*\x-2*\y}={0} \OR \cnttest{\x-5*\y}={0}\OR \cnttest{4*\x+3*\y}={46}}
					{\node[circle,fill=red,scale=.3](a) at (\x/2,\y/2) {};}
					{\node[circle,fill=black,fill opacity=.2,scale=.3](a) at (\x/2,\y/2) {};}	
		}}}
		
		\foreach \x in {0,...,4}{
			\node[circle,fill=red,scale=.3](a) at (\x/2,5*\x/4) {};
			\node[circle,fill=red,scale=.3](a) at (\x/2,\x/10) {};
			\draw[-,opacity=.2] (\x/2,0)--(\x/2,5);
		}
		
		\foreach \x in {5,...,10}{
			\node[circle,fill=red,scale=.3](a) at (\x/2,23/3-2*\x/3) {};
			\node[circle,fill=red,scale=.3](a) at (\x/2,\x/10) {};
			\draw[-,opacity=.2] (\x/2,0)--(\x/2,5);
		}
		
		\foreach \y in {0,...,2}{
			\node[circle,fill=red,scale=.3](a) at (\y/5,\y/2) {};
			\draw[-,opacity=.2] (0,\y/2)--(5,\y/2);
		}
		
		\foreach \y in {3,...,10}{
			\node[circle,fill=red,scale=.3](a) at (\y/5,\y/2) {};
			\node[circle,fill=red,scale=.3](a) at (46/8 - 3*\y/8,\y/2) {};
			\draw[-,opacity=.2] (0,\y/2)--(5,\y/2);}
		
		\end{tikzpicture}
	\end{minipage}
	\hfill
	\begin{minipage}[b]{0.45\textwidth}
		\centering
		\begin{tikzpicture}
		\clip (-.2,-.2) rectangle (5.2,5.2);
		\draw [-]   (.5,1) to[out=80,in=130] (4,4)  to[out=-130,in=170] (3,2.5) to[out=-10,in=10] (4.5,1) to[out=190,in=-50] (3,.5) to[out=130,in=-50] (2.75,1.75) to[out=130,in=50] (2.5,.5) to[out=-130,in=-100]  (.5,1);
		\draw [-]   (3.5,3) to[out=130,in=130] (4.5,3.5)  to[out=-130,in=-20] (4.5,2) to[out=160,in=-50] (3.5,3);
		
		\foreach \X in {0,...,10}{
			\foreach \Y in {0,...,10}{
				\path let \n1={int(\X+11*\Y)} in node [circle]  (a-\n1) at (\X/2,\Y/2) {};
		}}
		
		\foreach \x in {0,...,10}{
			\draw[-,opacity=.2] (\x/2,0)--(\x/2,5);
			\draw[-,opacity=.2] (0,\x/2)--(5,\x/2);
		}

		\foreach \n in {0,...,12,19,20,21,22,32,33,34,43,44,45,54,55,56,57,62,65,66,67,68,75,76,...,80,87,88,...,92,97,98,...,120}{
			\node[circle,fill=black,fill opacity=.2,scale=.3](a) at (a-\n) {};
		}
		\foreach \n in {16,17,23,31,53,61,72,73,84,86,96}{
			\node[circle,fill=red,scale=.3](a) at (a-\n) {};
		}
		\foreach \n in {13,14,15,18,24,25,...,30,35,36,...,42,46,47,...,52,58,59,60,63,64,69,70,71,74,81,82,83,85,93,94,95}{
			\node[circle,fill=black,scale=.3](a) at (a-\n) {};
		}

		\node[circle,fill=red,scale=.3](a) at (1,.3) {};
		\node[circle,fill=red,scale=.3](a) at (1,2.37) {};
		\node[circle,fill=red,scale=.3](a) at (1.5,.18) {};
		\node[circle,fill=red,scale=.3](a) at (1.5,3.17) {};
		\node[circle,fill=red,scale=.3](a) at (2,.22) {};
		\node[circle,fill=red,scale=.3](a) at (2,3.75) {};
		\node[circle,fill=red,scale=.3](a) at (2.5,4.15) {};
		\node[circle,fill=red,scale=.3](a) at (3,4.36) {};
		\node[circle,fill=red,scale=.3](a) at (3.5,.43) {};
		\node[circle,fill=red,scale=.3](a) at (3.5,4.34) {};
		\node[circle,fill=red,scale=.3](a) at (3.5,2.32) {};
		\node[circle,fill=red,scale=.3](a) at (4,.77) {};
		\node[circle,fill=red,scale=.3](a) at (4,2.04) {};
		\node[circle,fill=red,scale=.3](a) at (4,2.35) {};
		\node[circle,fill=red,scale=.3](a) at (4,3.58) {};
		\node[circle,fill=red,scale=.3](a) at (4.5,1.62) {};
		\node[circle,fill=red,scale=.3](a) at (4.5,2.72) {};

		\node[circle,fill=red,scale=.3](a) at (.68,.5) {};
		\node[circle,fill=red,scale=.3](a) at (3.62,.5) {};
		\node[circle,fill=red,scale=.3](a) at (2.63,1) {};
		\node[circle,fill=red,scale=.3](a) at (2.87,1) {};
		\node[circle,fill=red,scale=.3](a) at (.63,1.5) {};
		\node[circle,fill=red,scale=.3](a) at (2.62,1.5) {};
		\node[circle,fill=red,scale=.3](a) at (2.87,1.5) {};
		\node[circle,fill=red,scale=.3](a) at (4.6,1.5) {};
		\node[circle,fill=red,scale=.3](a) at (.82,2) {};
		\node[circle,fill=red,scale=.3](a) at (4.05,2) {};
		\node[circle,fill=red,scale=.3](a) at (1.07,2.5) {};
		\node[circle,fill=red,scale=.3](a) at (3.87,2.5) {};
		\node[circle,fill=red,scale=.3](a) at (4.57,2.5) {};
		\node[circle,fill=red,scale=.3](a) at (1.37,3) {};
		\node[circle,fill=red,scale=.3](a) at (4.42,3) {};
		\node[circle,fill=red,scale=.3](a) at (1.76,3.5) {};
		\node[circle,fill=red,scale=.3](a) at (3.77,3.5) {};
		\node[circle,fill=red,scale=.3](a) at (2.28,4) {};


		\end{tikzpicture}
	\end{minipage}
	\captionsetup{font={small,it}} 
	\caption{
		Points of the grid $\Xi_n$ over two different domains $\om$. The points in $\Xi_n'$ are black, and their neighbours on the boundary are red.}
	\label{fig:neighbours_FD}
\end{figure}
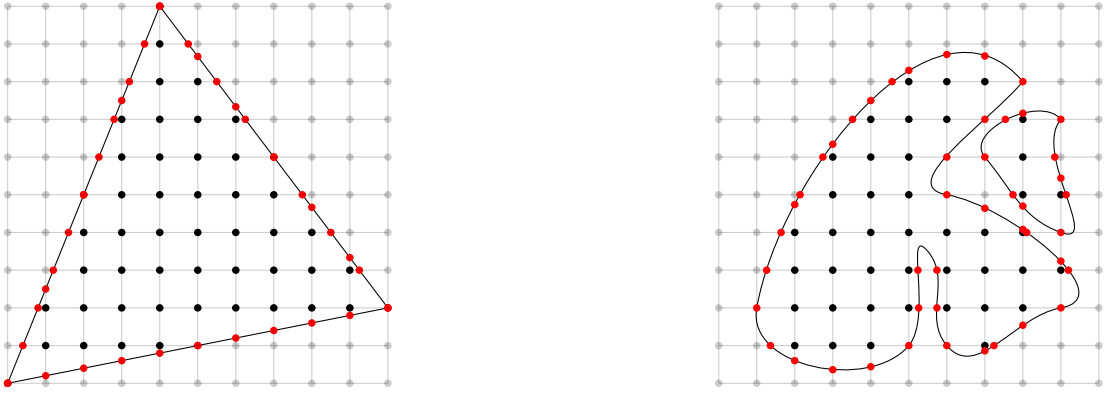

Notice that every neighbour is a point of $\om$, so when one of the neighbours is not included in $\Xi'_n$, it surely belongs the boundary $\partial\om$, and in any case we have $s_i^\pm>0$. Adding these boundary points to $\Xi_n'$, we obtain the discretization grid $\Xi_n^\om$ over $\om$, and we can rewrite the formulas (\ref{FD.jj})-(\ref{FD.}) for $x_\jj\in\Xi'_n$  as 
\begin{align}
\notag\left.\frac{\partial}{\partial x_i}\left(a_i\frac{\partial u}{\partial x_i}\right)\right|_{x=x_\jj}
&
\approx 
\frac
{a_i
	\dfrac{\partial u}{\partial x_i}(x_{\jj + s_i^+\mathbf e_i/2})
	-
	a_i
	\dfrac{\partial u}{\partial x_i}(x_{\jj - s_i^-\mathbf e_i/2})
}
{\frac 12(s_i^++s_i^-)h_i}
\\
\label{FDsw.jj}
&
\approx 
a_i(x_{\jj + s_i^+\mathbf e_i/2})
\frac
{u(x_{\jj + s_i^+\mathbf e_i}) - u(x_\jj)}
{\frac 12s_i^+(s_i^++s_i^-)h_i^2}-
a_i(x_{\jj - s_i^-\mathbf e_i/2})
\frac
{u(x_\jj) - u(x_{\jj - s_i^-\mathbf e_i})}
{\frac 12s_i^-(s_i^++s_i^-)h_i^2}
\\ 
\label{FDsw.j}\left.b_i\frac{\partial u}{\partial x_i}\right|_{x=x_\jj} &\approx b_i(x_\jj)\frac{u(x_{\jj+s_i^+\mathbf e_i})-u(x_{\jj-s_i^-\mathbf e_i})}{(s_i^++s_i^-)h_i},\\ 
\label{FDsw.}cu|_{x=x_\jj}&=c(x_\jj)u(x_\jj),
\end{align}
called the difference scheme of \textit{Shortley and Weller} \cite{short}. Notice that when $s_j^\pm=1$  for every $j$ and sign $\pm$, we fall again in the classical scheme of central differences.

The evaluations $u(x_\jj)$ of the solution at the grid points $x_\jj\in \Xi_n^\om$ are approximated by the values $u_\jj$, where $u_\jj=0$ for $x_\jj\in \partial\om$, and the vector $\bm u=(u_\jj)^T_{x_\jj\in \om^\circ }$ is the solution of the linear system
\begin{align}\label{sistema.FD}
\notag &-\sum_{i=1}^d a_i(x_{\jj + s_i^+\mathbf e_i/2})
\frac
{u_{\jj + s_i^+\mathbf e_i} - u_\jj}
{\frac 12s_i^+(s_i^++s_i^-)h_i^2}-
a_i(x_{\jj - s_i^-\mathbf e_i/2})
\frac
{u_\jj - u_{\jj - s_i^-\mathbf e_i}}
{\frac 12s_i^-(s_i^++s_i^-)h_i^2}
\\
&+\sum_{i=1}^d
b_i(x_\jj)\frac{u_{\jj+s_i^+\mathbf e_i}-u_{\jj-s_i^-\mathbf e_i}}{(s_i^++s_i^-)h_i}
+c(x_\jj)u_\jj=f(x_\jj),\qquad\jj : x_\jj\in \om^\circ.
\end{align}
If we order the indices $\jj$ in $\Xi_n'$ by lexicographic order, then we can  write the system in compact form as
\[
A_\nn^{\om^\circ} \bm u = \bm f,
\]
where $A_\nn^{\om^\circ}\in \f C^{d_n^{\om^\circ}\times d_n^{\om^\circ}}$ and $\bm f\in \f C^{d_n^{\om^\circ}}$. The coefficients are
\[
(A_\nn^{\om^\circ})_{\jj,\ii} = 
\begin{cases}
\sum_{i=1}^d \left[
\frac
{ a_i(x_{\jj + s_i^+\mathbf e_i/2})}
{\frac 12s_i^+(s_i^++s_i^-)h_i^2}+
\frac
{a_i(x_{\jj - s_i^-\mathbf e_i/2})}
{\frac 12s_i^-(s_i^++s_i^-)h_i^2}\right]
+c(x_\jj),
& \ii= \jj,\\
\frac
{-a_i(x_{\jj \pm \mathbf e_i/2})}
{\frac 12(s_i^++s_i^-)h_i^2}
+
\frac{\pm b_i(x_\jj)}{(s_i^++s_i^-)h_i},
& \ii = \jj \pm \bm e_i, s_i^{\pm} = 1\\
0, & \textnormal{otherwise}.
\end{cases}
\]
Notice that one can rewrite the nonzero off-diagonal  coefficients as
\[
(A_\nn^{\om^\circ})_{\jj,\jj \pm e_i} = \frac
{-a_i(x_{\jj \pm\mathbf e_i/2})}
{\frac 12(s_i^++s_i^-)h_i^2}
+
\frac{\pm b_i(x_\jj)}{(s_i^++s_i^-)h_i}
=
\frac{2}{s_i^++s_i^-}
\left(
\frac
{-a_i(x_{\jj \pm\mathbf e_i/2})}
{h_i^2}
+
\frac{\pm b_i(x_\jj)}{2h_i}
\right).
\]

\subsection{Spectral Analysis}

As already noted, if all $s_i^\pm$ are equal to 1, then the relations (\ref{FDsw.jj})-(\ref{FDsw.}) reduces to the classic finite difference scheme (\ref{FD.jj})-(\ref{FD.}), so we may ask how many are the points $x_\jj\in \om^\circ$ such that one of the  $s_i^\pm$ is not equal to 1. By the definition of $s_i^\pm$, this is equivalent to say that the segment $(x_\jj -h_i\bm e_i,x_\jj +h_i\bm e_i)$ does not lie completely inside $\om^\circ$. In the next result, we will prove that given any positive integer number $k$, the number of points $x_\jj \in \Xi_n'$ for which there exists a direction $\bm e_i$ such that $(x_\jj -kh_i\bm e_i,x_\jj +kh_i\bm e_i)$ does not lie completely inside $\om^\circ$ is negligible when compared with the number of points in $\Xi_n'$.

\begin{lemma}\label{near_border_points}
	Let
	\[
	D(n, k):= \set{ x_\jj \in \Xi'_n | \exists i, t\in (-k,k) : x_\jj + th_i\bm e_i \not\in \om^\circ   }.
	\]	
	For every $k>0$, we have
	\[
	\# D(n,k) =  o(N(\nn)).
	\]
\end{lemma}
\begin{proof}
	
	Notice that if $x_\jj\in D(n, k)$, then there exists a direction $\bm e_i$ and a value $t\in (-k,k)$ such that  $x_\jj + th_i\bm e_i\in \partial \om$ and $t\ne 0$. In particular, we infer that $d(x_\jj, \partial \om)<kh_i$ and if we denote $h=\max_i h_i$, then  $d(x_\jj, \partial \om)<kh$. 
	Using notations and results of \autoref{preliminaries_near_border_points_chi}, we know that $x_\jj\in K_{kh}\cap \Xi_n'$,  
	but $kh\to 0$ as $n$ goes to infinity, so
	\[
	\# D(n,k) \le d_n^{K_{kh}} = o(d_n^{\om^\circ}) \implies 	\# D(n,k) = o(N(\nn)).
	\] 
\end{proof}

We just proved that, except for few relations, the system (\ref{sistema.FD}) mimics a classical FD scheme.
We can thus consider  the extended problem
\begin{equation}\label{problem.dFD.01}
\left\{\begin{array}{ll}
\displaystyle-\sum_{i=1}^d\frac{\partial}{\partial x_i}\Bigl(a'_{i}\frac{\partial u}{\partial x_j}\Bigr)+\sum_{i=1}^db'_i\frac{\partial u}{\partial x_i}+c'u=f', & \mbox{in }(0,1)^d,\\
u=0, & \mbox{on }\partial([0,1]^d).
\end{array}\right. 
\end{equation}
where $a'_{i},b'_i,c',f'$ are functions that extend $a_{i},b_i,c,f$ 
\[
a'_{i}(x)
=
\begin{cases}
a_{i}(x), & x\in \om,\\
0, & x\not\in \om,
\end{cases}\qquad 
b'_{i}(x)
=
\begin{cases}
b_{i}(x), & x\in \om,\\
0, & x\not\in \om,
\end{cases}\]\[ 
c'(x)
=
\begin{cases}
c(x), & x\in \om,\\
0, & x\not\in \om,
\end{cases}\qquad 
f'(x)
=
\begin{cases}
f(x), & x\in \om,\\
0, & x\not\in \om.
\end{cases}
\]
Notice that $b'_i,c'$ are bounded functions since $\om$ is a compact set, and moreover $a'_{i}$ are bounded and continuous a.e. functions. In \cite{perturbation}, it is showed that these conditions on the coefficients are enough to prove that the matrices $A_\nn$ induced by the relations (\ref{FD.jj})-(\ref{FD.}) build a GLT sequence with symbol
\[
\{ n^{-2}A_\nn \}_n\GLT k(x,\theta) =
\sum_{i=1}^{d}\nu_i^2 a'_{i}(x)(2 - 2\cos(\theta_i)),
\]
where $\nn +\uu = n\bm{\nu}$. 
This is also enough to let us conclude that $\{ n^{-2} A_\nn^{\om^\circ} \}_n$ is actually a reduced GLT sequence.
\begin{theorem}
	\[
	\{ n^{-2}A_\nn^{\om^\circ} \}_n\sim_{GLT}^{\om^\circ} \kappa(x,\theta)  = \sum_{i=1}^{d}\nu_i^2 a_{i}(x)(2 - 2\cos(\theta_i)).
	\]
\end{theorem}
\begin{proof}
	Denote with $B_\nn^{\om^\circ}$ and $Z^{\om^\circ}_\nn$ the matrices
	\[
	B_\nn^{\om^\circ} = R_{\om^\circ}(A_\nn),
	\qquad Z^{\om^\circ}_\nn = B_\nn^{\om^\circ} - A_\nn^{\om^\circ},
	\]  
	where the rows and columns are associated to the points $x_\jj\in \Xi'_n$. If  $x_\jj\in \Xi_n'\setminus D(n,2)$, then $x_\jj$  is a point of the grid $\Xi_n$ inside $\om^\circ$  such that all its neighbours  still belong to $\om^\circ$. In this case, 
	\[
	(A_\nn^{\om^\circ})_{\jj,\ii} = (B_\nn^{\om^\circ})_{\jj,\ii} =(A_\nn)_{\jj,\ii} = \begin{cases}
	c(x_\jj) + \sum_{i=1}^d\frac{ a_i(x_{\jj+\bm e_i/2}) + a_i(x_{\jj-\bm e_i/2}) }{h_i^2} & \ii = \jj,\\
	-\frac{a_i(x_{\jj\pm \bm e_i/2})}{h_i^2} \pm \frac{b_i(x_\jj)}{2h_i} & \ii = \jj \pm \bm e_i,\\
	0, & \textnormal{otherwise},
	\end{cases} 
	\]
	hence the row corresponding to $x_\jj$ in $Z^{\om^\circ}_\nn$ is zero. Using the result in \autoref{near_border_points}, we  conclude that the number of non-zero rows in $Z^{\om^\circ}_\nn$ is $o(N(\nn))$, so $\ms{Z^{\om^\circ}}$ is a zero-distributed sequence, since \autoref{dimension_chi} assures us that
	\[
	\rk(Z^{\om^\circ}_\nn) = o(N(\nn)) \implies \rk(Z^{\om^\circ}_\nn)= o(d_n^{\om^\circ}).
	\] 
	From \textbf{GLT 3} and \textbf{GLT 4}, we conclude that
	\[
	\{n^{-2}Z_\nn  \}_n
	\sim_{GLT}^{\om^\circ} 0, \qquad
	\{ n^{-2}B_\nn^{\om^\circ} \}_n = 
	R_{\om^\circ}(\{ n^{-2}A_\nn \}_n)
	\sim_{GLT}^{\om^\circ} \kappa \]\[
	\implies 
	\{ n^{-2}A_\nn^{\om^\circ} \}_n = \{n^{-2}Z_\nn  \}_n +  \{ n^{-2}B_\nn^{\om^\circ} \}_n\sim_{GLT}^{\om^\circ} \kappa.
	\]
\end{proof}
A more involved analysis is needed to conclude that $\{ n^{-2}A_\nn^{\om^\circ} \}_n\sim_\lambda \kappa$. If $A_\nn^{\om^\circ}$ were Hermitian matrices, the result would follow from \textbf{GLT 1}, but it is almost never the case. Notice that $\kappa$ is a real valued function, so we can decompose $A_\nn^{\om^\circ}$ into its Hermitian and skew-Hermitian part. Using \textbf{GLT 1, 4}, we have
\[
\Re(n^{-2}A_\nn^{\om^\circ}) = \frac{n^{-2}A_\nn^{\om^\circ} + n^{-2}(A_\nn^{\om^\circ})^H}{2} 
\implies
\{\Re(n^{-2}A_\nn^{\om^\circ})  \}_n \sim_{GLT}^{\om^\circ} \kappa, \qquad \{\Re(n^{-2}A_\nn^{\om^\circ})  \}_n \sim_\lambda \kappa.
\]
On the other hand,  the skew-Hermitian part is zero-distributed, but in order to write the expression for its coefficients, we need to remind that the values $s_i^\pm$ depend on the point $x_\jj$. To avoid confusion, in this case we will denote them by $s_i^\pm(\jj)$.
\[
\Im(n^{-2}A_\nn^{\om^\circ}) = \frac{n^{-2}A_\nn^{\om^\circ} - (n^{-2}A_\nn^{\om^\circ})^*}{2} 
\implies 
\{n^{-2}\Im(A_\nn^{\om^\circ})  \}_n \sim_{GLT}^{\om^\circ} 0
\]
\[
(\Im(n^{-2}A_\nn^{\om^\circ}) )_{\jj,\ii} = 
\begin{cases}
\frac{n^{-2}}{1+s_i^-(\jj)}
\left(
\frac
{-a_i(x_{\jj +\mathbf e_i/2})}
{h_i^2}
+
\frac{b_i(x_\jj)}{2h_i}
\right)
-
\frac{n^{-2}}{s_i^+(\ii)+1}
\left(
\frac
{-a_i(x_{\ii -\mathbf e_i/2})}
{h_i^2}
+
\frac{- b_i(x_\ii)}{2h_i}
\right),
& \ii = \jj +\bm e_i,\\
\frac{n^{-2}}{s_i^+(\jj)+1}
\left(
\frac
{-a_i(x_{\jj- \mathbf e_i/2})}
{h_i^2}
+
\frac{ -b_i(x_\jj)}{2h_i}
\right)
-
\frac{n^{-2}}{1+s_i^-(\ii)}
\left(
\frac
{-a_i(x_{\ii +\mathbf e_i/2})}
{h_i^2}
+
\frac{b_i(x_\ii)}{2h_i}
\right),
& \ii = \jj -\bm e_i,\\
0, & \textnormal{otherwise}.
\end{cases}
\]
Notice that $s_i^\pm \in (0,1]$, so we can bound every entry by 
\begin{equation}\label{elements_near_border}
\left|(\Im(n^{-2}A_\nn^{\om^\circ}) )_{\jj,\ii} \right|
\le \nu(2\nu\|a\|_\infty + n^{-1}\|b\|_\infty),
\end{equation}
where $\nu = \max_i\nu_i$.
Moreover, suppose $x_\jj$ is a grid point in $\Xi_n' \setminus D(n,3)$. In particular, we have $s_i^\pm(\jj) = s_i^\pm(\jj + e_i) = s_i^\pm(\jj - e_i) = 1$ for every $i$. In this case, the row $\jj$ is easier to write
\[
(\Im(n^{-2}A_\nn^{\om^\circ}) )_{\jj,\ii} = 
\begin{cases}
\frac{n^{-2}}{h_i}
\left(
\frac{b_i(x_\jj)+ b_i(x_\ii)}{4}
\right),
& \ii = \jj + \bm e_i,\\
-\frac{n^{-2}}{h_i}
\left(
\frac{b_i(x_\jj)+ b_i(x_\ii)}{4}
\right),
& \ii = \jj - \bm e_i,\\
0, & \textnormal{otherwise},
\end{cases}
\]
and we can bound the entries by 
\begin{equation}\label{elements_far_from_border}
\left|(\Im(n^{-2}A_\nn^{\om^\circ}) )_{\jj,\ii} \right|
\le 2\nu n^{-1} \|b\|_\infty.
\end{equation}
\autoref{near_border_points} assures us that almost all points in $\Xi'_n$ respect these conditions.
Now we are ready to prove that $\{ n^{-2}A_\nn^{\om^\circ} \}_n\sim_\lambda \kappa$.

\begin{theorem}\[
	\{ n^{-2}A_\nn^{\om^\circ} \}_n\sim_\lambda \kappa(x,\theta)  = \sum_{i=1}^{d}\nu_i^2a_{i}(x)(2 - 2\cos(\theta_i)).\]
\end{theorem}
\begin{proof}
	Using the decomposition into Hermitian and skew-Hermitian part, we write
	\[
	n^{-2}A_\nn^{\om^\circ} =
	\Re(n^{-2}A_\nn^{\om^\circ}) +
	\Im(n^{-2}A_\nn^{\om^\circ})
	\]
	where $\Re(n^{-2}A_\nn^{\om^\circ})$ are Hermitian and $\{\Re(n^{-2}A_\nn^{\om^\circ})  \}_n \sim_\lambda \kappa$.
	Notice that every row of $\Im(A_\nn)$ has at most $2d$ non-zero elements. Using  \autoref{near_border_points} and the relations (\ref{elements_near_border},\ref{elements_far_from_border}) we obtain
	\begin{align*}
	\|\Im(n^{-2}A_\nn^{\om^\circ})\|_2^2 &= \sum_{\jj}\sum_{\ii} |(\Im(n^{-2}A_\nn^{\om^\circ}))_{\jj,\ii}|^2 \\
	& =
	\sum_{\jj: x_\jj \in \Xi_n' \setminus D(n,3)}\sum_{\ii} |(\Im(n^{-2}A_\nn^{\om^\circ}) )_{\jj,\ii}|^2
	+
	\sum_{\jj: x_\jj \in  D(n,3)}\sum_{\ii} |(\Im(n^{-2}A_\nn^{\om^\circ}) )_{\jj,\ii}|^2\\
	&\le
	\sum_{\jj: x_\jj \in \Xi_n' \setminus D(n,3)}\sum_{\ii}  4\nu^2 n^{-2} \|b\|^2_\infty
	+
	\sum_{\jj: x_\jj \in  D(n,3)}\sum_{\ii} \nu^2(2\nu\|a\|_\infty + n^{-1}\|b\|_\infty)^2
	\\
	&\le
	\sum_{\jj: x_\jj \in \Xi_n' \setminus D(n,3)}  8d\nu^2 n^{-2} \|b\|^2_\infty
	+
	\sum_{\jj: x_\jj \in  D(n,3)} 2d\nu^2(2\nu\|a\|_\infty + n^{-1}\|b\|_\infty)^2
	\\
	&\le
	8d\nu^2 n^{-2} \|b\|^2_\infty d_n^{\om^\circ}
	+
	2d\nu^2(2\nu\|a\|_\infty + n^{-1}\|b\|_\infty)^2 o(d_n^{\om^\circ})
	= o(d_n^{\om^\circ}).
	\end{align*}
	\textbf{GLT 2} let us conclude that 
	\[
	\{ n^{-2}A_\nn^{\om^\circ} \}_n\sim_\lambda \kappa.
	\]
	
\end{proof}

\begin{remark}
	The Shortley-Weller approximation just described is actually so general it comprehends classical finite differences methods used on regular domains. For example, in 2 dimensions, every triangular domain can be transformed by affine maps into the isosceles right triangle $T$ described by the vertices with coordinates $(0,0),(0,1),(1,0)$. If we superimpose the regular grid $\Xi_n$ onto the triangle, we find that the union of the points on the border  for every $n$ is a dense set in $\delta T$.

	\begin{figure}[H]
		\centering
		\begin{tikzpicture}
		\draw[-] (0,0) -- (5,0) node[right] {$(1,0)$} -- (0,5) node[left] {$(0,1)$} -- (0,0)  node[below] {$(0,0)$};
		\foreach \x in {0,...,10}{
			\foreach \y in {0,...,10}{
				\ifthenelse{\cnttest{\x+\y}<{10} \AND \NOT \x =0 \AND \NOT \y = 0}
				{\node[circle,fill=black,scale=.3] (a) at (\x/2,\y/2) {};}
				{\ifthenelse{ \x =0 \OR \y = 0 \OR \cnttest{\x+\y}={10}}
					{\node[circle,fill=red,scale=.3](a) at (\x/2,\y/2) {};}
					{\node[circle,fill=black,fill opacity=.2,scale=.3](a) at (\x/2,\y/2) {};}	
		}}}  
		\end{tikzpicture}
		\label{regular_grid_Triangle_T}
		\caption{Superimposition of $\Xi_n$ onto the triangle $T$.}
	\end{figure}
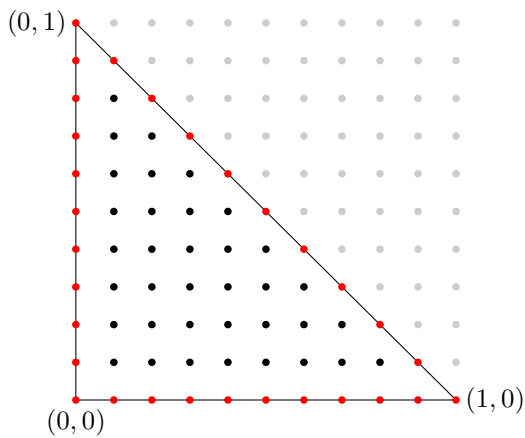

	Operating a classical second order method to discretize Problem \ref{problem.dFD} in 2 dimensions, we fall again in the Shortley-Weller method, so we already know the symbol of the resulting linear system.
	
\end{remark}

\section{$P_1$ Method}\label{P1}

Consider a linear  partial differential equation
\[
\mc L(u)(x) = f(x) \qquad x\in \om^\circ
\]
equipped with some boundary conditions (Dirichlet, Neumann, etc.) when $x\in\partial\om$, where  $\om\cu [0,1]^d$ is a  closed Peano-Jordan measurable set with positive measure and  $f$ is a function defined over $\om$.

A common way to discretize the problem is to use a finite elements method, that is based on the choice of a basis for the functions on the domain $\om$. The basis does not necessarily depend on a grid of points inside $\om$, but usually they do, so on a generic $\om$ there's again the problem to describe the boundary. 
For this reason, usually the domains are polyhedral or with a regular enough boundary. When we deal with more general shapes, we may need to map the domain into a regular one, or to modify the grids of discretization, and a more involved analysis is required.

Let us consider the problem 
\begin{equation}\label{problem.dFE.om}
\left\{\begin{array}{ll}
\displaystyle-\sum_{i,j=1}^2\frac{\partial}{\partial x_j}\Bigl(a_{i,j}\frac{\partial u}{\partial x_i}\Bigr)+\sum_{i=1}^2b_i\frac{\partial u}{\partial x_i}+cu=f, & \mbox{in }\om^\circ,\\
u=0, & \mbox{on }\partial\om,
\end{array}\right.
\end{equation}
where $\om$ is a closed set inside $[0,1]^2$ with negligible boundary and positive measure. Moreover 
$a_{i,j}$, $b_i$, $c$ and $f$ are given complex-valued continuous functions defined on $\om$ 
and $a_{i,j}\in C^1(\om)$. If $A=(a_{i,j})_{i,j=1}^2$ is a matrix of functions and $\bm b = (b_1,b_2)^T$, then the equivalent weak form of (\ref{problem.dFE.om}) reads as
\begin{equation}\label{weak.problem.dFE.triangle}
\displaystyle
\int_{\om^\circ} (\nabla u)^TA\nabla w + (\nabla u)^T\bm b w + cuw  = \int_{\om^\circ} fw,\qquad \forall w\in H_0^1(\om)
.
\end{equation}

The space $[0,1]^2$ is divided into triangles as shown in Figure \ref{general_domain_FE}, whose vertices are the nodes of $\Xi_n$. The $P_1$ finite elements method, studied in \cite{Beck,tyr}, uses base functions supported on the grid triangles that fall inside $\om$. We say that the adjacent nodes of a point $p\in\Xi_n$ are its neighbours, and  we call $N(p)$ the set composed of $p$ and its neighbours. 
Each point $p$ is a vertex for at most 6 triangles, that we call $T_{i,p}$ as shown in Figure \ref{near_point_triangles}, and we denote their union as $T_p$ (notice that they depend also on $n$, but for brevity we omit the index). The collection of all the triangles in the scheme associated to the grid $\Xi_n$ is 
\[\mc T_n = 
\Set{T_{i,p} | p\in \Xi_n, \; i=1,\dots,6}.
\]
For every  point $p\in \Xi_n$ such that $T_p\cu [0,1]^2$, we define a function $\psi_{p,n}$ that is linear on each triangle, whose value is $1$ at $p$ and $0$ on every other point of $\Xi_n$.
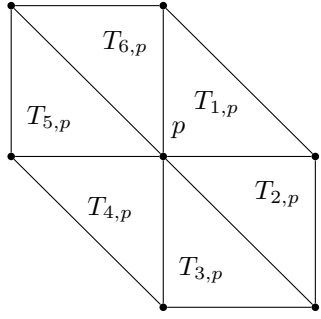
\begin{figure}[H]
	\centering
	\begin{tikzpicture}
	\node[circle,fill=black,scale=.3] (a) at (0,0) {};
	\node[label={[shift={(.2,0)}]$p$}] () at (0,0) {};
	
	\node[circle,fill=black,scale=.3] (a1) at (2,0) {};
	\draw[-] (a1) -- (a);
	\node[circle,fill=black,scale=.3] (a2) at (0,2) {};
	\draw[-] (a2) -- (a);
	\node[circle,fill=black,scale=.3] (a3) at (-2,2) {};
	\draw[-] (a3) -- (a);
	\node[circle,fill=black,scale=.3] (a4) at (-2,0) {};
	\draw[-] (a4) -- (a);
	\node[circle,fill=black,scale=.3] (a5) at (0,-2) {};
	\draw[-] (a5) -- (a);
	\node[circle,fill=black,scale=.3] (a6) at (2,-2) {};
	\draw[-] (a6) -- (a);
	\draw[-] (a1)--(a2)--(a3)--(a4)--(a5)--(a6)--(a1);
	
	\node[] () at (.7,.7) {$T_{1,p}$};
	\node[] () at (1.5,-.5) {$T_{2,p}$};
	\node[] () at (.5,-1.5) {$T_{3,p}$};
	\node[] () at (-.7,-.7) {$T_{4,p}$};
	\node[] () at (-.5,1.5) {$T_{6,p}$};
	\node[] () at (-1.5,.5) {$T_{5,p}$};
	\end{tikzpicture}
	\caption{triangles and neighbours associated to the point $p$}
	\label{near_point_triangles}
\end{figure}
We can explicitly write $\psi_{p,n}$ and its partial derivatives. If $p=(x_p,y_p)$ and $\wt x = x-x_p$, $\wt y = y-y_p$, then
\[
\psi_{p,n}(x,y)=
\begin{cases}
1-\frac{\wt x+\wt y}{h}, & ( x, y)\in T_{1,p},\\
1-\frac{\wt x}{h}, & (x, y)\in T_{2,p},\\
1+\frac{\wt y}{h} , & ( x, y)\in T_{3,p},\\
1+\frac{\wt x+\wt y}{h} , & ( x, y)\in T_{4,p},\\
1+\frac{\wt x}{h} , & ( x, y)\in T_{5,p},\\
1-\frac{\wt y}{h} , & ( x, y)\in T_{6,p},\\
0 , & \textnormal{otherwise},
\end{cases}
\]\[ 
\frac{\partial}{\partial x}\psi_{p,n}(x,y)=
\begin{cases}
-\frac{1}{h}, & ( x, y)\in T_{1,p},\\
-\frac{1}{h}, & (x, y)\in T_{2,p},\\
0 , & ( x, y)\in T_{3,p},\\
\frac{1}{h} , & ( x, y)\in T_{4,p},\\
\frac{1}{h} , & ( x, y)\in T_{5,p},\\
0, & ( x, y)\in T_{6,p},\\
0 , & \textnormal{otherwise},
\end{cases}
\qquad 
\frac{\partial}{\partial y}\psi_{p,n}(x,y)=
\begin{cases}
-\frac{1}{h}, & ( x, y)\in T_{1,p},\\
0, & (x, y)\in T_{2,p},\\
\frac{1}{h} , & ( x, y)\in T_{3,p},\\
\frac{1}{h} , & ( x, y)\in T_{4,p},\\
0 , & ( x, y)\in T_{5,p},\\
-\frac{1}{h} , & ( x, y)\in T_{6,p},\\
0 , & \textnormal{otherwise},
\end{cases}
\]
where $h = 1/(n+1)$. $P_1$ methods usually arises when the domain is not a square, but it is polyhedral or  regular enough. For example, as we can see in Figure \ref{regular_grid_Triangle_T_FE}, the subdivision scheme adopted has the property to describe also the boundary of the triangle, in opposition to the classical tensor-product hat-functions considered in \cite[Section 7.4]{GLT-bookII}.

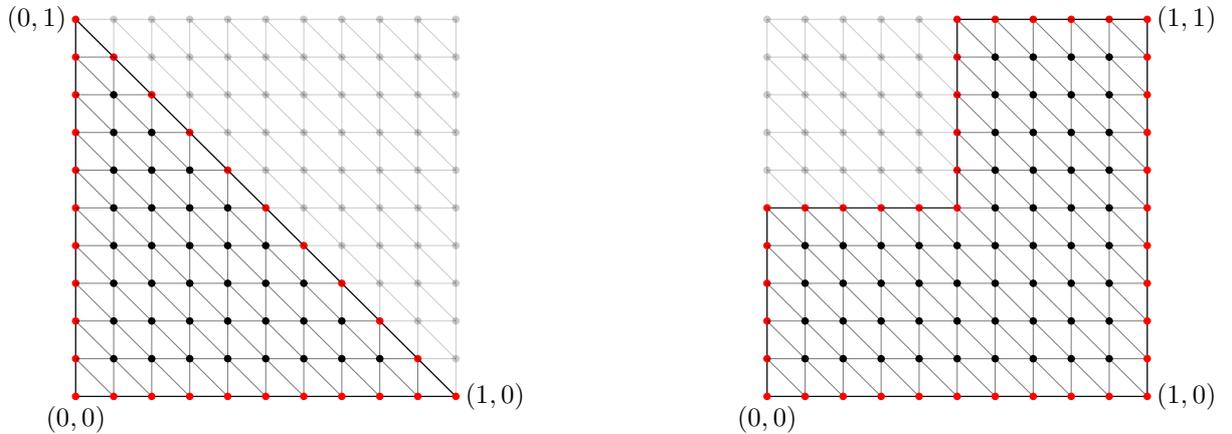
\begin{figure}[H]
	\centering
	\begin{minipage}[b]{0.45\textwidth}
		\centering
		\begin{tikzpicture}
		\draw[-] (0,0) -- (5,0) node[right] {$(1,0)$} -- (0,5) node[left] {$(0,1)$} -- (0,0)  node[below] {$(0,0)$};
		\foreach \x in {0,...,10}{
			\draw[-,opacity = .5] (\x/2,0) -- (\x/2,5-\x/2);
			\draw[-,opacity = .2]  (\x/2,5-\x/2) --
			(\x/2,5);
			\draw[-,opacity = .5] (0,\x/2) -- (5-\x/2,\x/2);
			\draw[-,opacity = .2]  (5-\x/2,\x/2) --
			(5,\x/2);
			\draw[-,opacity = .5] (\x/2,0) -- (0,\x/2);
			\draw[-,opacity = .2] (\x/2,5) -- (5,\x/2);
			\foreach \y in {0,...,10}{
				\ifthenelse{\cnttest{\x+\y}<{10} \AND \NOT \x =0 \AND \NOT \y = 0}
				{\node[circle,fill=black,scale=.3] (a) at (\x/2,\y/2) {};}
				{\ifthenelse{ \x =0 \OR \y = 0 \OR \cnttest{\x+\y}={10}}
					{\node[circle,fill=red,scale=.3](a) at (\x/2,\y/2) {};}
					{\node[circle,fill=black,fill opacity=.2,scale=.3](a) at (\x/2,\y/2) {};}	
		}}}
		\end{tikzpicture}
	\end{minipage}
	\hfill
	\begin{minipage}[b]{0.45\textwidth}
		\centering
		\begin{tikzpicture}
		\draw[-] (0,0) -- (5,0)node[right] {$(1,0)$}  -- (5,5)node[right] {$(1,1)$}  -- (2.5,5) --  (2.5,2.5) -- (0,2.5) --  (0,0)node[below] {$(0,0)$}  ;
		
		\foreach \x in {0,...,10}{
			\ifthenelse{\cnttest{\x}<{5}}
			{\draw[-,opacity = .5] (\x/2,0) -- (\x/2,2.5);
				\draw[-,opacity = .2]  (\x/2,2.5) -- (\x/2,5);
				
				\draw[-,opacity = .5] (0,\x/2) -- (5,\x/2);
				
				\draw[-,opacity = .5] (\x/2 ,0) -- (0,\x/2 );
				\draw[-,opacity = .5] (5,3+\x/2) -- (3+\x/2,5 );
				
				\foreach \y in {0,...,10}{
					\ifthenelse{\cnttest{\y}<{5} \AND \cnttest{\x}>{0} \AND \cnttest{\y}>{0}}
					{\node[circle,fill=black,scale=.3] (a) at (\x/2,\y/2) {};}
					{
						\ifthenelse{\cnttest{\y}>{5}}
						{\node[circle,fill=black,scale=.3,fill opacity = .2] (a) at (\x/2,\y/2) {};}
						{\node[circle,fill=red,scale=.3] (a) at (\x/2,\y/2) {};}
					}
				}
			}
			{\draw[-,opacity = .5]  (\x/2,0) -- (\x/2,5);
				
				\draw[-,opacity = .5] (2.5,\x/2) -- (5,\x/2);
				\draw[-,opacity = .2]  (0,\x/2) -- (2.5,\x/2);
				
				\draw[-,opacity = .5] (\x/2 ,0) -- (\x/2 -2.5,2.5);
				\draw[-,opacity = .2] (\x/2 -2.5,2.5) -- (0,\x/2);
				
				\ifthenelse{\cnttest{\x}<{10}}{
					\draw[-,opacity = .5] (5,5-\x/2) -- (2.5,7.5-\x/2 );
					\draw[-,opacity = .2] (2.5,7.5-\x/2 ) -- (5-\x/2,5);}{}
				\foreach \y in {0,...,10}{
					\ifthenelse{\cnttest{\y}<{10} \AND \cnttest{\x}>{5} \AND \cnttest{\y}>{0}\AND \cnttest{\x}<{10}}
					{\node[circle,fill=black,scale=.3] (a) at (\x/2,\y/2) {};}
					{
						\ifthenelse{\cnttest{\y}>{0}\AND \cnttest{\y}<{5} \AND \x = 5}
						{\node[circle,fill=black,scale=.3] (a) at (\x/2,\y/2) {};}
						{\node[circle,fill=red,scale=.3] (a) at (\x/2,\y/2) {};}
					}
				}
				
			}
			
		}

		\end{tikzpicture}
	\end{minipage}
	\caption{Superimposition of $\Xi_n$ onto the triangle $T$ and an $L$ shape. the subdivision mesh is referred to the $P_1$ finite elements method.}
	\label{regular_grid_Triangle_T_FE}
\end{figure}

This does not happen when dealing with more complicated domains $\om$, as shown in \autoref{general_domain_FE}. In fact we can see that, for example, on a curvilinear shape, the points of $\Xi_n$ are not enough to approximate the boundary $\partial\om$.

\begin{figure}[H]
	\centering
	\begin{tikzpicture}
	
	\draw (2.5,5) arc (90:270:2.5cm) -- (5,0) -- (2.5,5);
	
	\node[circle] (a) at (5,0) [right] {$(1,0)$} ;
	\node[circle] (a) at (0,0) [below] {$(0,0)$} ;
	\node[circle] (a) at (0,5) [left] {$(0,1)$} ;
	
	\foreach \x in {0,...,10}{
		\foreach \y in {0,...,10}{
			\ifthenelse{\cnttest{\x}<{5} \OR \cnttest{\x}={5}\AND \cnttest{(\x-5)*(\x-5) + (\y-5)*(\y-5)}<{25}}
			{\node[circle,fill=black,scale=.3] (a) at (\x/2,\y/2) {};}
			{\ifthenelse{ \cnttest{\x}<{5} \OR \cnttest{\x}={5}\AND \cnttest{(\x-5)*(\x-5) + (\y-5)*(\y-5)}={25}}
				{\node[circle,fill=red,scale=.3](a) at (\x/2,\y/2) {};}
				{\ifthenelse{ \cnttest{\x}>{5} \AND \cnttest{\y+2*\x}<{20} \AND \cnttest{\y}>{0} }
					{\node[circle,fill=black,scale=.3] (a) at (\x/2,\y/2) {};}				
					{\ifthenelse{  \cnttest{\y+2*\x}={20} \OR \cnttest{\y}={0} \AND \cnttest{\x}>{5} }
						{\node[circle,fill=red,scale=.3](a) at (\x/2,\y/2) {};}
						{\node[circle,fill=black,fill opacity=.2,scale=.3](a) at (\x/2,\y/2) {};}							
	}}}}}
	
	\foreach \x in {2,...,8}{
		\node[circle,fill=red,scale=.3](a) at (.5,\x/2) {};
	}
	\node[circle,fill=red,scale=.3](a) at (1,4) {};
	\foreach \x in {3,4}{
		\foreach \y in {1,9}{
			\node[circle,fill=red,scale=.3](a) at (\x/2,\y/2) {};
	}}
	\node[circle,fill=red,scale=.3](a) at (2.5,4.5) {};
	\node[circle,fill=red,scale=.3](a) at (3,3.5) {};
	\node[circle,fill=red,scale=.3](a) at (3.5,2.5) {};
	\node[circle,fill=red,scale=.3](a) at (4,1.5) {};
	\node[circle,fill=red,scale=.3](a) at (4.5,.5) {};
	
	\foreach \x in {5,6,...,10}{
		\draw[-,opacity=.2] (\x/2,10-\x)--(\x/2,5);
		\draw[-] (\x /2,0)--(\x/2,10-\x);
	}			
	\draw[-,opacity=.2] (0,0)--(0,5);
	\draw[-,opacity=.2] (.5,0)--(.5,1);
	\draw[-,opacity=.2] (.5,4)--(.5,5);
	
	\draw[-] (.5,1)--(.5,4);
	
	\draw[-] (1,.5)--(.5,1);
	\foreach \x in {2,3,4}{
		\draw[-,opacity=.2] (\x/2,0)--(\x/2,0.5);
		\draw[-,opacity=.2] (\x/2,4.5)--(\x/2,5);
		\draw[-] (\x /2,.5)--(\x/2,4.5);
	}

	\foreach \y in {0,2,...,8,10}{
		\draw[-,opacity=.2] (5-\y/4,\y/2)--(5,\y/2);
		\draw[-] (2.5,\y/2)--(5-\y/4,\y/2);
	}			
	\foreach \y in {1,3,...,7,9}{
		\draw[-,opacity=.2] (5-\y/4 - 1/4,\y/2)--(5,\y/2);
		\draw[-] (2.5,\y/2)--(5-\y/4 - 1/4,\y/2);
	}

	\draw[-,opacity=.2] (0,0)--(2.5,0);
	\draw[-,opacity=.2] (0,5)--(2.5,5);
	\draw[-,opacity=.2] (0,0.5)--(1,.5);
	\draw[-,opacity=.2] (0,4.5)--(1,4.5);
	
	\draw[-] (1,0.5)--(2.5,0.5);
	\draw[-] (1,4.5)--(2.5,4.5);
	\draw[-] (0,2.5)--(2.5,2.5);

	\foreach \y in {2,3,4}{
		\draw[-,opacity=.2] (0,\y/2)--(0.5,\y/2);
		\draw[-,opacity=.2] (0,5-\y/2)--(0.5,5-\y/2);
		\draw[-] (.5,\y/2)--(2.5,\y /2);
		\draw[-] (.5,5-\y/2)--(2.5,5-\y/2);
		
	}

	\draw[-] (1.5,0.5)--(.5,1.5);
	\draw[-] (0,2.5)--(2.5,0);
	
	\draw[-] (5,0)--(1,4);
	\draw[-] (4.5,1)--(1,4.5);
	\draw[-] (4,2)--(1.5,4.5);
	\draw[-] (3.5,3)--(2,4.5);
	\draw[-] (3,4)--(2.5,4.5);
	
	\foreach \x in {0,1,2,3}{
		\draw[-,opacity=.2] (\x/2,0)--(0,\x/2);
	}
	\foreach \x in {1,2,3,4}{
		\draw[-,opacity=.2] (\x/2,5)--(\x/2+.5,4.5);
	}
	\draw[-,opacity=.2] (0,5)--(1,4);
	
	\foreach \y in {4,6,7,8,9}{
		\draw[-,opacity=.2] (0,\y/2)--(0.5,\y/2-.5);
	}	
	\draw[-,opacity=.2] (2,0)--(1.5,.5);
	
	\foreach \y in {1,...,5}{
		\draw[-,opacity=.2] (5,\y/2)--(5-\y/2 ,\y);
	}	
	\foreach \x in {6,7,8,9}{
		\draw[-,opacity=.2] (\x/2,5)--(5,\x /2);
		\draw[-] (\x/2,0)--(0.5,\x/2-.5);
	}
	
	\end{tikzpicture}
	\caption{Example of a general domain $\om$ and induced mesh.}
	\label{general_domain_FE}
\end{figure}
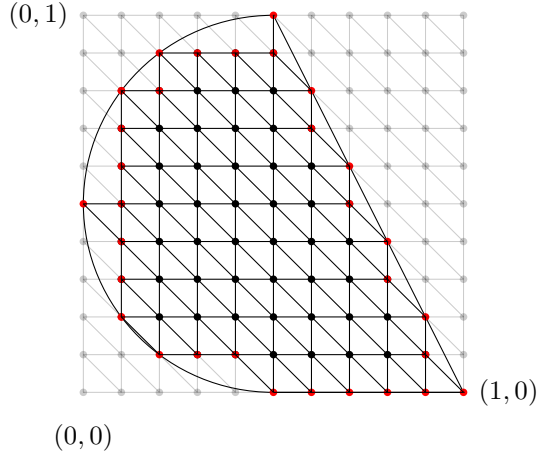

When we work on a closed set $\om\cu [0,1]^2$ with $\mu(\partial \om)=0$, we focus on the points $p$ such that $T_p$ is contained in $\om$, so we call
\[
\Xi_n(\om) := \set{p\in\Xi_n | T_p\cu \om}.
\]
We look for a function $u$ that is a linear combination of the $\psi_{p,n}$ such that (\ref{weak.problem.dFE.triangle}) is satisfied for every $w=\psi_{p,n}$. If we substitute $u=\sum_{p\in \Xi_n(\om)} u_p \psi_{p,n}$ and $w=\psi_{q,n}$ into (\ref{weak.problem.dFE.triangle}), then  we obtain the system
\begin{equation}
\label{FE}
\sum_{p\in \Xi_n(\om)} s_{q,p}u_p=f_q, \qquad 
s_{q,p}=\int_{\om^\circ} (\nabla \psi_{p,n})^TA\nabla \psi_{q,n} + (\nabla \psi_{p,n})^T\bm b \psi_{q,n} + c\psi_{p,n}\psi_{q,n}, \qquad f_q  = \int_{\om^\circ} f\psi_{q,n}
\end{equation}
for every $q\in \Xi_n(\om)$. We call $S_\nn$ the resulting matrix with entries $s_{p,q}$ for every $p,q\in \Xi_n(\om)$, where the nodes are sorted in lexicographic order.  We can notice that $p\in \Xi_n(\om)\implies p\in \om^\circ\cap \Xi_n$, even if the converse is not always true, so 
\[
|\Xi_n(\om)|\le  d_n^{\om^\circ} = O(n^2)
\]
where $|\Xi_n(\om)|$ is the size of the matrix $S_\nn$. It leads to solve the system
\[
S_\nn \bm u = \bm f. 
\]

\begin{remark}
	A different boundary condition does not change the stiffness matrix, so the analysis is the same if we impose, for example, $u = g$ on $T_D$ and $\partial u/\partial n = h$ on $T_N$ where $\partial T = T_D \coprod T_N$. 
\end{remark}

\subsection{Case on the Square}

When $\om =[0,1]^2$, we already know that, under suitable hypotheses on the regularity of the coefficients, the sequence of stiffness matrices $\serie S$ described in (\ref{FE}) is actually a multilevel GLT sequence, for which we can compute GLT and spectral symbol. 
Here we prove that the same holds when $A,\bm b,c$ are just $L^1$ functions.

\begin{theorem}\label{FE_symbol_square}
	We call $B$ the $3\times 2$ matrix
	\[
	B = \begin{pmatrix}
	1 & 1\\
	1 & 0\\
	0 & 1
	\end{pmatrix}
	\]
	and we indicate with $B_1,B_2,B_3$ its rows.
	Given  $L^1$ functions $A:(0,1)^2\to \f C^{2\times 2}$, $\bm b:(0,1)^2 \to \f C^2$ and $c:(0,1)^2 \to \f C$, we have that the sequence $\ms{S}$ is a multilevel GLT sequence with symbol $k(x,\theta)$, where
	\begin{align}\label{k_FE}
	k(x,\theta) =& r_{0,0}(x) +  r_{0,1}(x)\exp(-i\theta_2)
	+  r_{1,0}(x)\exp(-i\theta_1)+  r_{-1,0}(x)\exp(i\theta_1) \\&+  \nonumber r_{0,-1}(x)\exp(i\theta_2)
	+ r_{1,-1}(x)\exp(-i\theta_1+i\theta_2)
	+  r_{-1,1}(x)\exp(i\theta_1 - i\theta_2),
	\end{align}
	\begin{align}
	\label{r_FE}r_{0,0} 
	=
	B_1A(B_1)^T
	+&B_2A(B_2)^T
	+B_3A(B_3)^T
	,\\
	\nonumber r_{0,1} 
	=
	-\frac 12B_3A(B_1)^T
	-\frac 12B_1A(B_3)^T
	,&\qquad
	\nonumber r_{0,-1} =
	-\frac 12B_3A(B_1)^T
	-\frac 12B_1A(B_3)^T
	,\\
	\nonumber r_{1,-1} =
	\frac 12B_2A(B_3)^T
	+\frac 12B_3A(B_2)^T,&\qquad
	r_{-1,1} =
	\frac 12B_2A(B_3)^T
	+\frac 12B_3A(B_2)^T
	,\\
	\nonumber r_{-1,0} = 
	-\frac 12B_1A(B_2)^T
	-\frac 12B_2A(B_1)^T,&\qquad
	r_{1,0} 
	=
	-\frac 12B_1A(B_2)^T
	-\frac 12B_2A(B_1)^T.
	\end{align}
	If $A$ is also Hermitian for every $x\in (0,1)^2$, then the sequence $\ms{ S}$ has  $k(x,\theta)$ as spectral symbol.
\end{theorem}
\begin{proof}
	We split the matrix $ S_\nn$ into $ P_\nn +  Z_\nn$, where 
	\[
	( P_\nn)_{p,q} = \int_{(0,1)^2} (\nabla \psi_{p,n})^TA\nabla \psi_{q,n}, \qquad  ( Z_\nn)_{p,q} = \int_{(0,1)^2}  (\nabla \psi_{p,n})^T\bm b \psi_{q,n} + c\psi_{p,n}\psi_{q,n}
	\] 
	and we prove that $\ms{ P}\GLT k(x,\theta)$ and $\ms{ Z}$ is zero distributed. \\
	
	Notice that $\psi_{p}$ is supported on $T_p$, so $ ( S_\nn)_{p,q}, ( P_\nn)_{p,q}, ( Z_\nn)_{p,q}$ are different from zero only when $q$ is one of the 6 neighbours of $p$ or $p$ itself, that is $q\in N(p)$. Moreover, every $\psi_{p,n}$ is nonnegative and less than 1, and each component of $\nabla \psi_{p,n}$ is bounded by $1/h$ in absolute value.

	Notice that the area of $T_{p}$ is $3h^2$ for every $p$. Moreover, the functions $b_1,b_2,c$ are $L^1$, so for every $\ve>0$ there exists a $\delta>0$ such that 
	\[
	\mu(U)\le \delta\implies \int_U |b_1|+|b_2|+|c|\le \ve.
	\]
	Notice that every triangle $T_{(i)}$ of the triangulation $\mc T_n$ is inside $T_p$ for at most 3 different points $p$, that are its vertices, and if $3h^2\le \delta$,  we get\\
	\begin{align}\label{FE_Z}
	\nonumber	\| Z_\nn\|_2^2  &=  
	\sum_{p,q\in (0,1)^2\cap \Xi_n} |( Z_\nn)_{p,q}|^2 	
	=
	\sum_{p,q\in (0,1)^2\cap \Xi_n} \left|\int_{(0,1)^2}  (\nabla \psi_{p,n})^T\bm b \psi_{q,n} + c\psi_{p,n}\psi_{q,n}\right|^2 \\
	&\nonumber\le
	\sum_{p\in (0,1)^2\cap \Xi_n} 
	\sum_{q\in N(p)} \left[\int_{T_{p}}  |(\nabla \psi_{p,n})^T\bm b| \psi_{q,n} + |c|\psi_{p,n}\psi_{q,n}\right]^2
	\\
	&\nonumber\le
	\sum_{p\in (0,1)^2\cap \Xi_n} 
	\sum_{q\in N(p)} \left[\int_{T_{p}}  \frac {|b_1| + |b_2|}h+ |c| \right]^2 \le 
	\frac 1{h^2}\sum_{p\in (0,1)^2\cap \Xi_n} 
	\sum_{q\in N(p)} \left[\int_{T_{p}}  |b_1| + |b_2|+ |c| \right]^2
	\\
	&	\nonumber\le 
	\frac 7{h^2}\sum_{p\in (0,1)^2\cap \Xi_n} 
	\left[\int_{T_{p}}  |b_1| + |b_2|+ |c| \right]\ve\\
	&\nonumber\le 
	\frac 7{h^2}\ve \sum_{T_{(i)}\in \mc T_n} 
	3\left[\int_{T_{(i)}}  |b_1| + |b_2|+ |c| \right]\\&
	\le 21 \frac{\ve}{h^2}( \|b_1\|_1 + \|b_2\|_1 +\|c\|_1 ).
	\end{align}
	Since we can take $\ve$ arbitrarily small as $n$ tends to infinity, we infer that  $n^{-1}\| Z_\nn\|_2\to 0$, so we can use \textbf{Z2} and conclude that $\ms{ Z}$ is zero-distributed.\\
	
	Let us analyse now the matrix $ P_\nn$.
	The elements of $ P_\nn$ on the row associated to $p=x_\jj$ are different from zero only when $q\in N(p)$. 
	Call $t_{p,a,b} = (P_\nn)_{p,p+a\bm e_1 + b\bm e_2}$, and a computation shows that 
	\begin{align*}
	h^2t_{p,0,0} &
	=
	\int_{T_{1,p}\cup T_{4,p}}B_1A(B_1)^T
	+\int_{T_{2,p}\cup T_{5,p}}B_2A(B_2)^T
	+\int_{T_{3,p}\cup T_{6,p}}B_3A(B_3)^T
	,\\
	h^2t_{p,0,1} &
	=
	-\int_{T_{6,p}}B_3A(B_1)^T
	-\int_{T_{1,p}}B_1A(B_3)^T
	,\\
	h^2t_{p,1,0} &
	=
	-\int_{T_{1,p}}B_1A(B_2)^T
	-\int_{T_{2,p}}B_2A(B_1)^T,\\
	h^2t_{p,1,-1} &=
	\int_{T_{2,p}}B_2A(B_3)^T
	+\int_{T_{3,p}}B_3A(B_2)^T,\\
	h^2t_{p,0,-1} &=
	-\int_{T_{3,p}}B_3A(B_1)^T
	-\int_{T_{4,p}}B_1A(B_3)^T,\\
	h^2t_{p,-1,0} &= 
	-\int_{T_{4,p}}B_1A(B_2)^T
	-\int_{T_{5,p}}B_2A(B_1)^T,\\
	h^2t_{p,-1,1} &=
	\int_{T_{5,p}}B_2A(B_3)^T
	+\int_{T_{6,p}}B_3A(B_2)^T,
	\end{align*}
	and $t_{p,a,b} = 0$ for every other $a,b$.\\
	
	Assume that $A$ is a continuous function, so that there exists a modulus of continuity $\omega_A$ defined as
	\[
	\omega_A(\delta) = \max_{i,j}\sup_{p,q:|p-q|\le \delta}
	\left| (
	A(p) - A(q))_{i,j}
	\right| 
	\] 
	and such that $\lim_{\delta\to 0} \omega_A(\delta) = 0$. Let us define a 2-level GLT sequence $\serie{ G}$ as 
	\begin{align*}
	G_\nn &= D_\nn(r_{0,0})T_\nn(1) +  D_\nn(r_{0,1})T_\nn(\exp(-i\theta_2))
	+  D_\nn(r_{1,0})T_\nn(\exp(-i\theta_1)) +  D_\nn(r_{-1,0})T_\nn(\exp(i\theta_1)) \\&
	+  D_\nn(r_{0,-1})T_\nn(\exp(i\theta_2))
	+D_\nn(r_{1,-1})T_\nn(\exp(-i\theta_1+i\theta_2))
	+  D_\nn(r_{-1,1})T_\nn(\exp(i\theta_1 - i\theta_2)),
	\end{align*}
	with symbol $k(x,\theta)$. 
	The elements of $ P_\nn -  G_\nn$ on the row associated to $p=x_\jj$ are different from zero only when $q\in N(p)$.  If we call $z_{p,a,b} =  (P_\nn)_{p,p+\bm e_1 +b\bm e_2}
	-
	(Q_\nn)_{p,p+\bm e_1 +b\bm e_2}$, then
	\begin{align*}
	|z_{p,0,0}| 
	\le&
	\left|B_1A(p)(B_1)^T - \frac 1{h^2}\int_{T_{1,p}\cup T_{4,p}}B_1A(B_1)^T\right|
	+
	\left|B_2A(p)(B_2)^T - \frac 1{h^2}\int_{T_{2,p}\cup T_{5,p}}B_2A(B_2)^T\right|\\
	& +
	\left|B_3A(p)(B_3)^T -
	\frac 1{h^2}\int_{T_{3,p}\cup T_{6,p}}B_3A(B_3)^T\right|\\
	=&
	\left|\frac 1{h^2}\int_{T_{1,p}\cup T_{4,p}}B_1(A(p)-A(x))(B_1)^Tdx\right|
	+
	\left|\frac 1{h^2}\int_{T_{2,p}\cup T_{5,p}}B_2(A(p)-A(x))(B_2)^Tdx\right|\\
	& +
	\left|
	\frac 1{h^2}\int_{T_{3,p}\cup T_{6,p}}B_3(A(p)-A(x))(B_3)^Tdx\right|
	\\
	& \le 4\omega_A(h\sqrt 2) + \omega_A(h\sqrt 2) + \omega_A(h\sqrt 2)   = 6\omega_A(h\sqrt 2),
	\end{align*}
	\begin{align*}
	|z_{p,0,1}| 
	\le&
	\left|\frac 12B_3A(p)(B_1)^T - \frac 1{h^2}\int_{T_{6,p}}B_3A(B_1)^T\right|
	+
	\left|\frac 12B_1A(p)(B_3)^T - \frac 1{h^2}\int_{T_{1,p}}B_1A(B_3)^T\right|\\
	=&
	\left|\frac 1{h^2}\int_{T_{6,p}}B_3(A(p)-A(x))(B_1)^Tdx\right|
	+
	\left|\frac 1{h^2}\int_{T_{1,p}}B_1(A(p)-A(x))(B_3)^Tdx\right|
	\\
	& \le \omega_A(h\sqrt 2) + \omega_A(h\sqrt 2)   = 2\omega_A(h\sqrt 2),
	\end{align*}
	and analogous computations show that $ |z_{p,1,0}|, |z_{p,0,-1}|, |z_{p,-1,0}|$ are also bounded by $2\omega_A(h\sqrt 2)$. Moreover,
	\begin{align*}
	|z_{p,1,-1}| 
	\le&
	\left|\frac 12B_2A(p)(B_3)^T - \frac 1{h^2}\int_{T_{2,p}}B_2A(B_3)^T\right|
	+
	\left|\frac 12B_3A(p)(B_2)^T - \frac 1{h^2}\int_{T_{3,p}}B_3A(B_2)^T\right|\\
	=&
	\left|\frac 1{h^2}\int_{T_{2,p}}B_2(A(p)-A(x))(B_3)^Tdx\right|
	+
	\left|\frac 1{h^2}\int_{T_{3,p}}B_3(A(p)-A(x))(B_2)^Tdx\right|
	\\
	& \le \frac 12\omega_A(h\sqrt 2) + \frac 12\omega_A(h\sqrt 2)   = \omega_A(h\sqrt 2),
	\end{align*}
	and a similar argument shows that $ |z_{p,-1,1}|$ is also bounded by $\omega_A(h\sqrt 2)$. Since every row of $ P_\nn -  G_\nn$ has at most 7 non-zero elements and they are all bounded in absolute value by $6 \omega_A(h\sqrt 2)$, we conclude that
	\[
	\| P_\nn -  G_\nn\|_2
	\le \sqrt{
		7n^2 \cdot 36 \omega_A(h\sqrt 2)^2
	}
	\le 18 n \omega_A(h\sqrt 2)^2 = o(n)
	\]
	and using again  \textbf{Z2}, we obtain that $ P_\nn -  G_\nn$ is zero-distributed. Since $\ms{ G}$ has GLT symbol $k(x,\theta)$,
	we conclude that
	\[
	\ms{ S} = \ms{ G} + \{ P_\nn -  G_\nn \}_n + \ms{ Z} \GLT k(x,\theta).
	\]
	
	$ $\\
	
	If we now assume that $A$ is an $L^1$ function, then we can find a sequence $A_m$ of continuous functions such that $\|A-A_m\|_1 \le 2^{-m}$, where 
	\[
	\|C\|_1 = \sum_{i,j} \|c_{i,j}\|_1 = \int_{(0,1)^2}B_1|C|(B_1)^T. 
	\]
	If we define $r_{a,b,m}$ like in (\ref{r_FE}) with $A_m$ instead of $A$, and $k_m(x,\theta)$ like in (\ref{k_FE}) with $r_{a,b,m}$ instead of $r_{a,b}$, then we get $k_m\to k$ in $L^1$. Moreover, if $\ms{ S^{(m)}}$ is defined as above, but with $A_m$ instead of $A$, then from the previous analysis, we know that $\ms{ S^{(m)}}\GLT k_m$. The difference
	\[
	\ms{ S^{(m)}} - \ms{ S} = 
	\ms{ P^{(m)}}- \ms{ P}+
	\ms{ Z^{(m)}} - \ms{ Z}
	\]
	presents two zero-distributed sequences $\ms{ Z^{(m)}}$ and $\ms{ Z}$, so we need to analyse the other two sequences. Notice that for every measurable set $U\cu [0,1]^2$ and every indices $i,j$ we know that 
	\[
	\left| \int_{U}B_iA(B_j)^T 
	-
	\int_{U}B_iA_m(B_j)^T 
	\right|
	\le B_1\left[
	\int_{U}
	\left|A
	-
	A_m 
	\right|\right]
	(B_1)^T,
	\]
	but $A-A_m$ is also $L^1$, so given $\ve$ there exists a $\delta$ such that $\mu(U)<\delta$ implies that 
	\[B_1
	\int_{U}
	\left|A
	-
	A_m 
	\right|(B_1)^T\le \ve.
	\]
	If $\mu(T_p)=3h^2\le \delta$, 
	then
	we can bound the 1 Schatten norm of $ P_\nn^{(m)} - P_\nn$ by the sum of the absolute values of their elements, so 
	\begin{align*}
	\| P_\nn^{(m)} - P_\nn\|_1 &\le
	\sum_{p,q\in (0,1)^2\cap \Xi_n} |( P_\nn^{(m)} - P_\nn)_{p,q}| \\
	&\le 
	\sum_{p\in (0,1)^2\cap \Xi_n} 
	\sum_{q\in N(p)} |( P_\nn^{(m)} - P_\nn)_{p,q}|\\
	&\le \frac{1}{h^2}
	\sum_{p\in (0,1)^2\cap \Xi_n}
	\sum_{q\in N(p)} 
	6
	B_1 \int_{T_{p}}\left|A
	-
	A_m 
	\right| (B_1)^T
	\\
	&\le \frac{42}{h^2}
	\sum_{p\in (0,1)^2\cap \Xi_n}
	B_1 \int_{T_{p}}\left|A
	-
	A_m 
	\right| (B_1)^T
	\\
	&\le 3\frac{42}{h^2}
	\|A-A_m\|_1.
	\end{align*}
	Using   \textbf{ACS 4}, we obtain that $\{ P_\nn^{(m)}\}_n\acs\{ P_\nn\}_n$ and $\{ S_\nn^{(m)}\}_n\acs\{ S_\nn\}_n$. We conclude that $\ms { S}\GLT k$.\\
	
	When $A$ is Hermitian, we can prove that $ P_\nn$ is Hermitian. In fact
	\[
	( P_\nn)_{p,q} = \int_{(0,1)^2} (\nabla \psi_{p,n})^TA\nabla \psi_{q,n} =
	\ol{ \int_{(0,1)^2} (\nabla \psi_{p,n})^T\ol A\nabla \psi_{q,n}} =
	\ol{ \int_{(0,1)^2} (\nabla \psi_{q,n})^T\ol A\nabla \psi_{p,n}} =
	\ol {( P_\nn)_{q,p}}.
	\]
	Since $\ms{ S} = \ms{ P} +\ms{ Z}$ and from (\ref{FE_Z}), we know that $\|\wt{Z}_\nn\|_2 = o(n)$, we can apply \textbf{GLT 2} and conclude that $\ms { S}\sim_\lambda k$.
\end{proof}

\subsection{Problem on Sub-domains}
Let us now consider a closed Peano-Jordan measurable set $\om\cu[0,1]^2$ with positive measure.	
Consider the problem (\ref{weak.problem.dFE.triangle}) on $\om$, where now $A,\bm b,c$ are $L^1$ functions defined on $\om$. 
When we apply a $P_1$ discretization. The resulting  matrices   form a sequence equivalent to a  reduced GLT sequence that descends from the square case. In particular, we can prove the following theorem.
\begin{theorem}\label{FE_P1_subdomain:symbol}
	Given a closed Peano-Jordan measurable set $\Omega\cu[0,1]^2$ with positive measure. 
	Let $\wt A$, $\wt {\bm b}$ and $\wt c$ be extensions of $A$, $\bm b$ and $c$ to $(0,1)^2$, obtained by setting $\wt a_{i,j}(z) = \wt{b_j}(z)=\wt c(z)=0$ outside $\om$ for every $i,j$. Moreover, let $\wt k$ be the symbol described in Theorem \ref{FE_symbol_square} referred to the problem with coefficients $\wt A$, $\wt{\bm b}$, $\wt c$, and denote $k = \wt k|_{\om^\circ}$.	
	If $S^{\om}_{\nn}$ is the matrix resulting from the $P_1$ discretization using the grid $\Xi_{n}( \om)$, then
	\[
	\{ S^{\om}_{\nn}\}_n\svs k
	\]
	and if $A$ is Hermitian for every $x\in \om$, then $k$ is also a spectral symbol for $\{ S^{\om}_{\nn}\}_n$.
\end{theorem}
\begin{proof}
	Let $ S_{\nn}$ be the matrix resulting from the $P_1$ discretization of the problem with coefficients $\wt A$, $\wt{\bm b}$, $\wt c$ on the square $[0,1]^2$ using the grid $\Xi_{n}$.
	We want to show that $R_{\Xi_n(\om)}( S_{\nn}) =  S^{\om}_{\nn}$, that is, for every pair of points $(p,q)$ in $\Xi_{n}(\om)$, we prove $( S_{\nn})_{p,q} = (S^{\om}_{\nn})_{p,q}$.
	From (\ref{FE}), the equations for the two quantities are
	\begin{align*}
	(S_\nn)_{p,q} &= \int_{(0,1)^2} (\nabla \psi_{p,n})^T\wt A\nabla \psi_{q,n} + (\nabla \psi_{p,n})^T\wt{\bm b} \psi_{q,n} + \wt c\psi_{p,n}\psi_{q,n},	\\
	(S_\nn^\om)_{p,q} &= \int_{\om^\circ} (\nabla \psi_{p,n})^T A\nabla \psi_{q,n} + (\nabla \psi_{p,n})^T{\bm b} \psi_{q,n} +  c\psi_{p,n}\psi_{q,n},
	\end{align*}
	but $p\in\Xi_n(\om)$ so $T_p^\circ \cu \om^\circ$ and therefore the two quantities are the same since $A,\bm b$ and $c$ coincide with   $\wt A$, $\wt{\bm b}$ and $\wt c$ on $\om$. In this case, it may happen that $\Xi_n(\om) \subsetneq \Xi_n\cap \om^\circ$ since $\om$ may not be convex, but the two sets are actually almost the same. In fact,
	\[
	E_n := (\Xi_n\cap \om^\circ) \setminus \Xi_n(\om) = \set{ p\in \Xi_n\cap \om^\circ | T_p\not\cu \om  },
	\] 
	so any point $p\in E_n$ is at distance at most $h_n=1/(n+1)$ from the boundary $\partial \om$, and using \autoref{preliminaries_near_border_points_chi}, we conclude
	\[
	E_n \cu \set{ p\in \Xi_n | d(p,\partial \om)\le h_n  } \implies |E_n| \le s_n^{K_{h_n}} = o(N(\nn)).
	\]
	As a consequence,
	\[
	s_n^{\om^\circ \triangle \Xi_n(\om)} = |\set{p\in \Xi_n\cap \om^\circ| p \not\in \Xi_n(\om) }| =|E_n| =o(N(\nn))
	\]
	and \autoref{Different_Grids} assures us that it is sufficient to prove the thesis for $R_{\om^\circ}(S_\nn)$.

	Using the definition of reduced GLT, we can  affirm that
	\[
	\{R_{\om^\circ}(S_\nn)\}_n\GLT^{\om^\circ} k \implies \{ R_{\om^\circ}(S_\nn)\}_n\svs k.
	\]
	
	If we now assume that $A$ is an Hermitian matrix for every $x\in \om$, then automatically also $\wt A$ is Hermitian for every $x$, since it is equal to $A$ or it is the zero matrix. From the proof of \autoref{FE_symbol_square}, we know that   $ S_{\nn} =   P_{\nn} +  Z_{\nn}$, where $ P_{\nn}$ is Hermitian and $\| Z_{\nn}\|_2 = o(n)$. If we call $P^{\om^\circ}_{\nn} = R_{\om^\circ}( P_{\nn})$ and $Z^{\om^\circ}_{\nn} = R_{\om^\circ}( Z_{\nn})$
	then we find that $R_{\om^\circ}(S_\nn) = P^{\om^\circ}_{\nn} + Z^{\om^\circ}_{\nn}$, $P^{\om^\circ}_{\nn}$ is Hermitian and for Lemmas \ref{ker_reduction}, \ref{norm_reduction} and \ref{dimension_chi},
	\[
	\|Z^{\om^\circ}_{\nn}\|_2 = \|R_{\om^\circ}( Z_{\nn})\|_2 \le  \| Z_{\nn}\|_2 = o(n) \implies  \|Z^{\om^\circ}_{\nn}\|_2 = o\left(\sqrt{s_n^\om}\right).
	\]
	Notice that $\{P_{\nn}^{\om^\circ}\}_k\es k$, so we can use \textbf{GLT 2}, and conclude that  
	\[
	\{R_{\om^\circ}(S_\nn)\}_k\sim_{\lambda} k.
	\]
	
\end{proof}
Notice that $k(x,\theta)$ has the same form described in (\ref{k_FE}), (\ref{r_FE}), where $A$ is now defined only on $\om$. 

\subsection{P1 on Non-Regular Grids}

When the domain $\om$ is compact, but presents an irregular boundary, or when we want to focus the discretization to particular points in the domain, the adopted grids  are usually adapted to the problem geometry.  We can find examples of such grids and relative spectral analyses already in \cite{GLT-bookII} for $\om=[0,1]^d$ and in \cite{Beck} for more general domains. 
In both cases, the grids taken into account were produced starting from a regular grid and by applying an invertible function $\phi$. For clarity sake, we start from a smooth ($C^1$) embedding $\vf$ that maps $\om$ into $[0,1]^d$, and if $D=\vf(\om)$, then we  call the inverse $\phi := \vf^{-1}:D\to \om$. Notice that $\vf$ is in particular a closed locally Lipschitz map, so $D$ is a  compact set in $[0,1]^d$ and it is still Peano-Jordan measurable. We can thus induce a discretization grid on $\om$ given by $\phi(D\cap \Xi_n)$ for every $n$. 

{
}

\begin{figure}[H]
	\centering
	\begin{tikzpicture}
	\foreach \x in {0,...,10}{
		\foreach \y in {0,...,10}{

			\ifthenelse{\cnttest{(\x-5)*(\x-5) + (\y-5)*(\y-5)}<{25}}
			{\node[circle,fill=black,scale=.3] (b-\x-\y) at (6+\x/2,\y/2) {};
				\node[circle,fill=black,scale=.3] (a-\x-\y) at ({6*(\x+1)/sqrt(40+(\x+1)*(\x+1)+\y*\y)-2}, {6*\y/sqrt(40+(\x+1)*(\x+1)+\y*\y)}) {};}
			{\ifthenelse{ \cnttest{(\x-5)*(\x-5) + (\y-5)*(\y-5)}={25}}
				{\node[circle,fill=red,scale=.3](b-\x-\y) at (6+\x/2,\y/2) {};
					\node[circle,fill=red,scale=.3] (a-\x-\y) at ({6*(\x+1)/sqrt(40+(\x+1)*(\x+1)+\y*\y)-2}, {6*\y/sqrt(40+(\x+1)*(\x+1)+\y*\y)}) {};}
				{\node[circle,fill=black,fill opacity=.2,scale=.3](b-\x-\y) at (6+\x/2,\y/2) {};	
					\node[circle,fill opacity=0,scale=.3] (a-\x-\y) at ({6*(\x+1)/sqrt(40+(\x+1)*(\x+1)+\y*\y)-2}, {6*\y/sqrt(40+(\x+1)*(\x+1)+\y*\y)}) {};}							
			}

	}}
	
	
	\foreach \X in {0,...,9}{
		\foreach \Y in {0,...,10}{
			\ifthenelse{\NOT \cnttest{(\X-5)*(\X-5) + (\Y-5)*(\Y-5)}>{25} \AND \NOT \cnttest{(\X-4)*(\X-4) + (\Y-5)*(\Y-5)}>{25}}
			{
				\draw [domain=0:1, samples=10] 	plot (  {seg_x(\X,\Y,\X+1,\Y,\x)} , {seg_y(\X,\Y,\X+1,\Y,\x)} );
				
				\draw let \n1={int(\X+1)} in (b-\X-\Y) -- (b-\n1-\Y);}
			{\draw[opacity=.2] let \n1={int(\X+1)} in (b-\X-\Y) -- (b-\n1-\Y); 
			}
	}}
	
	\foreach \X in {0,...,10}{
		\foreach \Y in {0,...,9}{
			\ifthenelse{\NOT \cnttest{(\X-5)*(\X-5) + (\Y-5)*(\Y-5)}>{25} \AND \NOT \cnttest{(\X-5)*(\X-5) + (\Y-4)*(\Y-4)}>{25}}
			{
				\draw [domain=0:1, samples=10] 	plot (  {seg_x(\X,\Y,\X,\Y+1,\x)} , {seg_y(\X,\Y,\X,\Y+1,\x)} );
				\draw let \n1={int(\Y+1)} in (b-\X-\Y) -- (b-\X-\n1);}
			{\draw[opacity=.2] let \n1={int(\Y+1)} in (b-\X-\Y) -- (b-\X-\n1);
			}
			
	}}
	
	\foreach \X in {1,...,10}{
		\foreach \Y in {0,...,9}{
			\ifthenelse{\NOT \cnttest{(\X-5)*(\X-5) + (\Y-5)*(\Y-5)}>{25} \AND \NOT \cnttest{(\X-6)*(\X-6) + (\Y-4)*(\Y-4)}>{25}}
			{
				\draw [domain=0:1, samples=10] 	plot (  {seg_x(\X,\Y,\X-1,\Y+1,\x)} , {seg_y(\X,\Y,\X-1,\Y+1,\x)} );
				\draw let \n1={int(\Y+1)},\n2={int(\X-1)} in (b-\X-\Y) -- (b-\n2-\n1);}
			{\draw[opacity=.2] let \n1={int(\Y+1)},\n2={int(\X-1)} in (b-\X-\Y) -- (b-\n2-\n1);
			}			
	}}
	
	\draw [red,thick] (11,2.5) arc (0:360:2.5cm) ;
	
	\draw [red, thick,  domain=0:360, samples=80] 
	plot ({6*(5*cos(\x) + 5+1)/sqrt(40+(5*cos(\x) + 5+1)^2+(5*sin(\x) + 5)^2)-2}, {6*(5*sin(\x) + 5)/sqrt(40+(5*cos(\x) + 5+1)^2+(5*sin(\x) + 5)^2)} );
	
	\draw [->] (3.5,3) to [out=30,in=150] (5.5,3);
	\draw [->] (5.5,2) to [out=210,in=330] (3.5,2);
	
	\node () at (4.5,3.5) {$\vf$};
	\node () at (4.5,1.5) {$\phi$};

	\node[circle,fill=red,scale=.3](c) at (a-4-1) {};
	\node[circle,fill=red,scale=.3](c) at (a-5-1) {};
	\node[circle,fill=red,scale=.3](c) at (a-6-1) {};
	\node[circle,fill=red,scale=.3](c) at (a-7-1) {};
	\node[circle,fill=red,scale=.3](c) at (a-3-1) {};
	\node[circle,fill=red,scale=.3](c) at (a-8-2) {};
	\node[circle,fill=red,scale=.3](c) at (a-1-3) {};
	\node[circle,fill=red,scale=.3](c) at (a-1-4) {};
	\node[circle,fill=red,scale=.3](c) at (a-1-5) {};
	\node[circle,fill=red,scale=.3](c) at (a-1-6) {};
	\node[circle,fill=red,scale=.3](c) at (a-1-7) {};
	\node[circle,fill=red,scale=.3](c) at (a-9-3) {};
	\node[circle,fill=red,scale=.3](c) at (a-9-4) {};
	\node[circle,fill=red,scale=.3](c) at (a-9-5) {};
	\node[circle,fill=red,scale=.3](c) at (a-9-6) {};
	\node[circle,fill=red,scale=.3](c) at (a-9-7) {};
	\node[circle,fill=red,scale=.3](c) at (a-2-8) {};
	\node[circle,fill=red,scale=.3](c) at (a-3-9) {};
	\node[circle,fill=red,scale=.3](c) at (a-4-9) {};
	\node[circle,fill=red,scale=.3](c) at (a-5-9) {};
	\node[circle,fill=red,scale=.3](c) at (a-6-9) {};
	\node[circle,fill=red,scale=.3](c) at (a-7-9) {};
	
	\node[circle,fill=red,scale=.3](c) at (b-4-1) {};
	\node[circle,fill=red,scale=.3](c) at (b-5-1) {};
	\node[circle,fill=red,scale=.3](c) at (b-6-1) {};
	\node[circle,fill=red,scale=.3](c) at (b-7-1) {};
	\node[circle,fill=red,scale=.3](c) at (b-3-1) {};
	\node[circle,fill=red,scale=.3](c) at (b-8-2) {};
	\node[circle,fill=red,scale=.3](c) at (b-1-3) {};
	\node[circle,fill=red,scale=.3](c) at (b-1-4) {};
	\node[circle,fill=red,scale=.3](c) at (b-1-5) {};
	\node[circle,fill=red,scale=.3](c) at (b-1-6) {};
	\node[circle,fill=red,scale=.3](c) at (b-1-7) {};
	\node[circle,fill=red,scale=.3](c) at (b-9-3) {};
	\node[circle,fill=red,scale=.3](c) at (b-9-4) {};
	\node[circle,fill=red,scale=.3](c) at (b-9-5) {};
	\node[circle,fill=red,scale=.3](c) at (b-9-6) {};
	\node[circle,fill=red,scale=.3](c) at (b-9-7) {};
	\node[circle,fill=red,scale=.3](c) at (b-2-8) {};
	\node[circle,fill=red,scale=.3](c) at (b-3-9) {};
	\node[circle,fill=red,scale=.3](c) at (b-4-9) {};
	\node[circle,fill=red,scale=.3](c) at (b-5-9) {};
	\node[circle,fill=red,scale=.3](c) at (b-6-9) {};
	\node[circle,fill=red,scale=.3](c) at (b-7-9) {};

	\end{tikzpicture}
	\label{}
	\caption{}
\end{figure}


We now discretize the diffusion problem (\ref{problem.dFE.om}) using modified $P_1$ finite elements on a compact domain $\om\cu \f R^2$ with positive measure, $\mu(\partial\om) = 0$ and grids described by the function $\phi$.			
\begin{equation}\tag{\ref{problem.dFE.om}}
\left\{\begin{array}{ll}
\displaystyle-\sum_{i,j=1}^2\frac{\partial}{\partial x_j}\Bigl(a_{i,j}\frac{\partial u}{\partial x_i}\Bigr)+\sum_{i=1}^2b_i\frac{\partial u}{\partial x_i}+cu=f, & \mbox{in }\om^\circ,\\
u=0, & \mbox{on }\partial\om,
\end{array}\right.
\end{equation}
where  
$a_{i,j}$, $b_i$, $c$ and $f$ are given complex-valued $L^1$ functions defined on $\om$.			

The basis function we consider on $\om$ are produced from the classical $P_1$ elements by composition with the map $\vf$. In fact, if $p\in \Xi_n(D)\cu \Xi_n\cap D$ and $p' = \phi(p)$ we can define the basis function associated to $p'$ as
\[
\xi_{p',n} := \psi_{p,n} \circ \vf.
\]
Note that the support of $\xi_{p',n}$ is $T_{p'}:= \phi(T_p)$ and $T_p\cu D\iff T_{p'}\cu \om$. In the classical $P_1$ setting, we consider a basis function for each point in $\Xi(D)$, so here we will produce a function $\xi_{p',n}$ only for the points $p'\in \phi(\Xi(D))$, and we call the set of such points
\[
\Xi(\om):= \phi(\Xi(D)) = \phi\left( \set{p\in \Xi_n| T_p\cu D} \right) = \set{ p'\in \phi(\Xi_n \cap D) |  T_{p'}\cu \om  }.
\]

The weak form of the problem (\ref{weak.problem.dFE.triangle}) leads us to a  linear system similar to the ones already considered. In fact,  if we substitute $u=\sum_{p'\in \Xi_n(\om)} u_{p'} \xi_{p',n}$ and $w=\xi_{q',n}$ into problem (\ref{weak.problem.dFE.triangle}), then  we obtain the relation 
\begin{equation}
\label{FE_general_domain}
\sum_{p'\in \Xi_n(\om)} s^\om_{q',p'}u_{p'}=f_{q'}, \qquad 
s^\om_{q',p'}=\int_{\om^\circ} (\nabla \xi_{p',n})^TA\nabla \xi_{q',n} + (\nabla \xi_{p',n})^T\bm b \xi_{q',n} + c\xi_{p',n}\xi_{q',n}, \qquad f_{q'}  = \int_{\om^\circ} f\xi_{q',n}
\end{equation}
for every $q'$ in $\Xi_n(\om)$. 
Sorting the relations in lexicographical order with respect to the appearance of $\vf(q')$ in the grid $\Xi_n$, we obtain a linear system 
$S_n^\om \bm u_n = \bm f_n$ of size $|\Xi_n(\om)|=|\Xi_n(D)|$.

The analysis of this particular instance descends from the fact that
we can find  opportune coefficients for the problem  (\ref{weak.problem.dFE.triangle})  on the domain $D$ so that
the linear system arising from the $P_1$ method applied to the regular grid  $\Xi_n(D)$ coincides with $S_n^\om$.
Consider in fact the problem
\begin{equation}\label{problem.dFE.mapped}
\left\{\begin{array}{ll}
\displaystyle-\sum_{i,j=1}^2\frac{\partial}{\partial x_j}\Bigl(\wt a_{i,j}\frac{\partial u}{\partial x_i}\Bigr)
+\sum_{i=1}^2\wt b_i\frac{\partial u}{\partial x_i}+\wt cu
=f, & \mbox{in }D^\circ,\\
u=0, & \mbox{on }\partial D,
\end{array}\right. 
\end{equation}
and its weak form
\begin{equation}\label{weak.problem.dFE.mapped}
\displaystyle
\int_{D^\circ} (\nabla u)^T\wt A\nabla w  
+ (\nabla u)^T\wt{\bm b}w 
+ u\wt c w
= \int_{D^\circ} fw,\qquad \forall w\in H_0^1(D)
.
\end{equation}
where  
\[
\wt A(\bm x) :=  J_\phi^{-1}(\bm x)A(\phi(\bm x))J_\phi^{-T}(\bm x)  |\det J_\phi(\bm x)|,\]\[ 
\wt {\bm b}(\bm x) := J_\phi^{-1}(\bm x)\bm b(\phi(\bm x)) |\det J_\phi(\bm x)|,
\qquad 
\wt c(\bm x) := c(\phi(\bm x))  |\det J_\phi(\bm x)|
\]
are $L^1$ functions on $D$. Applying the $P_1$ method to this problem we obtain the relations 
\begin{equation}
\label{FE_general_domain_mapped}
\sum_{p\in \Xi_n(D)} s^D_{q,p}\wt u_p=\wt f_q, \qquad 
s^D_{q,p}=\int_{D^\circ} (\nabla \psi_{p,n})^T\wt A\nabla \psi_{q,n}
+ (\nabla \psi_{p,n})^T\wt{\bm b} \psi_{q,n}
+ \psi_{p,n}\wt c  \psi_{q,n}
, \qquad \wt f_q  = \int_{D^\circ} f\psi_{q,n}
\end{equation}
for every $q\in \Xi_n(D)$, that give rise to the system $S_n^D \wt{\bm u} =\wt{\bm f}_n$ of size $|\Xi_n(D)|$. 
Notice that if $p',q' \in \Xi_n(\om)$ such that $p'=\phi(p)$ and $q'=\phi(q)$, then
\begin{align*}
\int_{\om^\circ} (\nabla  \xi_{p',n})^TA\nabla \xi_{{q'},n}&=\int_{D^\circ} (\nabla_{\bm x} \psi_{p}(\bm x))^T J_{\phi}(\bm x)^{-1} A(\phi(\bm x)) J_{\phi}(\bm x)^{-T}  \nabla_{\bm x} \psi_{q}(\bm x) |\det J_{\phi}(\bm x)| {{\rm d}}\bm x\\
&= \int_{D^\circ} (\nabla_{\bm x} \psi_{p}(\bm x))^T \wt A(\bm x)  \nabla_{\bm x} \psi_{q}(\bm x)  {{\rm d}}\bm x,\\
\int_{\om^\circ} (\nabla  \xi_{p',n})^T\bm b \xi_{{q'},n}&=\int_{D^\circ} (\nabla_{\bm x} \psi_{p}(\bm x))^T J_{\phi}(\bm x)^{-1} \bm b(\phi(\bm x)) \psi_{q}(\bm x) |\det J_{\phi}(\bm x)| {{\rm d}}\bm x\\
&= \int_{D^\circ} (\nabla_{\bm x} \psi_{p}(\bm x))^T \wt {\bm b}(\bm x) \psi_{q}(\bm x)  {{\rm d}}\bm x,\\
\int_{\om^\circ}  \xi_{p',n}c \xi_{{q'},n}&=\int_{D^\circ}  \psi_{p}(\bm x) c(\phi(\bm x)) \psi_{q}(\bm x) |\det J_{\phi}(\bm x)| {{\rm d}}\bm x\\
&= \int_{D^\circ} \psi_{p}(\bm x) \wt c(\bm x) \psi_{q}(\bm x)   {{\rm d}}\bm x,
\end{align*}
so 
comparing \autoref{FE_general_domain_mapped} and \autoref{FE_general_domain} we conclude that $s_{p,q}^D = s_{p',q'}^\om$ for every $p,q\in \Xi(D)$ and therefore $S_n^\om = S_n^D$. 
The symbols of the sequence can be easily computed from \autoref{FE_P1_subdomain:symbol}.

\section{Future Work}\label{FW}

We have introduced and thoroughly analysed the space of Reduced GLT, showing how they can prove useful in applications. 
Reduced GLT sequences have already been applied on discretizations of fractional PDEs on generic domains, and they can also be applied straightforwardly on graph structures, as showed in \cite{Bianchi}. 

Following the lead of the classical GLT sequences, the next step is to generalize the space of Reduced GLT to the case of block sequences, studied in \cite{GLT-bookIII,GLT-bookIV}, in order to tackle also systems of PDEs and high-order approximations on generic domains.

\section{Acknowledgment}

This preprint has not undergone peer review or any post-submission improvement or corrections. The Version of Record of this article is published in BIT Numerical Mathematics, and is available online at https://doi.org/10.1007/s10543-021-00896-7.

\end{document}